\documentclass[a4paper,12pt]{article}

\usepackage[utf8]{inputenc}
\usepackage[T1]{fontenc}
\usepackage[a4paper, margin=2.5cm]{geometry}

\usepackage{amsmath, amsthm, amssymb, enumerate}
\usepackage{graphicx}
\usepackage{enumerate}
\usepackage{authblk}
\usepackage[textsize=small, textwidth=2cm, color=yellow]{todonotes}
\usepackage{caption}
\usepackage[labelformat=simple]{subcaption}

\usepackage{inconsolata}
\usepackage{libertine}
\usepackage[absolute]{textpos}

\usepackage{comment}

\usepackage{thmtools}
\usepackage{thm-restate}

\usepackage{algpseudocode}

\usepackage[colorlinks=true, citecolor=red]{hyperref}

\usepackage{xfrac}
\usepackage{tikz}
\usetikzlibrary{shapes,snakes}
\usetikzlibrary{arrows,positioning} 
\tikzset{
    punkt/.style={
           circle,
           draw=black, very thick,
           text centered,
           inner sep = 2.5pt},
    dot/.style={
           circle,
           fill=black,
           inner sep=3pt
           },
    pil/.style={
           -,
           thick,
           },
     dir/.style={
           <-,
           thick,
           shorten <=2pt,
           shorten >=2pt,
           }
}

\declaretheorem[name=Theorem]{theorem} 
\declaretheorem[name=Lemma, sibling=theorem]{lemma}
\declaretheorem[name=Proposition, sibling=theorem]{proposition}
\declaretheorem[name=Definition, sibling=theorem]{definition}
\declaretheorem[name=Corollary, sibling=theorem]{corollary}
\declaretheorem[name=Conjecture, sibling=theorem]{conjecture}
\declaretheorem[name=Claim, sibling=theorem]{claim}
\declaretheorem[name=Observation, sibling=theorem]{observation}
\declaretheorem[name=Question, style = remark, sibling=theorem]{question}
\declaretheorem[name=Case, style = remark, numberwithin=theorem]{case}
\declaretheorem[name=Subcase, style = remark, numberwithin=case]{subcase}

\declaretheorem[name=Subsubcase, style = remark, numberwithin=subcase]{subsubcase}
\declaretheorem[name=Subsubsubcase, style = remark, numberwithin=subsubcase]{subsubsubcase}

\newcommand{\Uu}{{\mathcal{U}}}
\newcommand{\Vv}{{\mathcal{V}}}
\newcommand{\Ww}{{\mathcal{W}}}
\newcommand{\Zz}{{\mathcal{Z}}}
\newcommand{\Xx}{{\mathcal{X}}}
\newcommand{\Yy}{{\mathcal{Y}}}

\renewcommand{\leq}{\leqslant}
\renewcommand{\geq}{\geqslant}

\def\ie{{\em i.e.}} 

\def\cqedsymbol{\ifmmode$\lrcorner$\else{\unskip\nobreak\hfil
\penalty50\hskip1em\null\nobreak\hfil$\lrcorner$
\parfillskip=0pt\finalhyphendemerits=0\endgraf}\fi}

\title{Vizing's conjecture holds \thanks{The author is supported by National Science Center of Poland grant 2019/34/E/ST6/00443.}}
\author[1,2]{Jonathan Narboni}

\affil[1]{CNRS, LaBRI, Université de Bordeaux, France.}
\affil[2]{Theoretical Computer Science Department, Faculty of Mathematics and Computer Science,
Jagiellonian University, Kraków, Poland}
\date{}

\begin{document}
\maketitle


\begin{abstract}
    In 1964 Vizing proved that starting from any $k$-edge-coloring of a graph $G$ one can reach, using only Kempe swaps, a $(\Delta+1)$-edge-coloring of $G$ where $\Delta$ is the maximum degree of $G$. One year later he conjectured that one can also reach a $\Delta$-edge-coloring of $G$ if there exists one. Bonamy \textit{et. al} proved that the conjecture is true for the case of triangle-free graphs. In this paper we prove the conjecture for all graphs.
\end{abstract}

\section{Introduction}
In 1964 Vizing proved that the chromatic index of a graph $G$ (\ie the minimum number of colors needed to properly colors the edges of $G$), denoted by $\chi'(G)$, is at most $\Delta(G) +1 $ colors, where $\Delta(G)$ is the maximum degree of $G$.

\begin{theorem}\label{tmh:vizing_64}
Any simple graph $G$ satisfy $\chi'(G)\leq \Delta(G)+1$.
\end{theorem}

The proof heavily relies on the use of \textit{Kempe changes}. Kempe changes were introduced by Kempe in his unsuccessful attempt to prove the $4$-color theorem, but it turns out that this concept became one of the most fruitful tool in graph coloring. Throughout this paper, we only consider proper edge-colorings, and so we will only write colorings to denote proper edge-colorings. Given a graph $G$ and a coloring $\beta$, a Kempe chains $C$ is a maximal bichromatic component (Kempe chains were invented in the context of vertex-coloring, but the principle remains the same for edge-coloring). Applying a \textit{Kempe swap} (or Kempe change) on $C$ consists in switching the two colors in this component. Since $C$ is maximal, the coloring obtained after the swap is guaranteed to be a proper coloring, and if $C$ is not the unique maximal bichromatic component containing these two colors, the coloring obtained after the swap is a coloring different from $\beta$, as the partition of the edges is different.

The Kempe swaps induce an equivalence relation on the set of colorings of a graph $G$; two colorings $\beta$ and $\beta'$ are equivalent if one can find a sequence of Kempe swaps to transform $\beta$ into $\beta'$. In 1964, Vizing actually proved a stronger statement, he proved that any $k$-coloring of a graph $G$ (with $k>\Delta(G)$) is equivalent to a $(\Delta(G)+1)$-coloring of $G$.

\begin{theorem}
Let $G$ be a graph and $\beta$ a $k$-coloring of $G$ (with $k>\Delta(G)$). Then there exists a $(\Delta(G)+1)$-coloring $\beta'$ equivalent to $\beta$.
\end{theorem}

Note that some graphs only need $\Delta$ colors to be properly colored. One year later, Vizing proved that this result is generalizable to multigraphs, and ask the following question:

\begin{question}\label{que:vizing_65}
For any (multi)graph $G$ and for any $k$-coloring $\beta$ of $G$, is there always an optimal coloring equivalent to $\beta$ ?
\end{question}
Note that both in Theorem~\ref{tmh:vizing_64} and in Question~\ref{que:vizing_65} we do not have the choice in the target coloring. If we can choose a specific target optimal coloring, then the question can be reformulated as a reconfiguration question.

\begin{question}\label{que:vizing_plus_choice}
For any (multi)graph $G$ and for any $k$-coloring $\beta$, is any optimal coloring always equivalent to $\beta$ ?
\end{question}

If true, Question~\ref{que:vizing_plus_choice} would imply the following conjecture, as it suffices to take the optimal target coloring as an intermediate between the two non-optimal colorings.

\begin{conjecture}\label{conj:reconf}
Any two non-optimal colorings are equivalent.
\end{conjecture}

Mohar proved the weaker case where we have one additional color \cite{mohar2006kempe}.

\begin{theorem}[\cite{mohar2006kempe}]
All $(\chi'(G)+2)$-colorings are equivalent.
\end{theorem}

When considering the stronger case, McDonald \& al. proved that the conjecture is true for graphs with maximum degree $3$ \cite{mcdonald2012kempe}, Asratian and Casselgren proved that it is true for graphs with maximum degree $4$ \cite{casselgren}, and Bonamy \& al. proved that the conjecture is true for triangle-free graphs. In this paper, we prove that the conjecture is true for all graphs.

\begin{theorem}\label{thm:Kempe}
Let $G$ be a graph, all its $(\chi'(G)+1)$-edge colorings are Kempe-equivalent. 
\end{theorem}

Theorem~\ref{thm:Kempe} is a direct consequence of the following Lemma which is the main result of this paper.

\begin{lemma}\label{lem:main}
Let $G$ be a graph, any $(\chi'(G)+1)$-coloring of $G$ is equivalent to any $\chi'(G)$-coloring of $G$.
\end{lemma}

\section{General setting of the proof}

The proof inherits the technical setup of \cite{bonamy2021vizing}, in this section, we introduce this setting, and give the general outline of the proof of the main result.

\subsection{Reduction to $\chi'(G)$-regular graphs}
The general setting of the proof follows that of \cite{bonamy2021vizing} which itself follows that of \cite{mcdonald2012kempe} and of \cite{casselgren}. We first show that we can reduce the problem to the class of regular graphs. Indeed, given a graph $G$ and a coloring $\beta$, we can build a graph $G'$ s.t. :
\begin{itemize}
    \item $G'$ is $\chi'(G)$-regular,
    \item the coloring $\beta$ can be completed into a coloring $\beta'$ of $G'$, and
    \item if a coloring $\gamma'$ is equivalent to $\beta'$ in $G'$, then the restriction of $\gamma'$ to $G$ is equivalent to $\beta$.
\end{itemize}

To build $G'$ we step-by-step build a sequence $G_1,\cdots,G_t = G'$ where at each step, $\delta(G_{i+1}) = \delta(G_i) +1$ (where $\delta(G_i)$ is the minimum degree of $G_i$). For any $G_i$, the graph $G_{i+1}$ is build as follows:
\begin{itemize}
    \item take two copies of $G_i$, and
    \item add a matching between a vertex and its copy if this vertex has minimum degree in $G_i$.
\end{itemize}

It is clear that $G'$ is $\chi'(G)$-regular, and to extend a coloring of $G_i$ to a coloring of $G_{i+1}$, it suffices to copy the coloring for each copy of $G_i$, and since the edges of the matching connect two vertices of minimum degree, there will always be an available colors for those edges. Moreover, a Kempe swap in $G_i$ has a natural generalization in $G_{i+1}$: if a swap in $G_i$ corresponds to more than one swap in $G_{i+1}$, it suffices to apply the swap on all the corresponding components in $G_{i+1}$. From now on, we will only consider $\chi'$-regular graphs.

Note that colorings in regular graphs are easier to handle as the two following properties are verified:
\begin{itemize}
    \item for any $(\Delta(G))$-coloring of a $\chi'(G)$-regular graph $G$, every vertex $v$ is incident to exactly one edge of each color, and each color class is a perfect matching, and
    \item for any $(\Delta(G) +1)$-coloring $\alpha$ of a $\chi'(G)$-regular graph $G$, every vertex $v$ is incident to all but one color, we call this color the \textit{missing color} at $v$, and denote it by $m_{\alpha}(v)$ (we often drop the $\alpha$ when the coloring is clear from the context).
\end{itemize}

From now on, in the rest of the paper, we only consider $\chi'$-regular graphs.

\subsection{The good the bad and the ugly}

The general approach to Theorem~\ref{thm:Kempe} is an induction on the chromatic index. Given a graph $G$, a $\Delta(G)$-coloring $\alpha$ and a $(\Delta(G)+1)$-coloring $\beta$, our goal is to find a sequence of Kempe swaps to transform $\beta$ into $\alpha$. To do so, we consider a color class in $\alpha$, say the edges colored $1$. These edges induce a perfect matching $M$ in $G$, thus, if we can find a coloring $\beta'$ equivalent to $\beta$ s.t. for any edge $e$, $\beta'(e) = 1 \Leftrightarrow \alpha(e) = 1$, then we can proceed by induction on $G' = G \setminus M$, noting that $\chi'(G') = \chi'(G) -1$, and that the restrictions of $\alpha$ and $\beta'$ to $G'$ only use $\Delta(G) -1$, and $\Delta(G)$ colors respectively.

So, given a $(\Delta(G)+1)$-coloring $\beta$ of $G$, we can partition the edges of $G$ into three sets, an edge $e$ is called:
\begin{itemize}
    \item \textit{good}, if $e\in M$ and $\beta(e) =1$,
    \item \textit{bad}, if $e\in M$ and $\beta(e)\neq 1$,and 
    \item \textit{ugly}, if $e\not \in M$ and $\beta(e) = 1$.
\end{itemize}

\noindent A vertex missing the color $1$ is called a \textit{free vertex}. Toward contradiction, we assume that $\beta$ is not equivalent to $\alpha$, and we consider a $(\Delta(G)+1)$-coloring $\beta'$ equivalent to $\beta$ which minimizes the number of ugly edges among the colorings equivalent to $\beta$ that minimize the number of bad edges, we call \textit{minimal} such a coloring. Thus, if we can find a coloring $\beta''$ equivalent to $\beta'$ where the number of bad is strictly lower than in $\beta'$, or with the same number of bad edges, and strictly fewer ugly edges, we get a contradiction.

\subsection{Fan-like tools}
In his proof of 64, Vizing introduce a technical tool to apply Kempe swaps on an edge-coloring in very controlled way: \textit{Vizing's fans}. To define them, we first need to define an auxiliary digraph. Given a graph $G$, a $(\Delta(G)+1)$-coloring $\beta$ of $G$ and a vertex $v$, the directed graph $D_v$ is defined as follows:
\begin{itemize}
    \item the vertex set of $D_v$ is the set of edges incident with $v$, and
    \item we put an arc between two vertices $vv_1$ and $vv_2$ of $D_v$, if the edge $vv_2$ is colored with the missing color at $v_1$.
\end{itemize}

The \textit{fan around $v$ starting at the edge $e$}, denoted by $X_v(e)$, is the maximal sequence of vertices of $D_v$ reachable from the edge $e$. It is sometime more convenient to speak about the color of the starting edge of a fan, if $c$ is a color, $X_v(c)$ denotes the fan around $v$ starting at the edge colored $c$ incident with $v$. Note that since the graph $G$ is $\chi'(G)$-regular, each vertex misses exactly one color, and thus, in the digraph $D_v$, each vertex has outdegree at most $1$. Hence a fan $\Xx$ is well-defined and we only have three possibilities for the fan $\Xx$:

\begin{itemize}
    \item $\Xx$ is a path,
    \item $\Xx$ is a cycle, or
    \item $\Xx$ is a comet (\textit{i.e.} a path with an additional arc between the sink and an internal vertex of the path) .
\end{itemize}

If $X = (vv_1,\cdots,vv_k)$ is a fan, $v$ is called the central vertex of the fan, and $vv_1$ and $vv_k$ are respectively called the first and the last edge of the fan (similarly, $v_1$ and $v_k$ are the first and last vertex of $X$ respectively). 

 Given a $(\Delta(G)+1)$-coloring $\beta$ of $G$, and fan $\Xx = (vv_1,\cdots,vv_k)$ which is a cycle around a vertex $v$, where each vertex $v_i$ misses the color $i$ (and so each edge $vv_i$ is colored $(i-1)$), we can define the coloring $\beta' = X^{-1}(\beta)$ as follows:
\begin{itemize}
    \item for any edge $vv_i$ not in $\Xx$, $\beta'(vv_i) = \beta(vv_i)$, and
    \item for any edge $vv_i$ in $\Xx$, $\beta'(vv_i) = i$ and $m(v_i) = i-1$
\end{itemize}

The coloring $\Xx^{-1}(\beta)$ is called the \textit{invert} of $\Xx$, and we say that $X$ is \textit{invertible} if $\Xx$ and $\Xx^{-1}(\beta)$ are equivalent. In this paper, we prove that in any coloring, any cycle is invertible.

\begin{lemma}\label{lem:all_cycles_invertible}
In any $(\chi'(G)+1)$-coloring of a $\chi'(G)$-regular graph $G$, any cycle is invertible.
\end{lemma}

 We prove Lemma~\ref{lem:all_cycles_invertible} in Section~\ref{sec:general_outline}, and prove here Theorem~\ref{thm:Kempe}. The proof of Theorem~\ref{thm:Kempe} is derived from the proof of Theorem~1.6 in $\cite{bonamy2021vizing}$. We first need the following results from \cite{bonamy2021vizing} and \cite{casselgren} which we restate here (in a slightly different way).
 
 \begin{observation}[\cite{bonamy2021vizing}]\label{obs:every_bad_adj_ugly}
In a minimal coloring, every bad edge is adjacent to an ugly edge.
 \end{observation}
 
 \begin{lemma}[\cite{bonamy2021vizing}]\label{lem:ugly_induce_cycle}
 In a minimal coloring, any ugly edge $uv$ is such that the fans $X_v(uv)$ and $X_u(uv)$ are cycles.
 \end{lemma}
 
 \begin{lemma}[\cite{casselgren}]\label{lem:ugly_adj_free}
 In a minimal coloring both ends of an ugly edge are adjacent to a free vertex.
 \end{lemma}
 
 We first show that in a minimal coloring, there always exists a bad edge adjacent to an ugly edge and incident with a free vertex.
 
 \begin{lemma}\label{lem:bad_free_ugly}
In a minimal coloring, there exists a bad edge adjacent to an ugly edge and incident with a free vertex.
\end{lemma}

 \begin{proof}
 Let $\beta$ be a minimal coloring, if there is no bad edge in $\beta$, then all the edges of $M$ are colored $1$ in $\beta$ as desired. So there exists a bad edge $e$ in $\beta$, and by Observation~\ref{obs:every_bad_adj_ugly}, $e$ is adjacent to an ugly edge $e'$. By Lemma~\ref{lem:ugly_adj_free}, there exists a free vertex $u$ adjacent to an end of $e'$. As $u$ is a free vertex, $u$ is incident with a bad edge, we denote by $v$ the neighbor of $u$ such that the edge $uv$ is bad. If $v$ is a free vertex, then we swap the single edge $uv$ to obtain a coloring with fewer bad edges, so $v$ is not free, and thus $uv$ is adjacent to an ugly edge; this concludes the proof.
 \end{proof}
 
 We are now ready to prove Theorem~\ref{thm:Kempe}, but we first need some terminology and notations. Given a coloring $\alpha$, for any pair of colors $a$,$b$, we denote by $K(a,b)$ the graph induced in $G$ by the edges colored $a$ and $b$. The Kempe chain involving these two colors and containing the element $x\in V(G)\cup E(G)$ is denoted by $K_x^{\alpha}(a,b)$ (we often drop the $\alpha$ when the coloring is clear form the context). It is important to note that if $a$, $b$, $c$ and $d$ are $4$ different colors, then swapping a component of $K(a,b)$ before or after swapping a component of $K(c,d)$ does not change the coloring obtained after the two swaps.
 
Note also that in an edge-coloring, any Kempe chain $K(a,b)$ is a connected bipartite subgraph of maximum degree $2$, hence it is either a path, or an even cycle. To distinguish the notions of fans that can be paths or cycles, when a Kempe component $C$ of $K(a,b)$ is a path (respectively an even cycle) we say that $C$ si a $(a,b)$-bichromatic path (respectively a $(a,b)$-bichromatic cycle). If $u$ is a vertex missing the color $a$, then $K_u(a,b)$ is a $(a,b)$-bichroamtic path whose ends are $u$ and another vertex missing either $a$ or $b$.
 
\begin{proof}[Proof of Theorem~\ref{thm:Kempe}]
 Let $\beta$ be a minimal coloring. By Lemma~\ref{lem:bad_free_ugly}, there exists a bad edge $uv$ such that $u$ is free and $v$ is incident with an ugly edge $vw$. By Lemma~\ref{lem:ugly_induce_cycle}, the fans $X_v(vw)$ and $X_w(vw)$ are both cycles. The vertex $v$ does not belong to $X_v(vw)$, otherwise, by Lemma~\ref{lem:all_cycles_invertible} we invert $X_v(vw)$ and obtain a coloring with strictly fewer bad edges. Hence, the vertex $w$ is missing a color $c'$ different from $c = \beta(uv)$ (otherwise, $X_v(vw)$ is a cycle of size $2$ containing $u$). We now consider the component $C = K_w(c,c')$, note that since $w$ is missing the color $c'$, this component is a $(c,c')$-bichromatic path. If the component $C$ does not contain the vertex $v$, then we swap it to obtain a coloring where $w$ is missing the color of the edge $uv$ and we are done. Thus, $C$ contains $v$ and we have to distinguish whether $v$ is between $u$ and $w$ in $C$ or $u$ is between $w$ and $v$.
 
 \begin{case}[$u$ is between $w$ and $v$ in $C$]
 ~\newline
 In this case, by Lemma~\ref{lem:all_cycles_invertible} we can invert $X_v(vw)$ to obtain a coloring where the component $K_w(c,c')$ is now a $(c,c')$-bichromatic cycle that we swap. In the coloring obtained after the swap, $X_v(uv)$ is a cycle, and so by Lemma~\ref{lem:all_cycles_invertible} we can invert it to obtain a coloring with strictly fewer bad edges; a contradiction.
 \end{case}

  \begin{case}[$v$ is between $w$ and $u$ in $C$]
 ~\newline
 In this case, we consider the cycle $X_w(vw)$. If it does not contain the vertex $u$, we invert it by Lemma~\ref{lem:all_cycles_invertible} and obtain a coloring where $u$ and $v$ are free, so it suffices to swap the edge $uv$ to obtain a coloring with strictly fewer bad edges. Hence the vertex $u$ belongs to $X_w(vw)$. After inverting this cycle, we obtain a minimal coloring where $uv$ is still bad, $v$ is free, and $uw$ is ugly (the edge $vw$ is not ugly anymore in this coloring). By Lemma~\ref{lem:ugly_induce_cycle}, the fan $X_u(uw)$ is a cycle. The situation is now similar to the previous case: we invert the cycle $X_u(uw)$ to obtain a coloring where the component $K_w(c,c')$ is a $(c,c')$-bichromatic cycle. After swapping this cycle we obtain a minimal coloring where $X_u(uv)$ is a cycle. After inverting this cycle, we obtain a coloring with one fewer bad edge; a contradiction.
 \end{case}
\end{proof}

\subsection{General outline and notations}\label{sec:general_outline}
The proof is an induction on the size of the cycles. Towards contradiction, assume that there exist non-invertible cycles. A \emph{minimum} cycle $\Vv$ is a non-invertible cycle of minimum size (\textit{i.e.} in any coloring, any smaller cycle is invertible).

A cycle of size $2$ is clearly invertible as it only consists of a single Kempe chain composed of exactly two edges: to invert the cycle, it suffices to apply a Kempe swap on this component; so the size of a minimum cycle is at least $3$.

We now need some more notations. For any fan $\Vv = (vv_1,\cdots, vv_k)$, $V(\Vv)$ denotes the set of vertices $\{v_1,\cdots v_k\}$, and $E(\Vv)$ denotes the set of edges $\{vv_1,\cdots vv_k\}$. We denote by $\beta(\Vv)$ the set of colors involved in $\Vv$ (\textit{i.e.} $\beta(\Vv) = \beta(E(\Vv))\cup m(V(\Vv)) \cup m(v)$); if $\Vv$ involves the color $c$, $M(X,c)$ denotes the vertex of $V(\Vv)$ missing the color $c$. There is a natural order induced by a fan on its vertices (respectively on its edges), and if $i<j$ we say that the vertex $v_i$ (respectively the edge $vv_i$) is before the vertex $v_j$ (respectively the edge $vv_j$). For two vertices $v_i$ and $v_j$ of $\Vv$ we define the \textit{subfan} $\Vv_{[v_1,v_j]}$ as the subsequence $(vv_i,vv_{i+1},\cdots vv_{j})$. We often write $\Vv_{\geq v_i}$, $\Vv_{>v_i}$, $\Vv_{\leq v_i}$ and $\Vv_{<v_i}$ to respectively denote the subfans $(v_i,\cdots v_k)$, $(v_{i+1},\cdots v_k)$,$(v_1,\cdots, v_i)$, and $(v_1,\cdots v_{i-1})$.

If the fan $\Vv$ is a cycle in a coloring $\beta$ means applying a sequence of Kempe swaps to obtain the coloring $X^{-1}(\beta)$. If $\Vv$ is a fan which is a path, inverting $\Vv$ means applying a sequence of single-edge Kempe swaps on the edges of $\Vv$ such that the ends of the first edge of $\Vv$ are missing the same color $\beta(vv_1)$. Note that we often only partially invert paths, \textit{i.e.} we apply a sequence of single-edge Kempe swaps on the edges of the paths until we reach a coloring with a specific missing color at the central vertex. 

A cycle $\Vv = (vv_1,\cdots,vv_k)$ is called \textit{saturated} if for any $i$, $v_i\in K_{v}(m(v),m(v_i))$. Lemma~2.3 of \cite{bonamy2021vizing}, which we restate here, guarantees that if a cycle is not invertible, then it is saturated.

\begin{lemma}[\cite{bonamy2021vizing}]\label{lem:saturated_cycle}
Let $\Vv$ be a cycle, if $\Vv$ is not saturated, then $\Vv$ is invertible.
\end{lemma}

This directly implies the same result for any minimum cycle.

\begin{lemma}\label{lem:minimum_cycle_saturated}
    Any minimum cycle is saturated.
\end{lemma}

Let $X \subseteq E(G)\cup V(G)$, $\beta$ a coloring and $\beta'$ a coloring obtained from $\beta$ by swapping a component $C$. The component is called $X$-stable if :

\begin{itemize}
    \item for any $v\in X$, $m^{\beta}(v) = m^{\beta'}(v)$, and 
    \item for any $e\in X$, $\beta(e) = \beta'(e)$.
\end{itemize}
In this case, the coloring $\beta'$ is called $X$-\textit{identical} to $\beta$.

If $S = (C_1,\cdots, C_k)$ is a sequence of swaps to transform a coloring $\beta$ into a coloring $\beta'$ where each $C_j$ is a Kempe component. The sequence $S^{-1}$ is defined a the sequence of swaps $(C_k,\cdots, C_1)$. Such a sequence is called $X$-stable is each $C_j$ is $X$-stable.

\begin{observation}\label{obs:S_X-stable_so_S_invert_too}
Let $X \subseteq V(G)\cup E(G)$, and $S$ a sequence of swaps that is $X$-stable. Then the sequence $S^{-1}$ is also $X$-stable.
\end{observation}

If a sequence $S$ is $X$-stable, then the coloring obtained after apply $S$ to $\beta$ is called $X$-\textit{equivalent} to $\beta$.
Note that the notion of $X$-equivalence is stronger than the notion of $X$-identity. Since two colorings $\beta$ and $\beta'$ may be $X$-identical but not $X$-equivalent if there exists a coloring $\beta''$ in the sequence between $\beta$ and $\beta'$ that is not $X$-identical to $\beta$. We first have the following obsevration that we will often use in this paper.

\begin{observation}\label{obs:not_touching_subsequence_stable}
    Let $\Xx$ be a subfan in a coloring $\beta_0$, $v$ be a vertex which is not in $V(\Xx)$, and $S = (C_1,\cdots,C_k)$ be a sequence of trivial swaps of edges incident with $v$, $(\beta_1,\cdots,\beta_k)$ be the colorings obtained after each swap of $S$. If for any $i\in\{0,\cdots,k\}$, $m^{\beta_i}(v)\not\in\beta_0(\Xx)$, then the sequence $S$ is $(\Xx)$-stable.
\end{observation}
\begin{proof}
    Otherwise, assume that $\mathcal{S}$ is not $\Xx$-stable. Since the vertex $v$ is not in $V(\Xx)$, then no edge of $\Xx$ has been changed during the sequence of swap. Thus the missing color of a vertex of $\Xx$ has been changed during the sequence of swaps, we denote by $x$ the first such vertex. Let $s_i$ be the swap that change the color of the edge $vx$, it means that in the coloring $\beta_{i-1}$ the vertices $v$ and $x$ are missing the same color, so $m^{\beta_{i-1}} \in \beta_0(\Xx)$; a contradiction.
\end{proof}

The following observation gives a relation between $X$-equivalence and $(G\setminus)$-identity between colorings.

\begin{observation}\label{obs:X-equivalent_plus_identical}
	Let $\beta$ be a coloring, $X \subseteq V(G)\cup E(G)$, $\beta_1$ a coloring $X$-equivalent to $\beta$, and $\beta_2$ a coloring $(G \setminus X)$-identical to $\beta_1$. Then, there exists a coloring $\beta_3$ equivalent to $\beta_2$ that is $X$-identical to $\beta_2$ and $(G\setminus X)$-identical to $\beta$. 
\end{observation}
\begin{proof}
Let $S$ be the sequence of swaps that transforms $\beta$ into $\beta_1$. Since $\beta_1$ is $X$-equivalent to $\beta$, the sequence $S$ is $X$-stable and thus $E(S)\cap E(X) = V(S) \cap V(X) = \emptyset$. Since $\beta_2$ is $(G\setminus X)$-identical to $\beta_1$, it is $S$-identical to $\beta_1$. So applying $S^{-1}$ to $\beta_2$ is well-defined and gives a coloring $\beta_3$ $S$-identical to $\beta$. We first prove that $\beta_3$ is $(G \setminus X)$-identical to $\beta$. The coloring $\beta_1$ is $(G\setminus S)$-identical to the coloring $\beta$ by definition of $S$, and the coloring $\beta_2$ is $(G\setminus X)$-identical to $\beta_1$, so the coloring $\beta_2$ is $(G\setminus (X\cup S))$-identical to $\beta$. Again by definition of $S^{-1}$ the coloring $\beta_3$ is $(G \setminus S)$-identical to $\beta_2$, so it is $(G\setminus (S \cup X))$-identical to $\beta$. Since the coloring $\beta_3$ is also $S$-identical to $\beta$, in total, it is $(G \setminus X)$-identical to $\beta$.

We now prove that $\beta_3$ is $X$-identical to $\beta_2$. Since $E(S)\cap E(X) = V(S) \cap V(X) = \emptyset$, we have that $E(X)\subseteq E(G)\setminus E(S)$ and $V(X)\subseteq V(G)\setminus V(S)$. Moreover, the coloring $\beta_3$ is $(G\setminus S)$-identical to $\beta_2$ by definition of $S$, so the coloring $\beta_3$ is $X$-identical to $\beta_2$ as desired.
\end{proof}

If $\Xx$ is a fan, when two colorings are $(V(\Xx)\cup E(\Xx))$-identical (respectively $(V(\Xx)\cup E(\Xx))$-equivalent), we simply write that the two colorings are $\Xx$-identical (respectively $\Xx$-equivalent). Similarly, if two colorings are $((V(G)\cup E(G))\setminus X)$-identical (respectively $((V(G)\cup E(G))\setminus X)$-equivalent), we simply write that the two colorings are $(G\setminus X)$-identical (respectively $(G \setminus X)$-equivalent). 

Remark that if $\Vv$ is a cycle in a coloring $\beta$, then the coloring $\Vv^{-1}(\beta)$ is $(G \setminus \Vv)$-identical to $\beta$. So from the previous observation we have the following corollary.

\begin{corollary}
Let $\Vv$ be a cycle in a coloring $\beta$. If there exists a coloring $\beta'$ $\Vv$-equivalent to $\beta$ where $\Vv$ is invertible, then $\Vv$ is invertible in $\beta$. 
\end{corollary}
\begin{proof}
Let $\beta'' = \Vv^{-1}(\beta')$. The coloring $\beta'$ is $\Vv$-equivalent to $\beta$ and $\beta''$ is $(G\setminus \Vv)$-identical to $\beta'$. So by Observation~\ref{obs:X-equivalent_plus_identical} there exists a coloring $\beta_3$ that is $\Vv$-identical to $\beta''$ and $(G\setminus \Vv)$-identical to $\beta$. So the coloring $\beta_3$ is $(G\setminus \Vv )$-identical to $\Vv^{-1}(\beta)$.

Moreover, the coloring $\beta''$ is $\Vv$-identical to $\Vv^{-1}(\beta)$, so the coloring $\beta_3$ is also $\Vv$-identical to $\Vv^{-1}(\beta)$. Therefore we have $\beta_3 = \Vv^{-1}(\beta)$ as desired.   
\end{proof}

From the previous corollary, we have the following observation.

\begin{observation}\label{obs:Vv-minimum-stable}
Let $\Vv$ be a minimum cycle in coloring $\beta$, and $\beta'$ a coloring 
$\Vv$-equivalent to $\beta$. Then in the coloring $\beta'$, the sequence $\Vv$ is also a minimuù cycle such that for any $e\in E(\Vv)$, $\beta(e) = \beta'(e)$, and for any $v\in V(\Vv)$, $m^{\beta}(v) = m^{\beta'}(v)$.
\end{observation}

We often simply say that the cycle $\Vv$ is the same minimum cycle in the coloring $\beta'$.



A cycle $\Vv = (vv_1,\cdots,vv_k)$ is called \textit{tight} if for every $i$ $v_i\in K_{v_{i-1}}(m(v_i),m(v_{i-1}))$.
A simple observation is that any minimum cycle $\Vv$ is tight.
\begin{observation}\label{obs:tight}
Let $\Vv = (vv_1,\cdots, vv_k)$ be a minimum cycle in a coloring $\beta$. Then the cycle $\Vv$ is tight.
\end{observation}
\begin{proof}
Assume that $\Vv$ is not tight, so there exists $i$ such that $v_i\not\in K_{v_{i-1}}(m(v_i),mv_{i-1})$. Without loss of generality, we assume that $i = 2$ and that each $v_j$ is missing the color $j$. Note that this means that $\beta(vv_2) = 1$, $\beta(vv_3) = 2$ and $\beta(vv_1) = k$.

We swap the component $C_{1,2} = K_{v_{1}}(1,2)$ to obtain a coloring $\beta'$ that is $(\Vv \setminus \{v_1\})$-equivalent to the coloring $\beta$. Thus each $v_j$ is missing the color $j$ except $v_1$ which is now missing the color $2$. So now the fan $\Vv' = X_v(k)$ is equal to $(vv_1,vv_3,\cdots,vv_k)$, and thus is a cycle strictly smaller than $\Vv$. Since $\Vv$ is minimum, this cycle is invertible, and we denote by $\beta''$ the coloring obtained after its inversion.

The coloring $\beta''$ is $(G\setminus \Vv')$-identical to the coloring $\beta'$, so in particular it is $C_{1,2}$-identical to the coloring $\beta'$. Moreover, the coloring $\beta''$ is $(\Vv \setminus \{vv_1,vv_2,v_2\})$-identical to the coloring $\Vv^{-1}(\beta)$, and we have $\beta''(vv_1) = 2$, $\beta''(vv_2)= 1$, and $m^{\beta''}(v_2) = 2$.

So now in this coloring the component $K_{v_1}(1,2)$ is exactly $C_{1,2}\cup \{vv_1,vv_2\}$, and we swap back this component to obtain a coloring $\beta'''$. The coloring $\beta'''$ is now $C_{1,2}$-identical to $\beta$, and thus it is $(G\setminus \Vv)$-identical to $\beta$. Moreover, it is $(\Vv \setminus \{vv_1,vv_2,v_2\})$-identical to $\beta''$, so it is $(\Vv \setminus \{vv_1,vv_2,v_2\})$-identical to $\Vv^{-1}(\beta)$. Finally, we have $\beta'''(vv_1) = 1 = m^{\beta}(v_1)$, $\beta'''(vv_2) = 2 = m^{\beta}(v_2)$, and $m^{\beta'''}(v_2) = 1 = \beta(vv_2)$, so the coloring $\beta'''$ is $\Vv$-identical to $\Vv^{-1}(\beta)$. Since it is also $(G\setminus \Vv)$-identical to $\beta$, we have $\beta''' = \Vv^{-1}(\beta)$ as desired.
\end{proof}

The proof of Lemma~\ref{lem:all_cycles_invertible}, is a consequence of the two following Lemmas.

\begin{lemma}\label{lem:only_cycles}
Let $\Vv$ be a minimum cycle. For any color $c$ different from $m(v)$, the fan $X_v(c)$ is a cycle.
\end{lemma}

\begin{lemma}\label{lem:cycles_interactions}
Let $\Xx$ and $\Yy$ be two cycles around a vertex $v$. For any pair of vertices $(z,z')$ in $(\Vv\cup\Xx\cup\Yy)^2$, the fan $\Zz = X_z(c_{z'})$ is a cycle containing $z'$.
\end{lemma}

We prove Lemma~\ref{lem:only_cycles} in section~\ref{sec:only_cycles_around_v}, and Lemma~\ref{lem:cycles_interactions} in section~\ref{sec:cycles_interaction}, and prove here Lemma~\ref{lem:all_cycles_invertible}.

\begin{proof}[Proof of Lemma~\ref{lem:all_cycles_invertible}]
To prove the Lemma, we prove that the graph $G$ only consists of an even clique where each vertex misses a different color. This is a contradiction since in any $(\Delta(G)+1)$-coloring of an even clique, for any color $c$, the number of vertices missing the color $c$ is always even. By Lemma~\ref{lem:only_cycles}, all the fans around $v$ are cycles, so each neighbor of $v$ misses a different color. Moreover, by Lemma~\ref{lem:cycles_interactions}, there is an edge between each pair of neighbors of $v$, so $G = N[v]= K_{\Delta(G)+1}$. By construction, $G$ is $\Delta(G)$-colorable, so $G$ is an even clique and each vertex misses a different color, this concludes the proof.
\end{proof}

\subsection{Only cycles around \texorpdfstring{$v$}{v}: a proof of lemma~\ref{lem:only_cycles}}\label{sec:only_cycles_around_v}
In this section, we prove Lemma~\ref{lem:only_cycles}. If $\Xx$ and $\Xx'$ are two fans, then $\Xx$ and $\Xx'$ are called \textit{entangled} if for any $c\in \beta(\Xx) \cap \beta(\Xx')$, $M(X,c) = M(X',c)$. To prove Lemma~\ref{lem:only_cycles} we need the two following lemmas.

\begin{lemma}\label{lem:fans_around_Vv}
Let $\Vv$ be a minimum cycle in a coloring $\beta$ and let $u$ and $u'$ be two vertices of $\Vv$. Then fan $\Uu = X_u(m(u')) = (uu_1,\cdots,uu_{l})$ is a cycle entangled with $\Vv$.
\end{lemma}

\begin{lemma}\label{lem:cycles_around_v_starting_with_u}
Let $\Vv$ be a minimum cycle in a coloring $\beta$, $u$ and $u'$ be two vertices of $\Vv$, and $\Uu = X_u(m(u')) = (uu_1,\cdots,uu_l)$. Then for any $j\leq l$, the fan $X_v(\beta(uu_j))$ is a cycle.
\end{lemma}

Note that by Lemma~\ref{lem:fans_around_Vv}, we can directly conclude that $N[v]$ is a clique. Moreover, we directly have the following corollary.

\begin{corollary}\label{cor:X_u(m(v))_cycle}
Let $\Vv = (vv_1,\cdots, vv_k)$ be a minimum cycle in a coloring $\beta$. Then for any $j\leq k$, the fan $X_u(m(v))$ is a cycle entangled with $\Vv$.
\end{corollary}
\begin{proof}
Let $j\leq k$ and $\Uu = X_{v_j}(m(v_{j-1})) = X_{v_j}(\beta(vv_j))$. Then the first edge of $\Uu$ is $vv_{j}$, and since $v$ is missing the color $m(v)$, the second edge of $\Uu$ is colored $m(v)$. By Lemma~\ref{lem:fans_around_Vv}, $\Uu$ is a cycle entanlged with $\Vv$, so since $X_u(m(v)) = \Uu$, the fan $X_u(m(v))$ is a cycle entangled with $\Vv$ as desired.
\end{proof}


We prove Lemma~\ref{lem:fans_around_Vv} in Section~\ref{sec:fans_around_Vv}, Lemma~\ref{lem:cycles_around_v_starting_with_u} in Section~\ref{sec:cycles_around_v_starting_with_u}, and prove here Lemma~\ref{lem:only_cycles}.

\begin{proof}[Proof of Lemma~\ref{lem:only_cycles}]
Assume that there exists a fan $\Ww = (vw_1,\cdots,vw_t)$\break around $v$ which does not induce a cycle, we first prove that $\Ww$ is not a path.
\begin{claim}
The fan $\Ww$ cannot induce a path.
\end{claim}
\begin{proof}
Without loss of generality, we assume that the vertex $v$ is missing the color $1$. Assume that $\Ww$ induces a path, so $m(v) = m(w_t) = 1$. Let $v'\in \Vv$, by Corollary~\ref{cor:X_u(m(v))_cycle}, we have that $\Uu = X_{v'}(1)$ is a cycle containing $v$ in $\beta$. If we apply a single-edge Kempe swap on $vw_t$, then we obtain a coloring where $m(w_t) = m(v) = \beta(vw_t)$; we denote by $\beta'$ this coloring, and without loss of generality, we assume that $\beta(vw_t) = 2$. Again, by Corolloary~\ref{cor:X_u(m(v))_cycle}, we also have that $\Uu' = X_{v'}(2)$ is a cycle containing $v$ in the coloring $\beta'$, so $\Uu\cap\Uu'\neq \emptyset$, let $v'w''$ be the first edge they have in common, and let $w = M(\Uu,\beta(v'w''))$ and $w' = M(\Uu',\beta(v'w''))$. We now have to distinguish whether $v\in \{w,w'\}$ or not.
\begin{case}[$v\not\in \{w,w'\}$]
~\newline
In this case, $m_{\beta}(w) = m_{\beta'}(w) = m_{\beta}(w') =m_{\beta'}(w')$; we denote by $c$ this color. By Lemma~\ref{lem:cycles_around_v_starting_with_u}, $X_v(c)$ is a cycle containing $w$ in $\beta$, and $X_v(c)$ is a cycle containing $w'$ in $\beta'$, so $w = w'$; a contradiction.
\end{case}

\begin{case}[$v\in \{w,w'\}$]
~\newline
The case $v = w$ and $v = w'$ being symmetrical, we can assume that $v = w$. In this case, in the coloring $\beta'$, $w'$ is missing the color $c_v$, but by Lemma~\ref{lem:cycles_around_v_starting_with_u} $X_v(1)$ is a cycle containing $w$ or $\Vv$ is invertible, however, in the coloring $\beta'$, $X_v(1)$ induces a path which is a single edge; a contradiction.
\end{case}

\end{proof}

Thus the fan $\Ww$ is not a path. Now assume that $\Ww$ is a comet, then there exists $w$ and $w'$ in $\Ww$ which are missing the same color $c$. At least one of them is not in $K_v(1,c)$, the two cases being symmetrical, we can assume without loss of generality that $w$ is not in $K_v(1,c)$. So if we swap the component $K_w(1,c)$, we obtain a coloring where the fan $X_v(\beta(vw_1))$ is a path; a contradiction, so $\Ww$ is a cycle.
\end{proof}

\section{Fans around \texorpdfstring{$\Vv$}{V}: a proof of Lemma~\ref{lem:fans_around_Vv}}\label{sec:fans_around_Vv}

In this section, we prove Lemma~\ref{lem:fans_around_Vv} which will be often used in the proof of Lemma~\ref{lem:cycles_around_v_starting_with_u}.

\begin{proof}
\setcounter{case}{0}
We first prove that the fan $\Uu$ cannot induce a path.

\begin{claim}
The fan $\Uu$ cannot induce a path.
\end{claim}
\begin{proof}
Otherwise, assume that the fan $\Uu$ is a path, without loss of generality, we can assume that $\Uu$ is of minimal length (if $\Uu$ is not minimal, since it is a path, it contains a strictly smaller path). Thus $\Uu$ contains only one edge colored with a color in $\beta(\Vv)\setminus \{c_v\}$: its first edge. We now need to distinguish whether $j' = j-1$ or not (\textit{i.e.} whether $u = v_j$ and $u' = v_{j'}$ are consecutive or not in $\Vv$).

\begin{case}[$j' = j-1$]
~\\
In this case, $\Uu = X_u(uv)$, and the edge colored $c_v$ incident with $u$ is just after $uv$ in $\Uu$. As $\Uu$ is a path, we can invert it until we reach a coloring where $m(u) = m(v) = c_v$. Since $\Uu$ is minimal, no edge incident with a vertex of $\Vv$ different from $u$ has been recolored during the inversion. In the coloring obtained after the inversion, the fan $(vv_{j+1},\cdots,vv_j = vu)$ is a path that we can invert until we reach a coloring where $m(v) = m(v_{j+1}) = j$, we denote by $\beta'$ this coloring. Since $\Vv$ was tight in the coloring $\beta$, in the coloring $\beta'$ we have $C = K_{v_{j-1}}^{\beta'}(j,j-1) = K_{v_{j-1}}^{\beta}(j,j-1)\cup\{vv_{j-1}\}\setminus \{vv_{j+1},vv_j = vu\}$, so we swap this component to obtain a coloring where $m(v) = m(u) = j-1$, then we swap the edge $uv$ and obtain a coloring where $(uu_{l-1},\cdots,uu_0)$ is a path that we invert. In the coloring obtained after the inversion, we have that the component $K_{v_{j-1}}(j,j-1)$ is exactly $C\cup \{vv_j\}$, if we swap this component back we obtain exactly $\Vv^{-1}(\beta)$.
\end{case}
\begin{case}[$j' \neq j-1$]
~\\
In this case, since $\Uu$ is a path, we can invert it until we reach a coloring $\beta'$ where $m(u) = c_{u'} = j'$. Note that, similarly to the previous case, this inversion has not changed the colors of the edges incident with the vertices of $\Vv$, except those incident with $u$. We now consider the component $K_v(j',c_v)$ (which can have changed during the inversion of $\Uu$ as we swapped an edge colored $j'$), and we need to distinguish whether or not the vertices $u'$ and $u$ belong to this component; clearly these vertices does not both belong to this component.

\begin{subcase}[$u'\not\in K_v(j',c_u)$]
~\newline
In this case, we swap the component $C= K_{u'}(j',c_v)$ to obtain a coloring where\break$(vv_{j+1},\cdots,vv_{j'})$ is a path that we invert until we reach a coloring where\break $m(v)= m(v_{j+1}) = c_u$, we denote by $\beta'$ this coloring. As $\Vv$ was tight in $\beta$, we have that $C_j = K_{v_{j-1}}^{\beta'}(j,j-1) = K_{v_{j-1}}^{\beta}(j,j-1)\setminus \{vv_{j+1},vv_j = vu\}$, so we swap this component to obtain a coloring where $(vv_{j'+1},\cdots,vv_{j-1})$ is a path that we invert until we reach a coloring where $m(v) = m(v_{j'+1}) = j'$. In the coloring obtained after the inversion, the component $K_{u'}(j',c_v)$ is exactly $C\cup\{vu'\}$, thus we swap it back. Note that as $|\{c_{u'},c_v,j,j-1\}| = 4$, we can swap back $C$ before $C_j$. In the coloring obtained after swapping back the component, we have that the fan $(uu_{l-1},\cdots,uu_0)$ is a path that we invert. In the coloring obtained after the inversion, the component $K_{v_{j-1}}(j,j-1)$ is exactly $C\cup\{vv_{j-1},vv_j = vu\}$, thus we swap back this component and obtain exactly $\Vv^{-1}(\beta)$.
\end{subcase}
\noindent So $u'$ belongs to the component $K_v(j',c_u)$.
\begin{subcase}[$u\not\in K_v(c_{u'},c_u)$]
~\newline
In this case, we swap the component $C = K_u(j',c_v)$, note that, from the previous case, neither $v$ nor $u'$ belong to this component. In the coloring obtained after the swap, the fan $(vv_{j+1},\cdots,vv_j)$ is a path that we invert until we reach a coloring where $m(v) = m(v_{j+1}) = c_u$; we denote by $\beta'$ this coloring. As $\Vv$ was tight in $\beta$, we have that $C_j = K_{v_{j-1}}^{\beta'}(j,j-1) = K_{v_{j-1}}^{\beta}(j,j-1)\cup\{vv_{j-1}\}\setminus\{vv_{j+1},vv_j = vu\}$, so we swap this component to obtain a coloring where $m(v) = m(u) = j-1$, then we swap the edge $uv$ to obtain a coloring where $K_u(c_{u'},c_v)$ is exactly $C$. Hence we swap back this component, and in the coloring obtained after the swap, the fan $(uu_{l-1},\cdots,uu_0)$ is a path that we invert until we reach a coloring where $m(u) = j$. In this coloring, the component $K_{v_{j-1}}(j,j-1)$ is exactly $C_j\cup\{vv_{j-1},vv_j = vu\}$, thus we swap back this component to obtain exactly $\Vv^{-1}(\beta)$.
\end{subcase}

\end{case}

\end{proof}

\noindent Before proving that the fan $\Uu$ is not a comet, we prove that $\Uu$ and $\Vv$ are entangled.

\begin{claim}
The fans $\Uu$ and $\Vv$ are entangled.
\end{claim}

\begin{proof}
Assume that $\Uu$ and $\Vv$ are not entangled, then there exist $w = v_s\in \Vv$ and $w' = u_{s'} \in \Uu$ distinct from $w$ with $m(w) = m(w') = c$. If $m(w) = m(v) = c_v$, then, since $\Vv$ is saturated, $w\in K_v(c_v,c)$, so we swap $K_{w'}(c_v,c)$ to obtain a coloring where $\Vv$ is still a cycle of the same size, but where $X_u(c_{u'})$ is a path, by the previous claim, this is a contradiction.  

So $m(w) \neq m(v)$, and therefore, we successively swap the components $K_{w'}(t,t+1)$ with $t \in (s,\cdots,j)$. Note that this sequence of swaps has not changed the colors of the edges incident with a vertex of $\Vv$; it can though have changed the colors of the edges of $\Uu$. However, it is guaranteed that in the coloring obtained after the swaps, there exists a color $c'\in\beta(\Vv)$ such that $X_u(c')$ is a path, which is a contradiction by the previous claim.
\end{proof}

\noindent We now prove that $\Uu$ is not a comet.
\begin{claim}
The fan $\Uu$ is not a comet.
\end{claim}
\begin{proof}
Assume that $\Uu$ is a comet, then there exist $w$ and $w'$ in $\Uu$ with $m(w) = m(w') = c$ and where $w'$ is after $w$ in the sequence. By the previous claim, as $\Uu$ and $\Vv$ are entangled, we have that $c\not \in \beta(\Vv)$. We now consider the component $C_v = K_v(c,c_v)$. If $w'$ is not in $C_v$, then we swap $C_{w'} = K_{w'}(c,c_v)$ to obtain a coloring where $w'$ belongs to the fan $X_u(c_{u'})$ with $m(w') = m(v)$; this contradicts the fact that $X_u(c_{u'})$ and $\Vv$ are entangled. Note that if $u$ is in $C'$, and $m(v)\in\beta(\Uu)$, after swapping $C'$ the sequence $X_u(c_{u'})$ now starts at the edge colored $c$ in $\beta$, but this does not change the reasoning.
So the vertex $w'$ belongs to $C$, and thus the vertex $w$ does not belong to $C_v$, so we can swap $C_{w} = K_w(c,c_v)$ to obtain a coloring where the sequence $X_u(m(u'))$ contains $w$ which is missing the color $m(v)$, a contradiction. Note that if $u\in C_w$, then after swapping $C_w$, we obtain a coloring where $w'$ comes before $w$ in the fan $X_u(m(u'))$. Similarly to the previous case, this does not change the reasoning. 
\end{proof}
\noindent From the previous claims, the fan $\Uu$ is a cycle entangled with $\Vv$ as desired.
\end{proof}

\section[Cycles around \texorpdfstring{$v$}{v} starting with \texorpdfstring{$u$}{u}: a proof of Lemma~\ref{lem:cycles_around_v_starting_with_u}]{Cycles around \texorpdfstring{$v$}{v} starting with \texorpdfstring{$u$}{u}: a proof of\texorpdfstring{\break}{ }Lemma~\ref{lem:cycles_around_v_starting_with_u}}\label{sec:cycles_around_v_starting_with_u}

In this section we prove Lemma~\ref{lem:cycles_around_v_starting_with_u}. To prove the lemma we actually prove a stronger statement, bu we need first some definitions.


\begin{definition}\label{def:P_weak(i)}
Let $i\geq 0$, we define the property $P_{weak}(i)$ as the following:
For any minimum cycle $\Vv$ in a coloring $\beta$, for any pair of vertices $u$ and $u'$ of $\Vv$, let $\Uu = X_u(m(u')) =(uu_1,\cdots,uu_l)$. If $\beta(uu_{l-i}) \neq m(v)$, then $X_v(\beta(uu_{l-i})$ is not a path.
\end{definition}

\begin{definition}\label{def:P(i)}
Let $i\geq 0$, we define the property $P(i)$ as follows:

 For any minimum cycle $\Vv$ in a coloring $\beta$, for any pair of vertices $u$ and $u'$ of $\Vv$, let $\Uu = X_u(m(u')) =(uu_1,\cdots,uu_l)$. If $\beta(uu_{l-i}) \neq m(v)$, then  the fan $X_v(\beta(uu_{l-i})$ is a saturated cycle containing $u_{l-i-1}$,
\end{definition}

Lemma~\ref{lem:cycles_around_v_starting_with_u} is a direct consequence of the following lemma.

\begin{lemma}\label{lem:P(i)_is_true}
The property $P(i)$ is true for all $i$.
\end{lemma}

The proof of the lemma is an induction on $i$. However, before starting to prove the lemma, we need to introduce the notion of $(\Vv,u)$-independent fan for a vertex $u$ of a cycle $\Vv$.

\subsection{$(\Vv,u)$-independent fans}

Let $\Vv$ be a minimum cycle in a coloring $\beta$, and $u$ a vertex of $\Vv$. A $(\Vv,u)$-independent subfan $\Xx$ is a subfan around $v$ such that $\beta(\Vv)\cap \beta(\Xx) = \{\beta(u)\}$. We naturally define a $(\Vv,u)$-independent path (respectively a $(\Vv,u)$-independent cycle) as a $(\Vv,u)$-independent subfan that is also a path (respectively a cycle). If $v$ is a vertex not in $\Xx$ missing a color $c$, we say that $\Xx$ avoids $v$ if the last vertex of $\Xx$ is also missing the color $c$.

We first prove the following.

\begin{lemma}\label{lem:(Vv,u)-independent_subfan_avoiding_v}
Let $\Vv$ be a minimum cycle in a coloring $\beta$, $u$ a vertex of $\Vv$, $\Yy = (uy_1,\cdots,uy_r)$ a $(\Vv,u)$-independent subfan avoiding $v$ and $x$ the extremity of $K_{y_s}(m(u),m(v))$ which is not $y_r$. Then the fan $X_v(\beta(uy_1))$ is a path containing $x$ which is missing the color $m(v)$.
\end{lemma}

We decompose the proof into five separate lemmas.

\begin{proof}[Proof of Lemma~\ref{lem:(Vv,u)-independent_subfan_avoiding_v}]
Without loss of generality, we assume that the vertices $v$ and $u$ are respectively missing the colors $1$ and $2$, and that $\beta(uy_1) = 4$. Since the fan $\Vv$ is a minimum cycle in the coloring $\beta$, it is saturated by Lemma~\ref{lem:minimum_cycle_saturated}, so $u\in K_v(1,2)$ and thus $y_r\not\in K_v(1,2)$. We now swap the component $C_{1,2} = K_{y_r}(1,2)$ to obtain a coloring $\Vv$-equivalent to $\beta$, where $\Yy$ is now a $(\Vv,u)$-independent path. By Lemma~\ref{lem:(Vv,u)-independent_path_touch_other_extremity}, the fan $X_v(4)$ is a comet containing the other extremity of $K_{y_r}(1,2)$ which is $x$. In this coloring, the vertex $x$ is missing the color $2$, therefore in the coloring $\beta$, the fan $X_v(4)$ is a path containing $x$ which is missing the color $1$ as desired.
\end{proof}

\begin{lemma}\label{lem:path_plus_bichromatic_path}
Let $\Xx = (vv_1,\cdots,vv_k)$ be a path of length at least $3$ in a coloring $\beta$, $u = v_i$ for some $i\in [3,k]$, $u'' = v_{i-1}$, $u' = v_1$, and $C$ a $(\beta(vu),m(u))$-bichromatic path between $u''$ and $u'$ that does not contain $v$. Then $\beta$ is equivalent to a coloring $\beta'$ such that:
\begin{itemize}
\item $\beta'$ is $(G\setminus (C\cup \Xx))$-identical to $\beta$,
\item $\beta'$ is $(\Xx_{\geq u})$-identical to $\beta$,
\item for any edge $j\in [2,i-1]$, $m^{\beta'}(v_j) = \beta(vv_j)$,
\item $m^{\beta'}(u') = \beta(vu)$,
\item for any edge $j\in [1,i-2]$, $\beta'(vv_j) = m^{\beta}(v_j)$,
\item $\beta'(vu'') = \beta(vu')$,
\item for any edge $e\in C$:
\begin{itemize}
\item if $\beta(e) = \beta(vu)$, then $\beta'(e) = m^{\beta}(u)$, and
\item if $\beta(e) = m^{\beta}(u)$, then $\beta'(e) = \beta(vu)$.
\end{itemize}
\end{itemize}
\end{lemma}

\begin{proof}
Without loss of generality, we assume that the vertices $v$ is missing the color $1$, that the edge $vu'$ is colored $2$, and that the edge $vu$ is colored $3$. Note that this means that $m(u'') = 3$. In the coloring $\beta$, the fan $\Xx$ is a path, so we invert this path, and denote by $\beta_2$ the coloring obtained after the inversion. The coloring $\beta_2$ is $(G\setminus \Xx)$-identical to the coloring $\beta$ so $C$ is still a $(2,3)$-bichromatic path between $u'$ and $u''$ that does not contain $v$. Moreover, for any edge $j\in[1,i]$, $\beta_2(vv_j) = m^{\beta}(v_j)$ and $m^{\beta_2}(v_j) = \beta(vv_j)$. So the coloring $\beta_2$ is $(\Xx_{[v_2,v_{i-2}]}\cup\{u'',vu'\})$-identical to $\beta'$. The vertex $u'$ is now missing the color $2$, and the edge $vu''$ is now colored $3$. Moreover, the vertex $v$ is now missing the color $2$, so $K_v(2,3) = C\cup \{vu''\}$. We now swap this component and denote by $\beta_3$ the coloring obtained after the swap. 

The coloring $\beta_3$ is $(G \setminus (C \cup\Xx))$-identical to the coloring $\beta$, so it is $(G\setminus (C \cup \Xx))$-identical to $\beta'$. Moreover, for any edge $e\in C$:
\begin{itemize}
\item if $\beta(e) = 2$, then $\beta_3(e) = 3$, and 
\item if $\beta(e) = 3$, then $\beta_3(e) = 2$.
\end{itemize}
So the coloring $\beta_3$ is also $C$-identical to $\beta'$; thus it is $(G\setminus \Xx)$-identical to $\beta'$.

The coloring $\beta_3$ is $(\Xx_{[v_2,v_{i-2}]}\cup\{u'',vu'\})$-identical to $\beta_2$, so it is $(\Xx_{[v_2,v_{i-2}]}\cup\{u'',vu'\})$-identical to $\beta'$. In the coloring $\beta_3$, the edge $vu''$ is now colored $2$, and the vertex $u'$ is now missing the color $3$. So the coloring $\beta_3$ is also $(\{vu'',u'\})$-identical to $\beta'$, and thus it is $\Xx_{<u}$-identical to $\beta'$. In total, the coloring $\beta_3$ is $(G\setminus \Xx_{\geq u})$-identical to the coloring $\beta'$.

Finally, the coloring $\beta_3$ is $\Xx_{\geq u}$-identical to the coloring $\beta_2$ and the verices $v$ and $u$ are both missing the color $3$. So in the coloring $\beta_3$ the fan $X_v(1)$ is now a path. We invert this path and denote by $\beta_4$ the coloring obtained after the inversion. The coloring $\beta_4$ is $\Xx_{\geq u}$-identical to the coloring $\beta$, so it is $\Xx_{\geq u}$-identical to the coloring $\beta'$. Moreover, the coloring $\beta_4$ is also $(G\setminus \Xx_{\geq u})$-identical to the coloring $\beta_3$, so it is $(G\setminus \Xx_{\geq u})$-identical to the coloring $\beta'$. In total the coloring $\beta_4$ is identical to the coloring $\beta'$ as desired.
\end{proof}

\begin{lemma}\label{lem:cycle_plus_bichromatic_path_invertible}
Let $\Vv = (vv_1,\cdots,vv_k)$ a cycle of length at least $3$ in a coloring $\beta$, $u = v_i$, $u' = v_{i+1}$ and $u'' = v_{i-1}$ three consecutive vertices of $\Vv$, $\Yy = (uy_1,\cdots,uy_l)$ a $(\Vv,u)$-independent path, $\beta_{\Yy} = \Yy^{-1}(\beta)$, $C$ a $(\beta(vu),m(u'))$-bichromatic path in the coloring $\beta_{\Yy}$ between $u''$ and $u'$ that does not contain $v$ nor $u$, $X = E(C)\cup E(\Vv) \cup (V(\Vv)\cup \{v\}\setminus \{u\})$, and $\beta'_{\Yy}$ a coloring $X$-equivalent to $\beta_{\Yy}$. If there exists a coloring $\beta'$ equivalent to $\beta'_{\Yy}$ such that:
\begin{itemize}
\item $\beta'$ is $(G\setminus X)$-identical to $\beta'_{\Yy}$,
\item $\beta'$ is $(\Vv\setminus \{u', vu'',u, vu \})$-identical to $\Vv^{-1}(\beta)$,
\item $\beta'(vu'') = \beta'_{\Yy}(vu')$,
\item $\beta'(vu) = m^{\beta'_{\Yy}}(u'')$, and
\item $m^{\beta'}(u') = \beta'_{\Yy}(vu)$,
\item for any edge $e\in C$:
\begin{itemize}
\item if $\beta'_{\Yy}(e) = \beta(vu)$, then $\beta'(e) = m(u')$, and 
\item if $\beta'_{\Yy}(e) = m(u')$, then $\beta'(e) = \beta(vu)$. 
\end{itemize}
\end{itemize}
Then the cycle $\Vv$ is invertible.
\end{lemma}

\begin{proof}
Let $\gamma = \Vv^{-1}(\beta)$. Without loss of generality, we assume that the vertex $v$ and $u$ are respectively missing the colors $1$ and $2$ in the coloring $\beta$, and that $\beta(vu) = 3$. This means that $\beta'(vu'') = \beta'_{\Yy}(vu') = \beta(vu') = m^{\beta}(u) = 2$ and $m^{\beta'}(u') = \beta'_{\Yy}(vu)= \beta(vu) = m^{\beta}(u'') = 3$. By definition the coloring $\beta_{\Yy}$ is $(G\setminus (\Yy\cup\{u\}))$-identical to $\beta$. Since $\Yy$ is a $(\Vv,u)$-independent path, we have $E(\Yy)\cap E(\Vv) = \emptyset$, and $V(\Yy)\cap V(\Vv) = \emptyset$. So, in particular $\beta_{\Yy}(vu) = \beta(vu) = 3$. 
The coloring $\beta'$ is $(\{vu\})$-identical to $\beta'_{\Yy}$, so $\beta'(vu) = 3$.

Since the coloring $\beta'$ is $(G\setminus X)$-identical to $\beta'_{\Yy}$ and $\beta'_{\Yy}$ is $X$-equivalent to $\beta_{\Yy}$, by Observation~\ref{obs:X-equivalent_plus_identical}, there exists a coloring $\beta''$ which is $X$-identical to $\beta'$ and $(G\setminus X)$-identical to $\beta_{\Yy}$. 

The coloring $\beta''$ is $(G\setminus X)$-identical to $\beta_{\Yy}$, so it is $(G\setminus (X\cup \Yy\cup\{u\}))$-identical to $\beta$. This means that $\beta''$ is $(G\setminus (\Vv \cup \Yy \cup C))$-identical to $\beta$, and thus it is $(G\setminus (\Vv \cup \Yy \cup C))$-identical to $\gamma$. Moreover, $\beta''$ is $X$-identical to $\beta'$, and $\beta'$ is $(\Vv\setminus(\{u',vu'',u,vu\})$-identical to $\gamma$, so $\beta''$ is $(\Vv\setminus(\{u',vu'',u,vu\})$-identical to $\gamma$. In total, the coloring $\beta''$ is $(G \setminus (C \cup \Yy \cup \{u',vu'',u,vu\}))$-identical to $\gamma$.

In the coloring $\beta_{\Yy}$, the fan $X_u(2)$ is now a path, and we have $E(X_u(2)) = E(\Yy)$ and $V(X_u(2)) = V(\Yy)$. So in any coloring $(\Yy\cup\{u\})$-identical to $\beta_{\Yy}$, the fan $X_u(2)$ is a path. The $\beta''$ is $(G\setminus X)$-identical to $\beta_{\Yy}$, $E(X)\cap E(\Yy) = \emptyset$ and $V(X)\cap (V(\Yy)\cup\{u\}) = \emptyset$, so $\beta''$ is $(\Yy\cup \{u\})$-identical to $\beta_{\Yy}$, and thus $X^{\beta''}_u(2)$ is a path that we invert. Let $\beta_3$ be the coloring obtained after the inversion.

By definition of $\Yy$, the coloring $\beta_3$ is $(\Yy\cup\{u\})$-identical to the coloring $\beta$. So it is $\Yy$-identical to the coloring $\gamma$, and $u$ is now missing the color $2$. The coloring $\beta_3$ is also $(G\setminus (\Yy\cup \{u\}))$-identical to $\beta''$, so it is $(G\setminus (C\cup\{u',vu'',u,vu\}))$-identical to $\gamma$, and we have $\beta_3(vu'') = \beta''(vu'') = 2$, $\beta_3(vu) = \beta''(vu) = 3$ and $m^{\beta_3}(u') = m^{\beta''}(u') = 3$. Note that the coloring $\beta_3$ is also $C$-identical to the coloring $\beta''$.

The path $C$ is a $(2,3)$-bichromatic path between $u''$ and $u'$ and does not contain $v$ nor $u$, so, in the coloring $\beta_3$, we have $K_{u'}(2,3) = C\cup\{vu'',vu\}$. We now swap this component and denote by $\beta_f$ the coloring obtained after the swap. The coloring $\beta_f$ is $(G\setminus(C\cup\{u',vu'',u,vu\}))$-identical to the coloring $\beta_3$, so it is $(G\setminus (C\cup\{u',vu'',u,vu\}))$-identical to $\gamma$. Moreover, since $\beta_3$ is $C$-identical to $\beta''$, for any edge $e\in C$:
\begin{itemize}
\item if $\beta''(e) = \beta_3(e) = 2$, then $\beta_f(e) = 3$, and 
\item if $\beta''(e) = \beta_3(e) = 3$, then $\beta_f(e) = 2$.
\end{itemize}

So the coloring $\beta_f$ is $C$-identical to the coloring $\beta_{\Yy}$, and thus it is $C$-identical to the coloring $\gamma$. Finally, we have:
\begin{itemize}
\item $m^{\beta_f}(u) = 3 = \beta(vu) = m^{\gamma}(u)$,
\item $\beta_f(vu) = 2 = m^{\beta}(u) = \gamma(vu)$,
\item $m^{\beta_f}(u') = 2 = \beta(vu') = m^{\gamma}(u')$, and
\item $\beta_f(vu'') = 3 = m^{\beta}(u'') = \gamma(vu'')$.
\end{itemize}

Finally we have that $\beta_f$ is $(C\cup \{u',vu'',u,vu\})$-identical to $\gamma$, so it is identical to $\gamma$, and $\Vv$ is invertible as desired.
\end{proof}

\begin{lemma}\label{lem:(Vv,u)-independent_path_not_path}
Let $\Vv = (vv_1,\cdots,vv_i)$ a minimum cycle in a coloring $\beta$, $u = v_i$, $u' = v_1$ and $u'' = v_{i-1}$ three consecutive vertices of $\Vv$,  and $\Yy = (uy_1,\cdots,uy_l)$ a $(\Vv,u)$-independent path, $C = K_{u''}(m(u),m(u''))\setminus \{vu,vu'\}$, and $X = C \cup E(\Vv)\cup (V(\Vv) \cup \{v\}\setminus \{u\})$. In any coloring $\beta'_{\Yy}$ that is $X$-equivalent to the coloring $\beta_{\Yy} = \Yy^{-1}(\beta)$, the fan $X_v(m^{\beta}(u))$ is not a path.
\end{lemma}
\begin{proof}
Without loss of generality, we assume that the vertices $v$ and $u$ are respectively missing the colors $1$ and $2$, and that $\beta(vu) = 3$. This means that $\beta(vu') = m^{\beta}(vu) = 2$, $m^{\beta}(u'') = \beta(vu) = 3$, and $m^{\beta_{\Yy}}(u) = 4$. Assume that $\Xx =X^{\beta'_{\Yy}}_v(2)$ is a path. The vertex $v$ is still missing the color $1$ in the coloring $\beta_{\Yy}$ and thus it is still missing $1$ in $\beta'_{\Yy}$. The coloring $\beta_{\Yy}$ is $(\Vv\setminus \{u\})$-identical to the coloring $\beta$ and so is the coloring $\beta'_{\Yy}$. So $\{u',u''\}\subseteq V(\Xx)$ and $\beta'_{\Yy}(vu) = \beta(vu)$, so $u\in V(\Xx)$, and thus the size of $\Xx$ is at least $3$. Note that this means that $V(\Vv) = V(\Xx_{\leq u})$. 

The cycle $\Vv$ is a minimum cycle in $\beta$, so by Observation~\ref{obs:tight}, it is tight, and in particular, $u\in K_{u''}(2,3)$. So the $C$ is a $(2,3)$-bichromatic path between $u''$ and $u'$ that does not contain $u$ nor $v$. Since $\Yy$ is a $(\Vv,u)$-independent path, the coloring $\beta_{\Yy}$ is $C$-identical to $\beta$. The coloring $\beta'_{\Yy}$ is $C$-equivalent to $\beta_{\Yy}$ so $C$ is still the same bichromatic path in the coloring $\beta'_{\Yy}$. 

Since $\Xx$ is a path of path of length at least $3$, by Lemma~\ref{lem:path_plus_bichromatic_path} there exists a coloring $\beta'$ such that:
\begin{itemize}
\item $\beta'$ is $(G\setminus (C\cup \Xx))$-identical to $\beta'_{\Yy}$,
\item $\beta'$ is $(\Xx_{\geq u})$-identical to $\beta'_{\Yy}$,
\item for any edge $j\in [2,i-1]$, $m^{\beta'}(v_j) = \beta'_{\Yy}(vv_j)$,
\item $m^{\beta'}(u') = \beta'_{\Yy}(vu) = 3$,
\item for any edge $j\in [1,i-2]$, $\beta'(vv_j) = m^{\beta'_{\Yy}}(vv_j)$,
\item $\beta'(vu'') = \beta'_{\Yy}(vu') = 2$,
\item for any edge $e\in C$:
\begin{itemize}
\item if $\beta(e) = \beta'_{\Yy}(vu) = 3$, then $\beta'(e) = m^{\beta'_{\Yy}}(u) = 2$, and
\item if $\beta(e) = m^{\beta'_{\Yy}}(u) = 2$, then $\beta'(e) = \beta'_{\Yy}(vu) = 3$.
\end{itemize}
\end{itemize}

The coloring $\beta'$ is $(G\setminus (C \cup \Xx))$-identical to $\beta'_{\Yy}$, and is $\Xx_{\geq u}$-identical to $\beta'_{\Yy}$. So the coloring $\beta'$ is $(G\setminus X)$-identical to $\beta'_{\Yy}$.

Let $\gamma = \Vv^{-1}(\beta)$. For any $j\in[2,i-2]$, we have $\beta'(vv_j) = m^{\beta'_{\Yy}}(v_j) = m^{\beta}(v_j) = \gamma(vv_j)$, and $m^{\beta'}(v_j) = \beta'_{\Yy}(vv_j) = \beta(vv_j) = m^{\gamma}(v_j)$, so the coloring $\beta'$ is $(\Vv \setminus \{u',vu',u'',vu'',u,vu\})$-identical to $\gamma$. Moreover, $\beta'(vu) =  m^{\beta_{\Yy}}(u') = m^{\beta}(u') = \gamma(vu')$ and $m^{\beta'}(u'') = \beta_{\Yy}(vu'') = \beta(vu'') = m^{\gamma}(u'')$. So in total the coloring $\beta'$ is $(\Vv \setminus \{u',vu'',u,vu\})$-identical to the coloring $\gamma$.

The coloring $\beta'$ is $\Xx_{\geq u}$-identical to $\beta'_{\Yy}$, so in particular, $\beta'(vu) = \beta'_{\Yy}(vu) = m^{\beta'_{\Yy}}(u'')$. We also have that $\beta'(vu'') = 2 = \beta'_{\Yy}(vu')$, and $m^{\beta'}(u') = 3 = \beta'_{\Yy}(vu)$.

Finally, for any edge $e$ in $C$:
\begin{itemize}
\item if $\beta_{\Yy}(e) = \beta(e) = 2$, then $\beta'(e) = 3$, and 
\item if $\beta_{\Yy}(e) = \beta(e) = 3$, then $\beta'(e) = 2$.
\end{itemize}
So by Lemma~\ref{lem:cycle_plus_bichromatic_path_invertible}, the cycle $\Vv$ is invertible; a contradiction.
\end{proof}

\begin{lemma}\label{lem:(Vv,u)-independent_path_cycle}
Let $\Vv = (vv_1,\cdots, vv_k)$ be a minimum cycle in a coloring $\beta$, $u = v_j$, $u' = v_{j+1}$, and $u'' = v_{j-1}$ three consecutive vertices of $\Vv$ and $\Yy$ a $(\Vv,u)$-independent path. Then in the coloring $\beta_{\Yy} = \Yy^{-1}(\beta)$, the fan $\Xx = X_v(m^{\beta}(u)) = (vx_1,\cdots, vx_s)$ is a cycle. 
\end{lemma}
\begin{proof}
By Lemma~\ref{lem:(Vv,u)-independent_path_not_path}, the fan $\Xx$ is not a path. To show that it is a cycle, we prove that $\Xx$ is not a comet. Otherwise, assume that $\Xx$ is a comet, then there exists $i<s$ such that $m(x_i) = m(x_s)$. Without loss of generality, we assume that $m^{\beta}(v) = 1$, $m^{\beta}(u) = \beta(vu') = 2$, $\beta(vu) = m^{\beta}(u'') = 3$ and $m^{\beta_{\Yy}}(x_i) = m^{\beta_{\Yy}}(x_s) = 4$. We now have to distinguish the cases.

\begin{case}[$4\not\in\beta(\Vv)$]
~\newline
In the coloring $\beta$, the fan $\Vv$ is a minimum cycle, so by Observation~\ref{obs:tight}, it is tight and in particular, $u\in K_{u''}(2,3)$. Let $C = K_{u''}(2,3)\setminus \{vu',vu\}$. The path $C$ is a $(2,3)$-bichromatic path between $u''$ and $u'$ which does not contain $v$ nor $u$. Since $\Yy$ is a $(\Vv,u)$-independent path, the coloring $\beta_{\Yy}$ is $C$-identical to $\beta$, and thus $C$ is still a $(2,3)$-bichromatic path between $u''$ and $u'$ which does not contain $u$ nor $v$. Let $X = C \cup E(\Vv)\cup (V(\Vv)\cup \{v\}\setminus \{u\})$. We now consider the components of $K(1,4)$ in the coloring $\beta_{\Yy}$. The vertices $x_i$ and $x_s$ are not both part of $K_v(1,4)$. Note that we may have $x_i = u$. If $x_i$ does not belong to $K_{v}(1,4)$, then we swap the component $C_{1,4} = K_{x_i}(1,4)$ to obtain a coloring $\beta'$ $X$-equivalent to $\beta_{\Yy}$ where the fan $X_v(2)$ is now a path. By Lemma~\ref{lem:(Vv,u)-independent_path_not_path}; this is a contradiction. 

So $x_i\in K_v(1,4)$, and thus $x_s \not \in K_v(1,4)$. Similarly to the previous case, we now swap the component $K_{x_s}(1,4)$ and obtain a coloring $X$-equivalent to $\beta_{\Yy}$ where $X_v(2)$ is a path. By Lemma~\ref{lem:(Vv,u)-independent_path_not_path} this is again a contradiction.   
\end{case}

\begin{case}[$4 \in \beta(\Vv)$]
~\newline
In this case, we have that $x_i\in V(\Vv)$. Since $\Yy$ is a $(\Vv,u)$-independent path, it does not contain any vertex missing the color $4$ so $\beta$ is $\{x_s\}$-identical to $\beta_{\Yy}$, and this vertex is still missing the color $4$ in the coloring $\beta$. Since $\Vv$ is a minimum cycle in the coloring $\beta$, by Lemma~\ref{lem:minimum_cycle_saturated} it is saturated, so $x_i\in K_v(1,4)$, and thus $x_s\not\in K_v(1,4)$. We now swap the component $C_{1,4} = K_{x_s}(1,4)$, and denote by $\beta'$ the coloring obtained after the swap. The fan $\Yy$ was a $(\Vv,u)$-independent path in the coloring $\beta$, so the coloring $\beta'$ is $\Yy$-equivalent to $\beta$, and $\Yy$ is still a $(\Vv,u)$-independent path in this coloring. We now invert $\Yy$ and obtain a coloring $\beta'_{\Yy}$ which is $(X^{\beta_{\Yy}}_v(2)\setminus \{x_s\})$-equivalent to the coloring $\beta_{\Yy}$. So now, in the coloring $\beta'_{\Yy}$, the fan $X_v(2)$ is a path, by Lemma~\ref{lem:(Vv,u)-independent_path_not_path} this is a contradiction.
\end{case}
\end{proof}

\begin{lemma}\label{lem:(Vv,u)-independent_path_touch_other_extremity}
Let $\Vv = (vv_1,\cdots,vv_k)$ a minimum cycle in a coloring $\beta$, $u = v_j$ and $u' = v_{j+1}$ two consecutive vertices of $\Vv$, $\Yy = (uy_1,\cdots, uy_r)$ a $(\Vv,u)$-independent path, and $x$ the extremity of $K_{y_r}(m(u),m(v))$ which is not $y_r$. Then the fan $X_v(\beta(uy_1))$ is a comet containing $x$ which is missing the color $m(u)$.  
\end{lemma}
\begin{proof}
We assume that $\Yy$ is of minimum size such that $\Xx = X_v(\beta(uy_1))$ is not a comet containing $x$ missing the color $m(u)$. Without loss of generality, we assume that $m(v) = 1$, $m(u) = \beta(vu') = 2$, $\beta(uv) = m^{\beta}(u'') = 3$, and $\beta(uy_1) = 4$. 

If $|\Yy| = 1$, then $\Yy$ consists of a single edge. We swap this edge, and denote by $\beta'$ the coloring obtained after the swap. In the coloring $\beta'$, by Lemma~\ref{lem:(Vv,u)-independent_path_cycle}, the fan $X_v(2)$ is a cycle. In this coloring, the vertex $u$ is missing the color $4$, so $4\in\beta'(X_v(2))$. Let $\Xx' = (vx_1,\cdots,vx_s)$ be the maximal subfan of $X_v(2)$ starting with an edge colored $4$, and not containing any edge of $\Vv$. Note that $E(\Xx') = E(\Xx)$ and $V(\Xx') = V(\Xx)$. Note also that we have $m^{\beta'}(x_s) = 2$. The subfan $\Xx'$ does not contain any edge of $\Vv$, thus is does not contain the vertex $u$, and so it does not contain any vertex missing the color $4$. So the coloring $\beta$ is $\Xx$-equivalent to the coloring $\beta'$, and thus in the coloring $\beta$, the fan $X_v(4) = (vx_1,\cdots,vx_s,vu',\cdots,vu)$ is a comet where $x_s$ and $u$ are both missing the color $2$. In the coloring $\beta$, the cycle $\Vv$ is a minimum cycle, so it is saturated by Lemma~\ref{lem:minimum_cycle_saturated}, and thus $u\in K_v(1,2)$ and $x_s\not\in K_v(1,2)$. If $x_s$ is not in $K_{y_r}(1,2)$, then we swap $C_{1,2} = K_{x_s}(1,2)$, to obtain a coloring $\beta''$ which is $((\Xx \cup \Vv \cup \Yy)\setminus \{x_s\})$-equivalent to $\beta$. We now invert the path $\Yy$, and obtain a coloring where $X_v(2)$ is a path, by Lemma~\ref{lem:(Vv,u)-independent_path_cycle} this is a contradiction.

So $|\Yy|>1$. The size of $\Yy$ is minimum, so for any subpath $X_u(\beta(uy_j))$ of $\Yy$ with $j>1$, the fan $X_v(\beta(uy_j))$ is a comet containing $x$. So the fan $X_v(4)$ does not contain any vertex missing a color in $\beta(\Yy)$, otherwise it would be a comet containing $x$. Hence the coloring $\beta_{\Yy} = \Yy^{-1}(\beta)$ is $\Xx$-equivalent to $\beta$. In the coloring $\beta_{\Yy}$, the fan $X_v(2)$ is a cycle by Lemma~\ref{lem:(Vv,u)-independent_path_cycle}. Moreover, it contains the fan $\Xx$ since $u$ is missing the color $4$ in the coloring $\beta_{\Yy}$. Therefore, in the coloring $\beta$, the fan $\Xx = (vx_1,\cdots,vx_s,vu',\cdots,vu)$ is a comet containing $\Vv$ where $x_s$ and $u$ are both missing the color $2$. Similarly to the previous case, since $\Vv$ is a minimum, it is saturated by Lemma~\ref{lem:minimum_cycle_saturated}, so $u\in K_v(1,2)$, and thus $x_s\not\in K_v(1,2)$. If $x_s\not\in K_v(1,2)$, then we swap $C_{1,2} = K_{x_s}(1,2)$, and obtain a coloring where $X_v(4)$ is a path. This coloring is $\Yy$-equivalent to $\beta$, and thus if we invert $\Yy$ we obtain a coloring where $X_v(2)$ is a path, a contradiction by Lemma~\ref{lem:(Vv,u)-independent_path_cycle}. 
\end{proof}

In the following section we prove the property $P(0)$.

\subsection{Proof of $P(0)$}

In this section we prove the following lemma.
\begin{lemma}\label{lem:P(0)}
The property $P(0)$ is true.
\end{lemma}

To prove that $P(0)$ is true, we need the following lemma.

\begin{lemma}\label{lem:edge_between_u_u'_saturated_cycle}
    Let $\Vv = (vv_1,\cdots,vv_k)$ be a minimum cycle in a coloring $\beta$, $u = v_j$ and $u' = v_{j'}$ two vertices of $\Vv$. If $uu'\in E(G)\cap$, and $\beta(uu')\neq m(v)$, then the fan $\Xx = X_v(\beta(uu'))$ is a saturated cycle.
\end{lemma}

The following lemma is the first step of the proof of Lemma~\ref{lem:edge_between_u_u'_saturated_cycle}.

\begin{lemma}\label{lem:edge_between_u_u'_not_path}
Let $\Vv = (vv_1,\cdots,vv_k)$ be a minimum cycle in a coloring $\beta$, $u = v_j$ and $u' = v_{j'}$ two vertices of $\Vv$. If $uu'\in E(G)$ and $\beta(uu')\neq m(v)$, then the fan $\Xx = X_v(\beta(uu'))$ is not a path.
\end{lemma}

\begin{proof}

Otherwise, assume that $\Xx$ is a path. Without loss of generality, we assume that the vertices $v$, $u$ and $u'$ are respectively missing the colors $1$, $2$ and $3$. Since $\beta(uu') \not \in \{1,2,3\}$, we also assume that $\beta(uu') = 4$. Finally, we assume that $\Xx$ is of length one, indeed if the length of $\Xx$ is more than one, we invert it until we reach a coloring $\beta'$ $\Vv$-equivalent to $\beta$ where it has length one without changing the color of $uu'$.

We denote by $x$ the only vertex of $\Xx$, and by $\beta_2$ the coloring obtained after swapping the edge $vx$. The coloring $\beta_2$ is $\Vv$-equivalent to $\beta'$, so $\Vv$ is the same minimum cycle in the coloring $\beta_2$ by Observation~\ref{obs:Vv-minimum-stable}. By Lemma~\ref{lem:fans_around_Vv}, the fans $\Uu = X^{\beta'}_u(3) = (uu_1,\cdots,uu_{l})$ and $\Uu' = X^{\beta_2}_{u'}(3)$ are both cycles and $uu'$ is the last edge of both of these cycles; we denote by $w$ the vertex missing $4$ in $\Uu$. Note that since $\beta(uu') = 4$, the vertex $w$ is the vertex $u_{l-1}$, and $\Uu = (uu_1,\cdots, uw,uu')$. We first remark that $4\not\in\beta(\Vv)$, otherwise the fan $E(\Xx) = E(\Vv)$, and the fan $\Xx$ is a cycle and thus is not a path, as desired.

We first prove some basic properties on the fan $\Uu$.
\begin{proposition}\label{prop:no_edge_colored_with_c_in_beta(Vv)}
The fan $\Uu$ contains an edge colored $1$, and there is no edge colored with a color in $\beta(\Vv)$ between the edge colored $1$ and the edge colored $4$ in $\Uu$.
\end{proposition}
\begin{proof}
We first prove that there is an edge colored $1$ in the fan $\Uu$. Assume that $\Uu$ does not contain any edge colored $1$ in the coloring $\beta'$. Since the fan $\Uu$ is a cycle, it means that it does not contains any vertex missing the color $1$, and in particular it does not contain $v$. So the coloring $\beta_2$ is also $\Uu$-equivalent to the coloring $\beta'$. Therefore, $\Uu = \Uu'$ and $\Uu'$ contains the vertex $w$ that is still missing the color $4$. The fan $\Uu'$ is thus not entangled with $\Vv$, by Lemma~\ref{lem:fans_around_Vv} we have a contradiction.

So the fan $\Uu$ contains an edge colored $1$. Since by Lemma~\ref{lem:fans_around_Vv}, the fan $\Uu$ is a cycle entangled with $\Vv$, it contains the vertex $v$ which is missing the color $1$, and thus it contains the edge $uv$, and also the edge $vv_{j-1}$ (recall that $vv_{j-1}$ is the edge just before $vu = vv_j$ in the sequence $\Vv$). Note that the vertex $u'$ and $v_{j-1}$ may be the same vertex.

We now prove that, in the sequence $\Uu$, there is no edge colored with a color in $\beta(\Vv)$ between the edge colored $1$ and the edge colored $4$. Assume on the contrary that there exists such an edge $uu_t$ colored with a color $c\in \beta(\Vv)$. Similarly to the previous proof, this means that in the coloring $\beta_2$, the fan $X_u(c)$ is the sequence $(uu_t,uu_{t+1},\cdots,uw,uu')$ with $m(w) = m(v) = 4$. So this fan is not entangled with $\Vv$ and by Lemma~\ref{lem:fans_around_Vv} we again get a contradiction.
\end{proof}

\noindent Let $y_1$ be the neighbor of $u$ connected to $u$ by the edge colored $1$, and $y_2$ the vertex just after $y_1$ in the sequence $\Uu$. Note that since $\beta'(uu')\neq 1$, the vertex $y_1$ is different from the vertex $u'$ but may be equal to the vertex $w$. In this case, the vertices $y_2$ and $u'$ are the same vertex.

\begin{proposition}\label{prop:uy_1_belongs_to_K_v(1,5)}
The edge $uy_1$ belongs to the component $K_v(1,\beta(vu'))$.
\end{proposition}
\begin{proof}
Assume that $uy_1$ does not belong to $K_v(1,\beta(vu'))$. If the edge $vu'$ is just after the edge $vu$ in the fan $\Vv$ (\textit{i.e.} if $j' = j+1$), then it means that $\beta(vu') = 3$, and since $\beta(uy_1) = 1$, we have that the vertex $u$ does not belong to the component $K_v(1,3)$. So the fan $\Vv$ is not saturated, by Lemma~\ref{lem:minimum_cycle_saturated} we have a contradiction. So the edge $vu'$ is not the edge just after the edge $vu$ in the fan $\Vv$, and without loss of generality, we assume that $\beta(vu') = 5$.

Let $C_{1,5} = K_{y_1}(1,5)$, we first prove that the vertex $x$ belongs to this component. Since the vertex $y_1$ is not in $K_v(1,5)$, we have that $K_v(1,5) \neq C_{1,5}$. The fan $\Vv$ is a minimum cycle, it is saturated by Lemma~\ref{lem:minimum_cycle_saturated}, so after swapping $C_{1,5}$, we obtain a coloring $\beta''$ $\Vv$-equivalent to $\beta'$. By Observation~\ref{obs:Vv-minimum-stable} the cycle $\Vv$ is the same minimum cycle in this coloring. In the coloring $\beta''$ the edge $uuy_1$ is now colored $5$, and the fan $X_u(5)$ still contains the vertex $w$ missing the color $4$. Moreover, the vertex $x$ is still missing the color $1$, so we swap the edge $vu$ to obtain a coloring $\Vv$-equivalent to $\beta''$ where $X_u(5)$ contains the vertex $w$ which is missing the color $m(v) = 4$. So $X_u(5)$ is not entangled with $\Vv$, and by Lemma~\ref{lem:fans_around_Vv} we have a contradiction.
 
Therefore, the vertex $x$ belongs to the component $C_{1,5}$. We first swap the component $C_{1,5}$ and obtain a coloring $\beta''$ $\Vv$-equivalent to $\beta'$. In the coloring $\beta''$, the fan $X_u(5)$ now contains the vertex $w$ that is still missing $4$. So the vertex $X_v(5)$ contains the vertex $u'$ and we have $X_v(5) = X_v(3)$.

Since the cycle $\Vv$ is minimum, by Observation~\ref{obs:tight}, it is tight. In the coloring $\beta''$, the vertex $x$ is now missing the color $5$, we now apply a sequence a of Kempe swaps of the form $K_x(m(v_{t-1}),m(v_{t}))$ for $t\in (j'-1,j'-2,\cdots,j+1)$ to obtain a coloring $\beta_3$ where $m(x) = m(v_{j-1}) = 2$. Note that each of these swaps is $\Vv$-stable since after each swap the fan $\Vv$ is a minimum cycle and thus is tight. Moreover, since no edge of $\Uu$ between $uy_2$ and $uu'$ is colored with a color in $\beta'(\Vv)$, the coloring $\beta_3$ is $\Uu_{[y_2,w]}$-equivalent to $\beta''$.

Hence we have $X_u(\beta_3(uy_1))_{\leq w} = (uy_1,uy_2,\cdots,uw)$. The edge $uy_1$ may have been recolored during the sequence of swaps, but in the coloring $\beta_3$, $uy_1$ is guaranteed to be colored with a color in $\beta_3(\Vv)$. In the coloring $\beta_3$, the vertices $x$ and $u$ are missing the same color $2$ and the vertex $v$ is still missing the color $1$. the cycle $\Vv$ is minimum, so it is saturated by Lemma~\ref{lem:minimum_cycle_saturated}, and therefore $x \not\in K_v(1,2)$.

We swap the component $C_{1,2} = K_x(1,2)$ to obtain a coloring where $v$ and $x$ are missing the same color $1$ and where the edge $vx$ is colored $4$. We now swap the edge $vx$, and denote by $\beta_4$ the coloring obtained after theses swaps. The coloring $\beta_4$ is $\Vv$-equivalent to $\beta_3$, and is also $X_u(\beta_4(uy_1))_{[y_2,w]}$-equivalent to the coloring $\beta_3$. The vertices $v$ and $w$ are missing the same color $4$, so $X_u(\beta_4(uy_1))$ and $\Vv$ are not entangled in this coloring, and thus by Lemma~\ref{lem:fans_around_Vv} we have a contradiction.
\end{proof}

\begin{proposition}\label{prop:x_in_K_w(2,4)}
In the coloring $\beta_2$, the vertex $x$ belongs to $K_x(2,4)$.
\end{proposition}
\begin{proof}
Otherwise, assume that it is not the case. In the coloring $\beta_2$, the fan $\Vv$ is a minimum cycle, so it is saturated by Lemma~\ref{lem:minimum_cycle_saturated}. Therefore the vertex $u$ belongs to $K_v(2,4)$ and the vertex $w$ does not belong to this component. By Proposition~\ref{prop:no_edge_colored_with_c_in_beta(Vv)} $X_u(1)$ contains the vertex $w$. We swap the component $C_{2,4} = K_w(2,4)$, and obtain a coloring $\beta''$ $\Vv$-equivalent to $\beta'$. By Observation~\ref{obs:Vv-minimum-stable}, the cycle $\Vv$ is still the same minimum cycle in the coloring $\beta''$, and now the vertex $w$ is missing the color $2$. The coloring $\beta''$ is also $X_u(1)_{<w}$-equivalent to the coloring $\beta'$, so where $X_u(1)$ still contains the vertex $w$. The vertex $x$ is still missing the color $4$, so we swap the edge $vu$ to obtain a coloring $\beta_3$ where $X_u(1)$ contains the vertex $w$ missing the color $2$, and thus $X_u(1)$ is a path. By Lemma~\ref{lem:fans_around_Vv} we have a contradiction.
\end{proof}

We are now ready to prove the lemma. We need to distinguish whether or not $j = j'+1$.

\begin{case}[$j = j'+1$]
~\newline
In this case, we have $\beta'(vu) = m^{\beta'}(u') = 3$. In the coloring $\beta'$, the fan $\Vv$ is saturated, so $u'\in K_v(1,3)$ and thus $uy_1\in K_{u'}(1,3)$. Let $C_{1,3} = K_{u'}(1,3)\setminus \{uy_1,vu\}$, $C_{1,3}$ is a $(1,3)$-bichromatic path between $u'$ and $y_1$. In the coloring $\beta_2$, we consider the component $C_{2,4} = K_w(2,4)$; this component contains the vertex $x$ by Proposition~\ref{prop:x_in_K_w(2,4)}. After swapping $C_{2,4}$ we obtain a coloring $\beta_3$ $\Vv$-equivalent to $\Vv$ where the fan $X_u(1)$ is a path. By Observation~\ref{obs:Vv-minimum-stable} the fan $\Vv$ is still the same minimum cycle in the coloring $\beta_3$. Moreover, the coloring $\beta_3$ is $C_{1,3}$-equivalent to the coloring $\beta_2$, and thus $C_{1,3}$-equivalent to the coloring $\beta'$, so $C_{1,3}$ is still a $(1,3)$-bichromatic path between $u'$ and $y_1$. 

By Proposition~\ref{prop:no_edge_colored_with_c_in_beta(Vv)} there is no edge in $E(X_u(1))$ colored with a color in $\beta_4(\Vv)$, so we invert $X_u(1)$ to obtain a coloring $\beta_5$ that is $(C_{1,3}\cup(\Vv\setminus \{u\}))$-equivalent to  $\beta_4$. In the coloring $\beta_4$, the vertex $y_1$ is missing the color $1$, so $K_{u'}(1,3) = C_{1,3}$, and we swap this component; we denote by $\beta_5$ the coloring obtained after the swap.

In the coloring $\beta_5$, the vertices $u$ and $u'$ a both missing the color $1$, so we swap the edge $uu'$ to obtain a coloring where $u$ and $u'$ are missing the color $4$. In the coloring $\beta_5$, the fan $X_v(2)$ is now a path that we invert to obtain a coloring $\beta_6$. In the coloring $\beta_6$, the edge $uw$ is colored $2$, and the vertex $u$ is now missing the color $4$, so $K_u(2,4) = C_{2,4}\cup \{uw\}$, and we swap back this component, we denote by $\beta_7$ the coloring obtained after this swap. Note that since $|\{1,2,3,4\}| = 4$, we can swap back $C_{2,4}$ before $C_{1,3}$.

In the coloring $\beta_7$, the vertices $u$ and $v$ are both missing the color $2$, and the edge $vu$ is colored $3$, so we swap the edge $vu$ to obtain a coloring where $u$ and $v$ are both missing the color $4$. In the coloring obtained after the swap, the vertices $u$ and $y_1$ are both missing the color $3$, so the fan $X_u(4)$ is now a path that we invert. We denote by $\beta_8$ the coloring obtained after the swap.

In the coloring $\beta_8$, the edge $uu'$ is colored $1$, and the edge $uy_1$ is colored $3$, so $K_u(1,3) = C_{1,3} \cup \{uu',uy_1\}$ and this component is a $(1,3)$-bichromatic cycle that we swap. In the coloring obtained after the swap, the component $K_v(3,4) = \{uu',u'\}$, and it suffices to swap this component to obtain exactly $\Vv^{-1}(\beta')$. Since $\Vv$ is a minimum cycle, this is a contradiction.
\end{case}

So $j\neq j+1$, and since the role of $u$ and $u'$ is symmetric, we also have that $j'\neq j+1$. Therefore, inn the cycle $\Vv$, there exists a vertex $v_{j+1}$ and $v_{j-1}$ such that $|\{u,u',v_{j-1},v_{j'-1}\}| = 4$. Without loss of generality, we assume that $\beta'(vu') = m^{\beta'}(v_{j'-1}) = 5$, and that $\beta'(vu) = m^{\beta'}(v_{j-1}) = 6$.

\begin{case}[$j \neq j'+1$]
~\newline
For this case, we need to distinguish the cases based on the shape of\break $C_{1,5} = K_{uy_1}(1,5)$. Since $\Vv$ is saturated in the coloring $\beta'$, by Proposition~\ref{prop:uy_1_belongs_to_K_v(1,5)}, $C_{1,5}$ also contains $v$ and $v_{j'-1}$, and therefore this component is a $(1,5)$-path in this coloring. Moreover, the fan $\Vv$ is tight by Observation~\ref{obs:tight}, so $K_{v_{j-1}}(2,6)$ contains $vv_{j+1}$, and $vu$. Let $C_{2,6} = K_{v_{j-1}}(2,6)\setminus \{vu,vv_{j+1}\}$. The path $C_{2,6}$ is a $(2,6)$-bichromatic path between $v_{j+1}$ and $v_{j-1}$.

There are two cases, in the coloring $\beta'$, either $C_{1,5}$ is such that $u$ is between $v_{j'-1}$ and $y_1$, or $y_1$ is between $v_{j'-1}$ and $u$. We start both cases by swapping $C_{2,4} = K_w(2,4)$ in the coloring $\beta_2$, by Proposition~\ref{prop:x_in_K_w(2,4)} the vertex $w$ belongs to this component, and after the swap we have $m(w) = m(x)= m(u) = 2$. By Proposition~\ref{prop:no_edge_colored_with_c_in_beta(Vv)} $X_u(1)$ is a path that we invert to obtain a coloring $\beta_3$ $(\{uu'\}\cup (\Vv\setminus \{u\}))$-equivalent to $\beta_2$.

In the coloring $\beta_3$, depending on the shape of $C_{1,5}$, either $u$ is in $C = K_{v_{j'-1}}(1,5)$, or $y_1$ belongs to this component. We now have to distinguish the cases. Both cases are pretty similar, their proofs rely on the same principle: apply Kempe swaps to reach a coloring where the edges of $E(\Vv)\cup\{vw'\}$ induce two fans that are cycles smaller than $\Vv$ (and that are invertible since $\Vv$ is minimum).

\begin{subcase}[$u$ belongs to $C$]
~\newline
In this case, $C = K_{v_{j'-1}}(1,5)$ is a $(1,5)$-bichromatic path between $v_{j'-1}$ and $u$ and there is a $(1,5)$-bichromatic path  $C'$ between $y_1$ and $u'$. 

From the coloring $\beta_3$, we swap the component $C$ to obtain a coloring $\beta_4$ where the fan $X_v(5) = (vu',vv_{j'+1},\cdots,vv_{j-1},vu)$ is a cycle strictly smaller that $\Vv$, so since $\Vv$ is minimum, this cycle is invertible. Moreover, the fan $X_v(1) = (vx,vv_{j+1},\cdots,vv_{j'-1})$ is also a cycle strictly smaller than $\Vv$, and so it is also invertible. 

After inverting these two cycles, we obtain a coloring where the component $K_{v_{j'-1}}(1,5) = C \cup \{vv_{j'-1},vu\}$ is  $(1,5)$-bichromatic cycle that we swap back; we denote by $\beta_5$ the coloring obtained after the swap. Now the component $K_{y_1}(1,5)$ is exactly $C'$ and we swap it to obtain a coloring $\beta_6$.

In the coloring $\beta_6$, the fan $X_v(3) = (vu',vu,vv_{j-1}\cdots,vv_{j'+1})$ is now a cycle strictly smaller than $\Vv$, so we invert it. In the coloring obtained after this inversion, the $(2,6)$-bichromatic path $C_{2,6}$ is still a path between $v_{j+1}$ and $v_{j-1}$, but now $v_{j-1}$ is missing the color $6$, and $v_{j+1}$ is missing the color $2$. So $K_{v_{j+1}}(2,6) = C_{2,6}$, and we swap this component. Let $\beta_7$ be the coloring obtained after the swap.

In the coloring $\beta_7$, the fan $X_v(1) = (vu',vv_{j'+1},\cdots,vv_{j-1},vx)$ is now a cycle strictly smaller than $\Vv$ and we invert it. In the coloring obtained after the inversion, $K_{y_1}(1,5)$ is now exactly $C'$, and we swap back this component adn denote by $\beta_8$ the coloring obtained after the swap.

In the coloring $\beta_8$, the vertices $y_1$ and $u$ are both missing the color $1$, so the fan $X_u(2)$ is now a path that we invert to obtain a coloring where $u$ and $w$ are missing the color $2$. In the coloring obtained after the inversion, the component $K_{v_{j+1}}(2,6)$ is exactly $C_{2,6}\cup \{vv_{j-1},vu\}$ and we swap back this component. In the coloring obtained after the swap, the component $K_w(2,4)$ is exactly $C_{2,4}$, and thus after swapping back this component, we obtain exactly $\Vv^{-1}(\beta')$; a contradiction.
\end{subcase}

\begin{subcase}[$y_1$ belongs to $C$]
~\newline
In this case, $C = K_{v_{j'-1}}(1,5)$ is a $(1,5)$-bichromatic path between $v_{j'-1}$ and $y_1$ and there is a $(1,5)$-bichromatic path $C'$ between $u$ and $u'$. 
From the coloring $\beta_3$, we swap the component $C$ to obtain a coloring where $X_v(2) = (vv_{j+1},\cdots,vv_{j'-1},vx)$ is a cycle strictly smaller than $\Vv$, so it is invertible. After inverting it, we obtain a coloring where the component $K_{v_{j-1}}(2,6)$ is exactly $C_{2,6}$. We swap this component and denote by $\beta_4$ the coloring obtained after the swap.

In the coloring $\beta_4$, the fan $X_v(1) = (vv_{j'-1},\cdots,vv_{j+1},vu)$ is now a cycle stricly smaller than $\Vv$, so it is invertible. After inverting it, the component $K_{u'}(1,5)$ is now exactly $C'\cup \{vu',vu\}$ and so it is a $(1,5)$-bichromatic cycle containing $vu$ and $vu'$. After swapping this component, we obtain a coloring where the fan $X_v(1) = (vu',vv_{j'+1},\cdots,vv_{j-1},vx)$ is now a cycle strictly smaller than $\Vv$, and we invert it. We denote by $\beta_5$ the coloring obtained after the inversion.

In the coloring $\beta_5$, the component $K_{v_{j'-1}}(1,5)$ is exactly $C$, and we swap back this component. After the swap we obtain a coloring where the fan $X_v(5) = (vu,vv_{j+1},\cdots, vv_{j'-1})$ is now a cycle strictly smaller than $\Vv$, and so we invert it and denote by $\beta_6$ the coloring obtained after the swap.

In the coloring $\beta_6$, the component $K_{u'}(1,5)$ is now exactly $C'$ and we swap back this component. After the swap we obatin a coloring where $u$ and $y_1$ are both missing the color $1$, so the fan $X_u(2)$ is now a path that we invert. We denote by $\beta_7$ the coloring obtained after the swap.

In the coloring $\beta_7$ the component $K_{v_{j-1}}(2,6)$ is exactly $C_{2,6}\cup \{vv_{j-1},vu\}$ and we swap it back. After the swap of this component, we obtain a coloring where  $K_w(2,4)$ is exactly $C_{2,4}$. After swapping back this component, we obtain exactly $\Vv^{-1}(\beta')$. This is a contradiction.
\end{subcase}
\end{case}

\end{proof}

From the previous lemma we derive the following corollary.

\begin{corollary}\label{cor:edge_u_u'_1_no_path}
    Let $\Vv = (vv_1,\cdots,vv_k)$ be a minimum cycle in a coloring $\beta$, $u = v_j$ and $u' = v_{j'}$ two vertices of $\Vv$. If $uu'\in E(G)$ and $\beta(uu') = m(v)$, then no fan around $v$ is a path.
\end{corollary}
\begin{proof}
    Assume that there exists a fan $\Xx$ around $v$ which is a path. It suffices to swap the last edge $vx$ of $\Xx$ to obtain a coloring $\beta_2$ $(\Vv\cup\{uu'\})$-equivalent to $\beta$ such that $X_v(\beta_2(uu')) = \{vx\}$ is now a path (of length one). By Observation~\ref{obs:Vv-minimum-stable}, the fan $\Vv$ is a minimum cycle in the coloring $\beta_2$, so by Lemma~\ref{lem:edge_between_u_u'_not_path}, we get a contradiction.
\end{proof}

We are now ready to prove Lemma~\ref{lem:edge_between_u_u'_saturated_cycle}

\begin{proof}[Proof of Lemma~\ref{lem:edge_between_u_u'_saturated_cycle}]
    Let $\Vv = (vv_1,\cdots,vv_k)$ be a minimum cycle in a coloring $\beta$, $u = v_j$ and $u' = v_{j'}$ two vertices of $\Vv$. Without loss of generality, we assume that the vertices $v$, $u$ and $u'$ are respectively missing the colors $1$, $2$ and $3$. By Lemma~\ref{lem:fans_around_Vv}, the fan $\Uu = X_u(m(u'))$ is a cycle entangled with $\Vv$, so the edge $uu'$ is in $E(G)$. Assume the $\beta(uu')\neq 1$.
    
We first prove that $X_v(\beta(uu'))$ is a saturated cycle. If $\beta(uu')\in\beta(\Vv)$, then $X_v(\beta(uu'))$ is exactly the fan $\Vv$. Since $\Vv$ is minimum, by Lemma~\ref{lem:minimum_cycle_saturated}, it is saturated, so $X_v(\beta(uu'))$ is a saturated cycle as desired. 
    
Hence assume that $\beta(uu')\not\in\beta(\Vv)$, and without loss of generality, say $\beta(uu') = 4$. By Lemma~\ref{lem:edge_between_u_u'_not_path}, then fan $X_v(4)$ is not a path. 
    
We now prove that $X_v(4)$ is not a comet. Suppose that $X_v(4) = (vw_1,\cdots,vw_t)$ is a comet. So there exists $i< t$ with $m(w_i) = m(w_t)$, we denote by $c$ this color. If $c\in\beta(\Vv)$, the cycle $\Vv$ is a subfan of the fan $X_v(4)$, and thus $w_t = M(\Vv,c)\in V(\Vv)$ and $w_i\not\in V(\Vv)$. Since $\Vv$ is minimum, it is saturated by Lemma~\ref{lem:minimum_cycle_saturated}, so $w_t \in K_v(1,c)$, and thus $w_i \not \in K_v(1,c)$. We now swap the component $K_{w_i}(1,c)$ and obtain a coloring $\beta_2$ $(\Vv\cup\{uu'\})$-equivalent to $\beta$, so the cycle $\Vv$ is also a minimum cycle in the coloring $\beta_2$ by Observation~\ref{obs:Vv-minimum-stable}. In the coloring $\beta_2$, the fan $X_v(\beta_2(uu')) = X_v(4)$ is a now path, by Lemma~\ref{lem:edge_between_u_u'_not_path} this is a contradiction.

So $c\not\in\beta(\Vv)$. The vertices $w_i$ and $w_t$ are not both part of $K_v(1,c)$. If $w_i$ is not in $K_v(1,c)$, we swap $K_{w_i}(1,c)$ and obtain a coloring $\beta_2$, $(\Vv\cup\{uu'\})$-equivalent to $\beta$. So the coloring $\beta_2$, by Observation~\ref{obs:Vv-minimum-stable}, the fan $\Vv$ is a minimum cycle. But the fan $X_v(4) = X_v(\beta_2(uu'))$ is now a path, a contradiction by Lemma~\ref{lem:edge_between_u_u'_not_path}. 

So the vertex $w_i$ belongs to the component $K_v(1,c)$ and thus $w_t$ does not belong to this component. We now swap $K_{w_t}(1,c)$ and obtain a coloring $\beta_2$ which is $(\Vv\cup\{uu'\})$-equivalent to $\beta$. So by Observation~\ref{obs:Vv-minimum-stable}, the fan $\Vv$ is still the same minimum cycle in $\beta_2$, but the fan $X_v(4) = X_v(\beta_2(uu'))$ is now a path, again a contradiction by Lemma~\ref{lem:edge_between_u_u'_not_path}.

Therefore the fan $X_v(4)$ is a cycle. We now prove that is it saturated. Note that since $X_v(4)$ is a cycle, $\beta(X_v(4))\cap\beta(\Vv) = \{1\}$. Assume that $X_v(4) = (vw_1,\cdots,vw_t)$ is not saturated, so there exists $i$ such that $w_i\not\in K_v(1,m(w_i))$. We now have to distinguish whether $w_i = w_t$ or not.
\begin{case}[$w_i\neq w_t$]
~\newline
This case is similar to the case where $X_v(4)$ is a comet. In this case, the vertex $w_i$ is missing a color which is not in $\{1,2,3,4\}$, and we can assume without loss of generality that $m(w_i) = 5$. Since $w_i$ does not belong to $K_v(1,5)$, we swap the component $K_{w_i}(1,5)$ to obtain a coloring $\beta_2$ $(\Vv\cup\{uu'\})$-equivalent to $\beta$. In the coloring $\beta_2$, by Observation~\ref{obs:Vv-minimum-stable}, the fan $\Vv$ is the same minimum cycle, but the fan $X_v(4) = (w_1,\cdots,w_i)$ is now a path, a contradiction by Lemma~\ref{lem:edge_between_u_u'_not_path}.
\end{case}

\begin{case}[$w_i = w_t$]
~\newline
In this case, $w_t$ does not belong to $K_v(1,4)$. We first swap the component $C_{1,4} = K_{w_t}(1,4)$. If $uu'\not\in C_{1,4}$, then we obtain a coloring $\beta_2$ $(\Vv\cup\{uu'\})$ equivalent to $\beta$. So by Observation~\ref{obs:Vv-minimum-stable}, the fan $\Vv$ is a minimum cycle in the coloring $\beta_2$, but now the fan $X_v(4) = X_v(\beta_2(uu')) = (vw_1,\cdots,vw_t)$ is now a path; a contradiction by Lemma~\ref{lem:edge_between_u_u'_not_path}. 

So the edge $uu'$ is in $C_{1,4}$. After swapping $C_{1,4}$, we obtain a coloring $\beta_2$ $(\Vv)$-equivalent to $\beta$, so $\Vv$ is still a minimum cycle. But now $\beta_2(uu') = 1$, and $X_v(4)$ is a path, so by Corollary~\ref{cor:edge_u_u'_1_no_path}, we have a contradiction.
\end{case}

Hence $X_v(4)$ is a saturated cycle as desired.
    
\end{proof}

The proof of $P(0)$ is a direct consequence of the two previous lemmas.


\begin{proof}[Proof of Lemma~\ref{lem:P(0)}]
Let $\Vv$ be a minimum cycle around a vertex $v$ in a coloring $\beta$, $u$ and $u'$ two vertices of $\Vv$, $\Uu = X_u(m(u')) = (uu_1,\cdots,uu_l)$, assume that $\beta(uu')\neq m(v)$ and let $\Ww = X_v(\beta(uu_l)) = (vw_1,\cdots,vw_s)$. Without loss of generality, we assume that the vertices $v$, $u$ and $u'$ are respectively missing the colors $1$, $2$, and $3$, and that the edge $uu'$ is colored $4$. 

We first prove that $\Ww$ is a saturated cycle containing $u_{l-1}$. By Lemma~\ref{lem:edge_between_u_u'_saturated_cycle}, the fan $\Ww$ is a saturated cycle, and thus $w_s$ is missing the color $4$. We now prove that the fan $\Ww$ contains the vertex $u_{l-1}$. 

If $4\in\beta(\Vv)$, then $\Ww = \Vv$, and since $\Uu$ is entanlged with $\Vv$ by Lemma~\ref{lem:fans_around_Vv}, we have that $u_{l-1} = w_s \in \Vv = \Ww$. So the color $4$ is not in $\beta(\Vv)$. 

Assume that the fan $\Ww$ does not contain $u_{l-1}$, so in particular, $u_{l-1}\neq w_s$. The cycle $\Ww$ is saturated, so $w_s\in K_v(1,4)$, and thus $u_{l-1}\not\in K_v(1,4)$. By Lemma~\ref{lem:u_i_not_in_K_v(m(u_i),m(v))}
\begin{itemize}
\item $u\in K_{u_{l-1}}(1,4)$,
\item there exists $j\leq l-1$ such that $\beta(uu_j) = 1$, and 
\item the subfan $(uu_{j+1},\cdots, uu_{l-1})$ is a $(\Vv,u)$-independent subfan.
\end{itemize}

We now consider the coloring $\beta'$ obtained from $\beta$ after swapping the component $C_{1,4} = K_{u_{l-1}}(1,4)$. Let $\Xx = (uu_{j+1},\cdots, uu_{l-1})$. The coloring $\beta'$ is $(\Vv\cup \Ww \cup (\Xx\setminus \{u_{l-1}\}))$-equivalent to $\beta$, so $\Vv$ is a minimum by Observation~\ref{obs:Vv-minimum-stable}, and $\Ww = X_v(4)$ is still a cycle. The vertex $v$ is still missing the color $1$, but now the vertex $u_{l-1}$ is missing the color $1$, the edge $uu_j$ is colored $4$, and the edge $uu_l$ is colored $1$. So now $X_u(4) = \Xx' = (uu_j,\cdots,uu_i)$ is a $(\Vv,u)$-independent subfan avoiding the vertex $v$. By Lemma~\ref{lem:(Vv,u)-independent_subfan_avoiding_v} the fan $X_v(4)$ is a path; a contradiction.
\end{proof}

We now prove some properties of the fans around a vertex of a minimum cycle.

\subsection{Fans around the vertices of a minimum cycle}

We first prove that some fans around a vertex of a minimum cycle are not paths.

\begin{proposition}\label{prop:X_u_i_c_not_a_path}
    Let $\Vv = (v_1,\cdots,v_k)$ be a minimum cycle in a coloring $\beta$, $u = v_j$ and $u' = v_{j'}$ two vertices of $\Vv$, and $\Uu = X_u(m(u')) = (uu_1,\cdots,uu_l)$, and $w = u_s$ a vertex of $\Uu$. Then for any color $c\in\beta(\Vv)$, the fan $\Ww = X_w(c) = (ww_1,\cdots,ww_t)$ is not a path.
\end{proposition}
\begin{proof}
    Otherwise assume that the fan $\Ww$ is a path. The vertex $w$ is not a vertex of $V(\Vv)$, otherwise since $\Ww$ is a path, by Lemma~\ref{lem:fans_around_Vv} we have a contradiction. So the vertex $w$ is not in $V(\Vv)$.
    
    We invert it until we reach a coloring $\beta_2$ where $m^{\beta_2}(w)\in\beta(\Vv\cup\Uu_{<s})$, we denote by $c'$ this new missing color. Since $c\in\beta(\Ww)$, the color $c'$ is well defined. The coloring $\beta_2$ is $(\Vv\cup\Uu_{<s})$-equivalent to $\beta$. Thus by Observation~\ref{obs:Vv-minimum-stable}, the sequence $\Vv$ is still a minimum cycle in the coloring $\beta_2$. Let $\Uu' = X_u(m^{\beta_2}(u')) = (uu'_1,\cdots,uu'_{l'})$. Since $\beta_2$ is $(\Uu_{<s})$-equivalent to $\beta$, we have that $\Uu_{<s} = \Uu'_{<s}$, so the edge $uw$ is also in $E(\Uu')$, it is exactly the edge $uu'_s$. If $c'\in\beta(\Vv)$, then $\Uu'$ is not entangled with $\Vv$ in the coloring $\beta_2$, a contradiction by Lemma~\ref{lem:fans_around_Vv}. If $c'\in \beta{\Uu_{<s}}$, then $\Uu'$ is now a comet in the coloring $\beta_2$, again, by Lemma~\ref{lem:fans_around_Vv} we have a contradiction.
\end{proof}

\begin{lemma}\label{lem:X_u_i(beta_(Vv)_cup_beta(Uu_<u_i))_not_a_path}
Let $i\geq 0$, $\Vv$ be a minimum cycle in a coloring $\beta$, $u$ and $u'$ two vertices of $\Vv$, $\Uu = X_u(m(u')) = (uu_1,\cdots,uu_l)$, and $c\in \beta(\Vv) \cup \beta(\Uu_{<u_i})$. If $u_i\not\in V(\Vv)\cup\{v\}$. Then the fan $\Xx = X_{u_i}(c) = (u_ix_1,\cdots, u_ix_s)$ is not a path.
\end{lemma}
\begin{proof}
Assume $u_i\not\in V(\Vv)$ and that $\Xx$ is a path. Without loss of generality, we assume that there is no edge in $\Xx_{[x_2,x_s]}$ colored with a color in $\beta(\Vv)\cup\beta(\Uu_{<u_i})$, otherwise, it suffices to consider the subfan of $\Xx$ starting with this edge, this fan is also a path. We now invert $\Xx$ and obtain a coloring $\beta'$ where $m(u_i) = c$. The coloring $\beta'$ is $(\Vv\cup \Uu_{<u_i})$-equivalent to $\beta$. So by Observation~\ref{obs:Vv-minimum-stable}, the fan $\Vv$ is a minimum cycle in the coloring $\beta'$.
If $c\in\beta(\Vv)$, now the fan $X_u(m(u'))$ contains the vertex $u_i$ which is missing the color $c\in\beta(\Vv)$, so $X_u(m(u'))$ is not entangled with $\Vv$. If $m(u_i)\in \beta(\Uu_{<u_i}$, let $u'' = M(\Uu_{<u_i},c)$. Then $X_u(m(u'))$ is now a comet since it contains the vertices $u_i$ and $u''$ both missing the color $c$. In both cases, by Lemma~\ref{lem:fans_around_Vv} we have a contradiction.
\end{proof}

We now prove a sufficient condition for a fan around a vertex of a minimum to contain an edge colored with the color missing at the central vertex of the minimum cycle.

\begin{lemma}\label{lem:u_i_not_in_K_v(m(u_i),m(v))}
Let $\Vv$ be a minimum cycle in a coloring $\beta$, $u$ and $u'$ two vertices of $\Vv$, $\Uu = X_u(m(u')) = (uu_1,\cdots,uu_l)$ and $i\leq l$. If $\beta(uu_i)\neq m(v)$, $m(u_i)\not\in \beta(\Vv)$  and $u_i\not\in K_v(m(v),m(u_i))$, then:
\begin{itemize}
\item $u\in K_{u_i}(m(v),m(u_i))$,
\item there exists $j< i$ such that $\beta(uu_j) = m(v)$, and 
\item the subfan $(uu_{j+1},\cdots, uu_i)$ is a $(\Vv,u)$-independent subfan.
\end{itemize}
\end{lemma}
\begin{proof}
Let $\Vv$, $u$, $u'$ and $\Uu$ be as in the lemma. Without loss of generality, we assume that the vertices $v$, $u$, and $u'$ are respectively missing the colors $1$, $2$ and $3$. Assume that $u_i\not\in K_v(1,m(u_i))$. Since the cycle $\Vv$ is minimum, by Lemma~\ref{lem:minimum_cycle_saturated} it is saturated, so for any $u''\in V(\Vv)$, $u''\in K_v(m(v),m(u''))$, thus $u_i\not\in V(\Vv)$, so without loss of generality, we may assume that $u_i$ is missing the color $4$. We first consider the component $C_{1,4} = K_{u_i}(1,4)$. In the coloring $\beta$, by Lemma~\ref{lem:fans_around_Vv}, $\Uu$ is a cycle entangled with $\Vv$, so it does not contain any other vertex missing $4$. Since $v\not\in C_{1,4}$, then $V(C_{1,4})\cap V(\Uu) = \{u_i\}$. After swapping $C_{1,4}$, we obtain a coloring $\beta'$ $(V(\Uu)\setminus \{u_i\})$-identical to $\beta$ where $u_i$ is now missing the color $1$. Note that the coloring $\beta'$ is also $\Vv$-equivalent to $\beta$, and thus $\Vv$ is still a minimum cycle in $\beta'$. Moreover the vertex $v$ is still missing the color $1$ in $\beta'$.

We first prove that the vertex $u$ belongs to $C_{1,4}$ and that there is an edge colored $1$ in $\{uu_1,\cdots,uu_{i-1}\}$. If $u\not \in C_{1,4}$, or if there is no edge colored $1$ in $\{uu_1,\cdots,uu_{i-1}\}$, then the coloring $\beta'$ is also $(E(\Uu_{[uu_1,uu_i]}))$-identical to $\beta$, and so $X_u(3)$ now contains the vertex $u_i$ which is missing the color $1$, so $X_u(3)$ is not a cycle entangled with $\Vv$. Since the cycle $\Vv$ is minimum, we have a contradiction by Lemma~\ref{lem:fans_around_Vv}. So $u\in C_{1,4}$ and there is an edge $uu_j$ with $j< i$ colored $1$.

We now prove that $(uu_{j+1},\cdots, uu_i)$ is a $(\Vv,u)$-idenpendent subfan. Note that we have have $j+1 = i$ (\textit{i.e.} the subfan is of length $1$). Since $\beta'$ is $(V(\Uu)\setminus\{u_i\})$-identical to $\beta$, the sequence $(uu_{j+1},\cdots, uu_i)$ is a subfan. Assume that there exists $s\in \{j+1,\cdots, i\}$ such that $\beta(uu_s)\in \beta(\Vv)$. Then, in the coloring $\beta'$, $X_u(\beta(uu_s))$ contains the vertex $u_i$ that is missing the color $1$, thus it is not a cycle entangled with $\Vv$, by Lemma~\ref{lem:fans_around_Vv}, this is a contradiction.
\end{proof}

\begin{lemma}\label{lem:u_i_not_in_K_v(m(u_i),m(v))_generalized}
Let $\Vv$ be a minimum cycle in a coloring $\beta$, $u$ and $u'$ two vertices of $\Vv$, $\Uu = X_u(m(u')) = (uu_1,\cdots,uu_l)$ and $i\leq l$ such that $m(u_i)\not\in \beta(\Vv)$. Let $\beta'$ be a coloring obtained from $\beta$ by swapping a $(m(v),c)$-component $C$ that does not contain $v$ for some color $c \not \in (\beta(\Uu_{<u_i})\cup\{m(v)\})$. If there exists a coloring $\beta''$ such that:
\begin{itemize}
\item $\beta''$ is $(\Vv\cup\Uu_{<u_i})$-equivalent to $\beta'$, and 
\item $m^{\beta''}(u_i)\in \beta''(\Vv)\cup\beta''(\Uu_{<u_i})$.
\end{itemize}
Then
\begin{itemize}
\item $u\in C$,
\item there exists $j< i$ such that $\beta(uu_j) = m(v)$, and 
\item the subfan $(uu_{j+1},\cdots, uu_i)$ is a $(\Vv,u)$-independent subfan in $\beta$.
\end{itemize}
\end{lemma}

\begin{proof}
Let $\Vv$, $\Uu$, $u$, $u'$, $\beta'$, and $c$ be as in the lemma. Without loss of generality, we assume that the vertices $v$, $u$, and $u'$ are respectively missing the colors $1$, $2$ and $3$. Assume that there exists such a coloring $\beta''$. Note that since $m(u_i)\not\in\beta(\Vv)$, the vertex $u_i$ is not in $\Vv$. The cycle $\Vv$ is a minimum cycle in $\beta$, so it is saturated by Lemma~\ref{lem:minimum_cycle_saturated}. Therefore, if $c\in\beta(\Vv)$, then $M(\Vv,c)\in K_v(1,c)$, and thus $M(\Vv,c)\not\in C$. So $\beta'$ is $V(\Vv)$-equivalent to $\beta$. Moreover, $v\not\in C$ so $\beta'$ is also $(E(\Vv)\cup\{v\})$-equivalent to $\beta$. Therefore, the coloring $\beta'$ is $(\Vv\cup\{v\})$-equivalent to $\beta$.

We first prove that the vertex $u$ belongs to $C$ and that there exists an edge colored $1$ in $\Uu_{<u_i}$. Assume that $u$ does not belong to $C$, or that there is no edge colored $1$ in $\Uu_{<u_i}$ in $\beta$. We show that $\beta''$ is $(\Vv\cup\Uu_{<u_i})$-equivalent to $\beta$. To prove it, it suffices to prove that $\beta'$ is $\Uu_{<u_i}$-equivalent to $\beta$. The swap between $\beta$ and $\beta'$ only changes the colors of edges colored $1$ or $c$. Since $\{1,c\}\cap\beta(\Uu_{<u_i}) = \emptyset$ this means that the coloring $\beta'$ is $\Uu_{<u_i}$-equivalent to $\beta$. Since $\beta'$ is also $(\Vv\cup\{v\})$-equivalent to $\beta$ , in total it is $(\Vv\cup\Uu_{<u_i})$-equivalent to $\beta$. Note that the missing color of $v$ may be different in $\beta'$ and $\beta''$. Since $\beta''$ is $(\Vv\cup\Uu_{<u_i})$-equivalent to $\beta'$, the coloring $\beta''$ is $(\Vv\cup \Uu_{<u_i})$-equivalent to $\beta$. Note that the missing color of $v$ may be different in $\beta'$ and $\beta''$. Hence, in the coloring $\beta''$, by Observation~\ref{obs:Vv-minimum-stable}, the cycle $\Vv$ is a minimum cycle and we have that $X_u(m(u'))$ now contains the vertex $u_i$ which is missing a color in $(\beta''(\Vv)\cup\beta''(\Uu_{<u_i}))$. Let $c'$ be this color. Since the cycle $\Vv$ is minimum in $\beta''$, by Lemma~\ref{lem:fans_around_Vv}, $X_u(m(u'))$ is a cycle entangled with $\Vv$. If $c'\in\beta''(\Vv)$, then $X_u(m(u'))$ is not entangled with $\Vv$, and if $c'\in\beta''(\Uu_{<u_i})$ then $X_u(m(u'))$ is a comet. In both cases, we have a contradiction. So $u\in C$, and there exists $j<i$ such that $\beta(uu_j) = 1$.

We now prove that the subfan $\Xx = (uu_{j+1},\cdots, uu_i)$ is a $(\Vv,u)$-independent subfan in $\beta$. Note that we have have $j+1 = i$ (\textit{i.e.} the subfan is of length $1$). If $\Xx$ is not a $(\Vv,u)$-independent subfan, then there exists $s\in \{j+1,\cdots, i\}$ such that $c'\in\beta(uu_s)\in \beta(\Vv)$. Recall that the coloring $\beta'$ is $(\Vv\cup\{v\})$-equivalent to $\beta$, and thus that $\beta''$ is also $(\Vv\cup\{v\})$-equivalent to $\beta$. In the coloring $\beta'$, the edge $uu_j$ is now colored $c$, and this is the only edge in $E(\Uu_{<u_i})$ that has been recolored during the swap of $C$. Moreover, the cycle $\Vv$ is minimum in $\beta$, and thus by Lemma~\ref{lem:fans_around_Vv}, then fan $\Uu$ is a cycle, and does not contain any vertex missing the color $1$ in $V(\Uu_[u_j,u_i])$. Since $c\not\in\beta(\Uu_{<u_i})$, the coloring $\beta'$ is also $V(\Uu_[u_j,u_i])$-equivalent to the coloring $\beta$, and so it is $\Xx$-equivalent to $\beta$. The coloring $\beta''$ is $\Uu_{<u_i}$-equivalent to $\beta'$, so it is $(\Xx\setminus\{u_i\}$-equivalent to $\beta$ and thus the fan $X_u(c')$ now starts with the edge $uu_s$ and contains the vertex $u_i\not\in V(\Vv)$ which is missing a color in $\beta(\Vv) = \beta''(\Vv)$. Since $\Vv$ is also a minimum cycle in $\beta''$, by Lemma~\ref{lem:fans_around_Vv}, $X_v(c')$ is a cycle entangled with $\Vv$; this is a contradiction.
\end{proof}

In the following section we prove some properties that are guaranteed if the property $P$ is true up to some $i$.

\subsection{Properties guaranteed by $P(i)$}

The following lemma guarantees that the last vertices of two cycles will be the same.

\begin{lemma}\label{lem:not_changing_last_vertices}
Let $i\geq 0$, $\Vv$ be a minimum cycle around a vertex $v$ in a coloring $\beta$, $u$ and $u''$ two vertices of $\Vv$, $\Uu = X^{\beta}_u(m^{\beta}(u'')) = (uu_1,\cdots,uu_l)$, $\beta'$ a coloring $(\Vv\cup \Uu_{[u_{l-(i-1))},u_l]} \cup \bigcup\limits_{j\in [0,i-1]} X_v(\beta(uu_{l-j}))$-equivalent to $\beta$ and $\Uu' = X^{\beta'}_u(m^{\beta'}(u'')) = (uu'_1,\cdots,uu'_s)$.
If
\begin{itemize}
\item for any $j< i$ $P(j)$ is true, and
\item $\{m^{\beta}(v),m^{\beta'}(v)\}\cap \beta(\Uu_{[u_{l-(i-1)},u_l]})) = \emptyset$
\end{itemize}
then for any $t\leq i$, $u_{l-t} = u'_{s-t}$.
\end{lemma}

\begin{proof}
Let $i$, $\Vv$, $\Uu$, $\Uu'$, $\beta$, $\beta'$, $v$, $u$, and $u''$ be as in the lemma. Assume that $P(j)$ is true for all $j< i$, that $\{m^{\beta}(v),m^{\beta'}(v)\}\cap \beta(\Uu_{[u_{l-(i-1)},u_l]})) = \emptyset$ and that there exists $t\leq i$ such that $u_{l-t} \neq u_{s-t}$, without loss of generality, we may assume that such a $t$ is minimum. The cycle $\Vv$ is a minimum cycle in $\beta$, and $\beta'$ is $\Vv$-equivalent to $\beta$, so by Observation~\ref{obs:Vv-minimum-stable}, the cycle $\Vv$ is also a minimum cycle in the coloring $\beta'$. Therefore by Lemma~\ref{lem:fans_around_Vv}, the fans $\Uu$ and $\Uu'$ are both cycles entangled with $\Vv$ respectively in $\beta$ and $\beta'$. Note that since $m^{\beta}(v)\not\in \beta(\Uu_{[u_{l-(i-1)},u_l]}))$, and that $P(j)$ is true for all $j< i$, for all $j< i$, the fan $X_v(\beta(uu_{l-j})$ is a cycle containing $u_{l-j-1}$. Moreover, this also means that no vertex in $V(\Uu_{[u_{l-i},u_l]})$ is missing the color $m(v)$, and thus none of them is $v$. Therefore the vertex $v$ may be missing a different color in $\beta$ and in $\beta'$. Note also that, in the coloring $\beta$, the edge $uu_{l-i}$ may be colored $m^{\beta}(v)$ or $m^{\beta'}(v)$.

We first show that $t\neq 0$. Since the fans $\Uu$ and $\Uu'$ are both cycles we have $m^{\beta}(u_l) = m^{\beta'}(u_s) = m^{\beta}(u'')$, and moreover, $\Uu$ and $\Uu'$ are entangled with $\Vv$ so $u_l = u'' = u'_s$, and thus $t\neq 0$.

Since $t$ is minimum, $u_{l-(t-1)} = u'_{s-(t-1)}$. Moreover, $\beta'$ is $\Uu_{[u_{l-(i-1)},u_l]}$-equivalent to $\beta$, so in particular $\beta(uu_{l-(t-1)}) = \beta'(uu_{l-(t-1)}) = \beta'(uu'_{s-(t-1)})$, without loss of generality, we assume that this color is $1$. This means that both the vertices $u_{l-t}$ and $u'_{s-t}$ are missing the color $1$. To reach a contradiction we show that both these vertices belong to a same cycle. Since $1\neq m^{\beta}(v)$ and $P(t-1)$ is true, then $X^{\beta}(1)$ is a cycle containing $u_{l-t}$. Similarly $1\neq m^{\beta'}(v)$ and $P(t-1)$ is true so $X^{\beta'}_v(1)$ is a cycle containing $u'_{s-t}$. However, the coloring $\beta'$ is $(\bigcup\limits_{j\in [0,i-1]} X_v(\beta(uu_{l-j}))$-equivalent to the coloring $\beta$, so in particular $X^{\beta}_v(1) = X^{\beta'}_v(1)$; we denote by $\Xx$ this fan. The fan $\Xx$ is a cycle and contains two vertices $u_{l-t}$ and $u_{s-t}$ that are both missing the color $1$, this is a contradiction. 

\end{proof}

Now we prove that we can guarantee that there is no path around the central vertex of a minimum cycle

\begin{lemma}\label{lem:P(i-1)_and_P_weak(i)_no_path}
Let $i\geq 0$, $\Vv$ be a minimum cycle around a vertex $v$ in a coloring $\beta$, $u$ and $u''$ two vertices of $\Vv$, $\Uu = X^{\beta}_u(m(u'')) = (uu_1,\cdots,uu_l)$, and $\Xx = (vx_1,\cdots,vx_s)$ a fan around $v$. If
\begin{itemize}
\item for any $j< i$ $P(j)$ is true, 
\item $P_{weak}(i)$ is true, and
\item $\beta(uu_{l-i}) = m(v)$,
\end{itemize}
then $\Xx$ is not a path.
\end{lemma}

\begin{proof}
Let $i$, $\Vv$, $\Uu$, $\Xx$, $\beta$, $v$, $u$, and $u''$ be as in the lemma, and without loss of generality we assume that $m(v) = 1$. Assume that for any $j< i$ $P(j)$ is true, that $P_{weak}(i)$ is true, that $\beta(uu_{l-i}) = m(v)$ and that $\Xx$ is a path. The fan $\Xx$ is a path so the vertex $x_s$ is also missing the color $1$, without loss of generality, we assume that $\beta(vx_s) = 2$. Note that this means that $X_v(2)$ is also a path (of length $1$). The cycle $\Vv$ is minimum and by Lemma~\ref{lem:fans_around_Vv} the fan $\Uu$ is a cycle entangled with $\Vv$. Since $\beta(uu_{l-i}) = 1$, no edge in $E(\Uu_{[u_{l-(i-1)},u_l]})$ is colored $1$. Since $P(j)$ is true for all $j<i$, $X_v(\beta(uu_{l-j})$ is a cycle for all $j<i$; since $X_v(2)$ is a path, no edge in $E(\Uu_{[u_{l-(i-1)},u_l]})$ is colored $2$ either.

We now consider the coloring $\beta'$ obtained from $\beta$ by swapping the edge $vx_s$. Note that in the coloring $\beta'$, the vertex $v$ is now missing the color $2$, and the fan $X^{\beta'}_v(1)$ is now a path (of length $1$). The coloring $\beta'$ is clearly $\Vv$-equivalent to $\beta$ so by Observation~\ref{obs:Vv-minimum-stable}, the fan $\Vv$ is a minimum cycle in the coloring $\beta'$. Let $\Uu' = X^{\beta'}_u(m(u'')) = (uu'_1,\cdots,uu'_s))$. No edge in $E(\Uu_{[u_{l-(i-1)},u_l]})$ is colored $1$, so no vertex in $V(\Uu_{[u_{l-i},u_l]})$ is missing the color $1$, and thus $\beta'$ is also $\Uu_{[u_{l-(i-1))},u_l]}$-equivalent to $\beta$. Finally since no edge in  $E(\Uu_{[u_{l-(i-1)},u_l]})$ is colored $1$ and $P(j)$ is true for all $j<i$, the fans $X_v(\beta(uu_{l-j}))$ are cycles for all $j<i$. Therefore the coloring $\beta'$ is also $(\bigcup\limits_{j\in [0,i-1]} X_v(\beta(uu_{l-j}))$-equivalent to $\beta$. By Lemma~\ref{lem:not_changing_last_vertices}, for any $t\leq i$ $u_{l-t} = u'_{s-t}$, so in particular $u_{l-i} = u'_{s-i}$. In the coloring $\beta'$ the edge $uu_{s-i}$ is still colored $1$, and now the vertex $v$ is missing the color $2$. Since $P_{weak}(i)$ is true the fan $X^{\beta'}_v(1)$ is not a path, this is a contradiction.
\end{proof}

The next lemma considers $(\Vv,u)$-independent cycles.
\begin{lemma}\label{lem:P(i)_plus_independent_cycle_X_ui_not_path}
Let $i\geq 0$, $\Vv$ be a minimum cycle in a coloring $\beta$, $u$ and $u'$ two vertices of $\Vv$, $\Uu = X_u(m(u') = (uu_1,\cdots, uu_l)$, $h\leq i$ such that $\beta(uu_{l-h}) = m(v)$, $c'$ a color not in $\beta(\Vv)$ such that $\Yy = X_u(c') = (uy_1,\cdots,uy_r)$ is a $(\Vv,u)$-independent cycle different from $\Uu$ and $\Xx = X_v(c') = (vx_1,\cdots, vx_s)$ is a cycle different from $\Vv$ with $y_r = x_s = z$, and $c''$ a color in $\beta(\Vv)$. If $P(j)$ is true for all $j\leq i$, then $\Zz = X_z(c'')$ is not a path.
\end{lemma}

\begin{proof}
Without loss of generality, we assume that the vertices $v$, $u$ and $u'$ are respectively missing the colors $1$, $2$, and $3$; we also assume that $c' = 4$. Note that since $\Uu$ and $\Yy$ are different cycles, we have $\beta(\Uu)\cap \beta(\Yy) = \{m(u)\} = \{2$\}, and since $\Xx$ and $\Vv$ are different cycles, we have that $\beta(\Vv)\cap \beta(\Xx) = \{m(v)\} = 1$. Assume that the fan $\Zz$ is a path. The fan $\Vv$ is a minimum cycle, so by Lemma~\ref{lem:fans_around_Vv}, the fan $\Uu$ is a cycle entangled with $\Vv$, and thus $u_l = u'$.

We first invert $\Zz$ until we reach a coloring $\beta'$ where $m(z) = c\in (\beta(\Vv)\cup\beta(\Xx)\cup\beta(\Yy))\setminus \{4\}$. The coloring $\beta'$ is $\Vv$-equivalent to $\beta$, so by Observation~\ref{obs:Vv-minimum-stable}, the cycle $\Vv$ is the same minimum cycle in the coloring $\beta'$. The coloring $\beta'$ is also $\Uu$-equivalent to $\beta$, so, in the coloring $\beta'$, the fan $X_u(3)$ is exactly $\Uu$. Since the property $P(j)$ is true for all $j\leq h$, for any $j\leq h$ such that $\beta(uu_{l-j})\neq 1$, the fan $X_v(\beta(uu_{l-j}))$ is a saturated cycle containing $u_{l-j-1}$.

We first show that $c\not \in\beta(\Vv)$. Otherwise, assume that $c\in\beta(\Vv)$, then $c\not\in \beta(\Yy)$ since $\Yy$ is a $(\Vv,u)$-independent cycle, and $c\not\beta(\Xx)$ since $\Xx$ is different from $\Vv$. So the coloring $\beta'$ is $(\Xx\cup\Yy \setminus \{z\})$-equivalent to $\beta$. Hence, in the coloring $\beta'$, the fans $X_u(4)$ and $X_v(4)$ still contain the vertex $z$. If $c = 1$, then in the coloring $\beta'$, since the fan $X_u(4)$ still contains the vertex $z$, we have that $X_u(4)_{\leq z}$ is a $(\Vv,u)$-independent subfan avoiding $v$. However, the fan $X_v(4)$ is now a path containing $z$, by Lemma~\ref{lem:(Vv,u)-independent_subfan_avoiding_v}, we have a contradiction. So $c \not 1$. Since the fan $\Vv$ is a minimum cycle in the coloring $\beta'$, it is saturated by Lemma~\ref{lem:minimum_cycle_saturated}, thus $z\not\in K_v(1,c)$. We now swap the component $C_{1,c} = K_z(1,c)$, and denote by $\beta''$ the coloring obtained after the swap. The coloring $\beta''$ is $\Vv$-equivalent to $\beta'$, so $\Vv$ is still a minimum cycle in the coloring $\beta''$ by Observation~\ref{obs:Vv-minimum-stable}. The coloring $\beta''$ is also $(X_u(4)_{\leq z})\cup X_v(4))$-equivalent to $\beta'$, so the fan $X_u(4)$ still contains the vertex $z$ which is now missing the color $1$. Similarly to the previous case, the subfan $X_u(4)_{\leq z}$ is now a $(\Vv,u)$-independent subfan avoiding $v$, and $X_v(4)$ is  now a path; again by Lemma~\ref{lem:(Vv,u)-independent_subfan_avoiding_v}, we have a contradiction. Without loss of generality, we assume that $c = 5$.

\begin{case}[$5\not \in \beta(\Xx)$]
~\newline
In this case, the coloring $\beta'$ is $(\Xx\setminus \{z\})$-equivalent to $\beta$, and so in the coloring $\beta'$, the fan $X_v(4)$ still contains the vertex $z$ which is now missing the color $5$.
\end{case}

\begin{subcase}[$5\in\beta(\Uu_{<u_{l-h}})$]
~\newline
Let $z'$ be the vertex of $\Uu_{<u_{l-h}}$ missing the color $5$. If the vertex $z'$ does not belong to $K_v(1,5)$, then we swap the component $C_{1,5} = K_{z'}(1,5)$, and denote by $\beta''$ the coloring obtained after the swap. The coloring $\beta''$ is clearly $\Vv$-equivalent to $\beta'$, so by Observation~\ref{obs:Vv-minimum-stable}, the cycle $\Vv$ is still the same minimum cycle in the coloring $\beta''$. Since in the coloring $\beta'$, there is no edge colored $1$ or $5$ in $\Uu_{<u_{l-h}}$, the coloring $\beta''$ is also $\Uu_{<u_{l-h}}$-equivalent to $\beta'$. So in the coloring $\beta''$, the fan $X_u(3)$ still contains the vertex $z'$ which is now missing the color $1$, and thus $X_u(3)$ is not entangled with $\Vv$. By Lemma~\ref{lem:fans_around_Vv}, we have a contradiction. So the vertex $z'$ belongs to $K_v(1,5)$, and thus the vertex $z$ does not belong to $K_v(1,5)$. We now swap the component $C_{1,5} = K_{z}(1,5)$, and denote by $\beta''$ the coloring obtained after the swap. The coloring $\beta''$ is also $\Vv$-equivalent to $\beta'$, so the fan $\Vv$ is a minimum cycle in the coloring $\beta''$. Moreover, since $5\in\beta(\Uu_{<u_{l-h}})$, $5\not \in \beta(\Yy)$. We also have that $5\not\in \beta(\Xx)$, so in total, the coloring $\beta''$ is $(\Yy\cup\Xx\setminus \{z\})$-equivalent to the coloring $\beta'$. This means that in the coloring $\beta''$, the fan $X_u(4)$ still contains the vertex $z$ which is now missing the color $1$, so $X_u(4)_{\leq z}$ is a $(\Vv,u)$-independent subfan avoiding $v$. We also have that the fan $X_v(4)$ is now a path containing the vertex $z$, so by Lemma~\ref{lem:(Vv,u)-independent_subfan_avoiding_v} we have a contradiction.
\end{subcase}

\begin{subcase}[$5\in \beta(\Uu_{\geq u_{l-h}})$]
~\newline
Let $s$ be such that $m(u_{l-s}) = 5$. The fan $X_v(5)$ is a saturated cycle containing $u_{l-s}$, so the vertex $u_{l-s}$ belongs to the component $K_v(1,5)$, and the vertex $z$ does not belong to this component. We now swap the component $K_z(1,5)$ and denote by $\beta''$ the coloring obtained after the swap. Since the color $5$ is not in $\beta(X_v(4))$, the coloring $\beta''$ is $(X_v(4)\setminus\{z\})$-equivalent to $\beta'$. So in the coloring $\beta''$, the fan $X_v(4)$ still contains the vertex $z$ which is missing the color $1$, so it is now a path. Since the color $5$ is in $\beta(\Uu_{>u_{l-h}})$, it is not in $\beta(\Yy)$, so the coloring $\beta''$ is $(X_u(4)\setminus \{z\})$-equivalent to $\beta'$, and thus $X_u(4)$ still contains the vertex $z$. So the subfan $X_u(4)_{\leq z}$ is now a $(\Vv,u)$-independent subfan avoiding $v$. By Lemma~\ref{lem:(Vv,u)-independent_subfan_avoiding_v}, the fan $X_v(4)$ is a path not containing $z$; a contradiction.
\end{subcase}

\begin{subcase}[$5 \in \beta(\Yy)$]
~\newline
Let $z'$ be the vertex of $\Yy$ missing the color $5$ in the coloring $\beta$. The vertices $z$ and $z'$ are both missing the color $5$ in the coloring $\beta'$, so at least one of them is not in $K_v(1,5)$. If the vertex $z$ is not in $K_v(1,5)$, then we swap the component $C_{1,5} = K_z(1,5)$, and denote by $\beta''$ the coloring obtained after the swap. The coloring $\beta''$ is $\Vv$-equivalent to $\beta'$, so the cycle $\Vv$ is the same minimum cycle in the coloring $\beta''$ by Observation~\ref{obs:Vv-minimum-stable}. 

If the vertex $u$ does not belong to $C_{1,5}$, then the fan $X_u(5)$ now contains the vertex $z$ which is missing the color $1$. Thus $X_u(5)_{\leq z}$ is now a $(\Vv,u)$-independent subfan avoiding $v$, and by Lemma~\ref{lem:(Vv,u)-independent_subfan_avoiding_v}, the fan $X_v(5)$ is a path. However, $\beta''(uu_{l-h}) = \beta(uu_{l-h}) = 1$, and the property $P(j)$ is true for all $j\leq h$, so by Lemma~\ref{lem:P(i-1)_and_P_weak(i)_no_path}, there is no path around $v$. This is a contradiction.

So the vertex $u$ belogns to $C_{1,5}$, and now $X_u(1)$ contains the vertex $z$ which is missing the color $1$. So $X_u(1)$ is not entangled with $\Vv$, and by Lemma~\ref{lem:fans_around_Vv}, we also have a contradiction.
\end{subcase}

\begin{case}[$5\in\beta(\Xx)$]
Let $z'$ be the vertex of $\Xx$ missing the color $5$ in the coloring $\beta$.

\begin{subcase}[$5\in\beta(\Uu_{<u_{l-h}})$]
~\newline
Let $z''$ be the vertex of $\Uu_{<u_{l-h}}$ missing the color $5$ in the coloring $\beta$. Note that we may have $z'' = z'$. The vertices $z$ and $z''$ are both missing the color $5$ in the coloring $\beta'$, so they are not both part of $K_v(1,5)$. If $z''$ is not in $K_v(1,5)$, then we swap $K_{z''}(1,5)$, and denote by $\beta''$ the coloring obtained after the swap. The coloring $\beta''$ is $\Vv$-equivalent to $\beta'$, so by Observation~\ref{obs:Vv-minimum-stable}, the cycle $\Vv$ is the same minimum cycle in the coloring $\beta''$. Since there is no edge colored $1$ or $5$ in $E(\Uu_{\leq z''})$, the coloring $\beta''$ is also $\Uu_{\leq z''}$-equivalent to the coloring $\beta'$, so $X_u(3)$ still contains the vertex $z''$ which is now missing the color $1$, so it is not entangled with $\Vv$, by Lemma~\ref{lem:fans_around_Vv}, this is a contradiction.

So the vertex $z''$ belongs to $K_v(1,5)$, and thus the vertex $z$ does not belong to this component. We swap the component $C_{1,5} = K_z{1,5}$, and denote by $\beta''$ the coloring obtained after the swap. In the coloring $\beta''$, the fan $X_v(5)$ still contains the vertex $z$ which is missing the color $1$, so this fan is now a path. If the vertex $u$ does not belong to $C_{1,5}$, then $\beta''(uu_{l-h}) = 1$. Since the property $P(j)$ is true for all $j\leq h$, there is no path around $v$, a contradiction. So the vertex $u$ belongs to $C_{1,5}$. Now in the coloring $\beta''$, the fan $\Uu' = X_u(3) = (uu'_1,\cdots,uu'_{l'})$ is smaller, but for any $j\leq h$, we still have that $u_{l-j} = u'_{l'-j}\in V(\Uu')$. Note that we have $z'' = u'_{l'-h-1}$. So we have $\Uu' = (uu'_1,\cdots, uu'_{l-h-1} = uz'',uu_{l-h},\cdots, uu_l)$. Since the property $P(j)$ is true for all $j\leq h$, the fan $X_v(\beta(uu_{l-h})) = X_v(5)$ is a cycle, a contradiction.
\end{subcase}

\begin{subcase}[$5\in\beta(\Uu_{>u_{l-h}})$]
~\newline
Let $s$ be such that $m^{\beta}(u_{l-s}) = 5$. In the coloring $\beta$, since the property $P(s)$ is true, the fan $X_v(5)$ is a saturated cycle containing $u_s$. But the color $5$ is in $\Xx$, so the fan $X_v(5)$ also contains the vertex $z$, and thus $X_v(5) = \Xx$. In the coloring $\beta'$, the fan $X_v(5)$ still contains the vertex $z$ which is now missing the color $5$. Since the property $P(s)$ is true, the cycle $X_v(5)$ is a a cycle containing the vertex $u_{l-s}$, so we have a contradiction.
\end{subcase}

\begin{subcase}[$5\in\beta(\Yy)$]
~\newline
Let $z''$ be the vertex of $\Yy$ missing the color $5$ in the coloring $\beta$. Note that we may have $z'' = z'$. Since the vertices $z$ and $z''$ are both missing the color $5$ in the coloring $\beta'$, they are not both part of $K_v(1,5)$. If $z$ is not in $K_v(1,5)$, then we swap the component $K_z(1,5) = C_{1,5}$ and denote by $\beta''$ the coloring obtained after the swap. The coloring $\beta''$ is $\Vv$-equivalent to $\beta'$, so by Observation~\ref{obs:Vv-minimum-stable}, the cycle $\Vv$ is the same minimum cycle in the coloring $\beta''$. Moreover, the coloring $\beta''$ is also $(X_v(5)\setminus \{z\})$-equivalent to $\beta'$, so this fan is now a path containing $z$.

If the vertex $u$ does not belong to $C_{1,5}$, then the coloring $\beta''$ is $(X_u(5)\setminus\{z\})$-equivalent to $\beta'$, and thus $X_u(5)$ still contains the vertex $z$ which is now missing the color $1$. So the subfan $X_u(<)_{\leq z}$ is now a $(\Vv,u)$-independent subfan avoiding $v$, and it also contains $z$, by Lemma~\ref{lem:(Vv,u)-independent_subfan_avoiding_v}, the fan $X_v(5)$ is a path that does not contain $z$, a contradiction.

So the vertex $u$ belogns to $C_{1,5}$, and now in the coloring $\beta''$, the fan $X_u(1)$ contains the vertex $z$ which is missing the color $1$. So this fan is not entangled with $\Vv$, by Lemma~\ref{lem:fans_around_Vv}, we also have a contradiction.  
\end{subcase}

\begin{subcase}[$5 \not \in (\beta(\Yy)\cup \beta(\Uu))$]
~\newline
The vertices $z$ and $z'$ are both missing the color $5$, so at least one of them is not in $K_v(1,5)$. 

If the vertex $z$ is not in $K_v(1,5)$, since $P(j)$ is true for all $j\leq h$, then for all $j\leq h$, the vertex $z$ is not in $X_v(\beta(uu_{l-j}))$. We now swap the component $C_{1,5} = K_z(1,5)$ to obtain a coloring $\beta''$ where the subfan $X_u(4)$ is now a $(\Vv,u)$-independent subfan avoiding $v$. The coloring $\beta''$ is $\Vv$-equivalent to $\beta'$, so by Observation~\ref{obs:Vv-minimum-stable} the cycle $\Vv$ is the same minimum cycle in the coloring $\beta'$. By Lemma~\ref{lem:(Vv,u)-independent_subfan_avoiding_v}, the fan $X_v(4)$ is now a path that does not contain $z$. Since for all $j\leq h$, $P(j)$ is true and $z$ does not belong to $X_v(\beta(uu_{l-j})$, the coloring $\beta''$ is also $(\bigcup\limits_{j\in[0,h]}X_v(\beta(uu_{l-j})))$-equivalent to the coloring $\beta'$.

If the vertex $u$ does not belong to $C_{1,5}$, then the coloring $\beta''$ is also $\Uu$ equivalent to the coloring $\beta'$. The edge $uu_{l-h}$ is still colored $1$, and the property $P(j)$ is true for all $j\leq h$. By Lemma~\ref{lem:P(i-1)_and_P_weak(i)_no_path} there is no path around $v$, a contradiction. So the vertex $u$ belongs to the component $C_{1,5}$, and the edge $uu_{l-h}$ is now colored $5$. Let $\Uu' = X^{\beta''}_u(3) = (uu'_1,\cdots, uu'_{l'})$. Since $z\not\in X_v(\beta(uu_{l-j}))$ for all $j\leq h$, the coloring $\beta''$ is $(\bigcup\limits_{j\in[0,h]}X_v(\beta(uu_{l-j})))$-equivalent to the coloring $\beta'$, and it is also $\Uu_{>u_{l-h}}$-equivalent to the coloring $\beta'$, so by Lemma~\ref{lem:not_changing_last_vertices}, for all $j\leq h$, we have $u'_{l-j} = u_{l-j}$. In particular, $u'_{l-h} = u_{l-h}$. The coloring $\beta''$ is $X_u(5)_{<z}$-equivalent to the coloring $\beta'$, so in the coloring $\beta''$, the fan $X_u(5)$ still contains the vertex $z$ which is now missing the color $1$, therefore the fan $X_u(5)$ is a path. Since $P(h)$ is true, we have a contradiction.

So the vertex $z$ belongs to $K_v(1,5)$, and the vertex $z'$ does not belong to this component. We now swap the component $C_{1,5} = K_{z'}(1,5)$, and obtain a coloring $\beta''$ that is $\Vv$-equivalent to $\beta'$. By Observation~\ref{obs:Vv-minimum-stable}, the fan $\Vv$ is the same minimum cycle in the coloring $\beta''$. Since the property $P(j)$ is true for all $j\leq h$, and $z'$ is not in $K_v(1,5)$, then for all $j\leq h$, the vertex $z'$ is not in $X_v(\beta(uu_{l-j}))$, and the coloring $\beta''$ is $(\bigcup\limits_{j\in[0,h]}X_v(\beta(uu_{l-j})))$-equivalent to the coloring $\beta'$. In the coloring $\beta''$, the fan $X_v(4)$ is now a path, so similarly to the previous case, the vertex $u$ belongs to the component $C_{1,5}$. Therefore, in the coloring $\beta''$, the edge $uu_{l-h}$ is colored $5$. 

Let $\Uu' = (uu'_1,\cdots, uu'_{l'})$. Since the coloring $\beta''$ is $(\bigcup\limits_{j\in[0,h]}X_v(\beta(uu_{l-j})))$-equivalent to the coloring $\beta'$, by Lemma~\ref{lem:not_changing_last_vertices}, for any $j\leq h$ we have $u'_{l'-j} = u_{l-j}$. The coloring $\beta''$ is $\Uu_{\leq v}$ equivalent to $\beta'$, so $v$ is in $\Uu''$, and thus we have $X_u(1) = X_u(5) = \Uu'$. So there exists a vertex $z''$ missing the color $5$ in the fan $X_u(1)$. Note that since $m(z') = 1$ and $m(z'') = 5$ we have $z\neq z''$, however,  we may have $z'' = z$, and in this case there exists a vertex in $X_u(1)$ missing a color in $\beta(X_u(4)_{\leq z})$. We now have to distinguish the cases.

\begin{subsubcase}[$z''\neq z$]
~\newline
We consider the coloring $\beta'$. In this coloring, the vertex $z''$ is in $X_u(5)$ since $u$ is in $C_{1,5}$. If the vertex $z''$ also belongs to $C_{1,5}$, then now $X_u(5)_{\leq z''}$ is a subfan avoiding $v$. If there is en edge $uu''$ in $E(X_u(5)_{\leq z''})$ colored with a color in $\beta'(\Vv)$ then the fan $X_u(\beta'(uu'')))$ is not entangled with $\Vv$, and by Lemma~\ref{lem:fans_around_Vv} we have a contradiction. So the subfan $X_u(5)_{\leq z''}$ is a $(\Vv,u)$-independent subfan avoiding $v$. By Lemma~\ref{lem:(Vv,u)-independent_subfan_avoiding_v} the fan $X_v(5)$ is a path, however, in the coloring $\beta'$ the fan $X_v(5)$ still contains the vertex $z$ that is misisng the color $5$, so $X_v(5)$ is a cycle, a contradiction.

So the vertex $z''$ does not belong to $C_{1,5}$, and thus is still missing the color $5$ in the coloring $\beta'$. We now swap the component $K_{z''}(1,5)$ to obtain a coloring $\beta_f$ where $X_u(5)_{\leq z''}$ is a $(\Vv,u)$-independent subfan avoiding $v$, and where $X_u(5)$ is a cycle. Again by Lemma~\ref{lem:(Vv,u)-independent_subfan_avoiding_v} we have a contradiction.
\end{subsubcase}
\begin{subsubcase}[$z'' = z'$]
~\newline
So there exists a vertex $w$ in $X_u(5)$ such that $m(w)\in \beta' (X_u(4)_{\leq z})$ and $w\not \in V(X_u(4)_{\leq z})$. We now need to distinguish whether or not $m(w) = 4$.

\begin{subsubsubcase}[$m(w) \neq 4$]
~\newline
In this case, without loss of generality, assume that $m(w) = 6$, and let $w'$ be the vertex of $X_v(4)_{\leq z}$ missing the color $6$. The vertices $w$ and $w'$ are both misisng the color $6$, so they are not both part of $K_v(1,6)$.

If $w'$ is not in $K_v(1,6)$, then we swap $C_{1,6} = K_{w'}(1,6)$ to obtain a coloring $\beta''$ where $X_u(4)_{\leq w'}$ is a $(\Vv,u)$-independent subfan avoiding $v$. The coloring $\beta''$ is $\Vv$-equivalent to $\Vv$, so the fan $\Vv$ is still the same minimum cycle in the coloring $\beta'$ by Observation~\ref{obs:Vv-minimum-stable}. So by Lemma~\ref{lem:(Vv,u)-independent_subfan_avoiding_v} the fan $X_v(4)$ is a path that does not contain $w'$. If the vertex $u$ does not belong to $C_{1,6}$, then the coloring $\beta''$ is $\Uu$-equivalent to $\beta'$, the property $P(j)$ is true for all $j\leq h$ and $\beta''(uu_{l-h}) = 1$ so by Lemma~\ref{lem:P(i-1)_and_P_weak(i)_no_path} there is no path around $v$, a contradiction.

So the vertex $u$ belongs to $C_{1,6}$, and we have $\beta''(uu_{l-h}) = 6$. Let $\Uu' = X_u^{\beta''}(3) = (uu'_1,\cdots,uu'_{l'})$. The coloring $\beta''$ is $(\bigcup\limits_{j\in[0,h]}X_v(\beta(uu_{l-j})))$-equivalent to the coloring $\beta'$ and is $\Uu_{>u_{l-h}}$-equivalent to the coloring $\beta'$ so by Lemma~\ref{lem:not_changing_last_vertices}, for all $j\leq h$ we have $u'_{l'-j} = u_{l-j}$. In the coloring $\beta'$, the fan $X_v(4)$ contains the vertex $z'$ missing the color $5$, and the fan $X_v(5)$ is a cycle that contains the vertex $z$, and in the coloring $\beta''$, the fan $X_u(4)$ is a path. So there exists a vertex $w''$ in $X_v(4)$ that is missing the color $6$ in the coloring $\beta'$ and that belongs to $C_{1,6}$. this vertex is now missing the color $1$ in the coloring $\beta''$. If $w''$ is in $X^{\beta'}_v(4)_{\leq z'}$, then the fan $X_v^{\beta''}(6)$ is now a comet containing two vertices ($z$ and $z'$) missing the color $5$. Since the property $P(h)$ is true, we have a contradiction. So the vertex $w''$ si in $X^{\beta'}_v(5)$. But now, in the coloring $\beta''$, the fan $X_v(6)$ contains the vertex $z$ which is still missing the color $5$, and the fan $X_v(5)$ is a path, so the fan $X_v(6)$ is a path. Again since $P(h)$ is true, we have a contradiction. So the vertex $w'$ belongs to $K_v(1,6)$.

Therefore, the vertex $w$ does not belong to $K_v(1,6)$, and we swap the component $C_{1,6} = K_w(1,6)$ to obtain a coloring $\beta''$ where $X_u(5)_{\leq w}$ is a subfan avoiding $v$. The coloring $\beta''$ is $\Vv$-equivalent to the coloring $\beta'$, so by Observation~\ref{obs:Vv-minimum-stable}, the fan $\Vv$ is the same minimum cycle in the coloring $\beta''$. Similarly to the previous case, the subfan $X_u(5)_{\leq w}$ is a $(\Vv,u)$-independent subfan avoiding $v$, so $X_v(5)$ is a path, and the vertex $u$ belongs to $C_{1,6}$. Since the vertex $z$ is still missing the color $5$, it means that in the coloring $\beta''$ the fan $X_u(1)$ now contains the vertex $w$ which is missing the color $1$, and so it is not entangled with $\Vv$. By Lemma~\ref{lem:fans_around_Vv} we have a contradiction.
\end{subsubsubcase}

\begin{subsubsubcase}[$m(w) = 4$]
~\newline
Since the fan $\Yy = X_z(5)$ is a path in the coloring $\beta$, the fan $X_z(4)$ is a path in the coloring $\beta'$. In the coloring $\beta'$ we invert the path $X_z(5)$ until we reach a coloring where $m(z)\in \beta'(X_u(5)_{\leq w})$ and denote by $\beta''$ the coloring obtained after the inversion. Note that since $4\in\beta'(X_z(5))\cap \beta(X_u(5)_{\leq w})$ the inversion is well defined. The coloring $\beta''$ is $\Vv$-equivalent to $\beta'$ so the fan $\Vv$ is the same minimum cycle in the coloring $\beta''$. The coloring $\beta''$ is also $\Uu$-equivalent to $\beta'$, so $X_u^{\beta''}(3) = \Uu$ and since $P(j)$ is true for all $j\leq h$, the fan $X_v(\beta''(uu_{l-j}))$ is a saturated cycle if $\beta''(uu_{l-j})\neq 1$. The coloring $\beta''$ is $(X_v(4)\setminus \{z\})$-equivalent to $\beta'$, so the fan $X_v(4)$ still contains the vertex $z'$ which is missing $5$, and the vertex $z$. Finally, the coloring $\beta''$ is $X_u(5)_{\leq w}$-equivalent to the coloring $\beta'$. Let $c_z$ be the missing color of $z$ in $\beta''$, and let $w'$ be the vertex of $X_u(5)_{\leq w}$ missing the color $c_z$. Note that is $c_z = 4$, then we have $w' = w$.

The proof is similar to the previous case, and we now consider the components of $K(1,c_z)$. The vertices $z$ and $w'$ are both missing the color $c_z$ so at least of them is not in $K_v(1,c_z)$. If the vertex $z$ is not in $K_v(1,c_z)$, then we swap the component $C_{1,c_z} = K_z(1,c_z)$ to obtain a coloring $\beta_f$ that is $\Vv$-equivalent to $\beta''$. By Observation~\ref{obs:X-equivalent_plus_identical} the fan $\Vv$ is the same minimum cycle in the coloring $\beta_f$, and now $X_u(4)_{\leq z}$ is a $(\Vv,u)$-independent subfan avoiding $v$, so by Lemma~\ref{lem:(Vv,u)-independent_subfan_avoiding_v} the fan $X_v(4)$ is now a path not containing $z$. Note that the coloring $\beta_f$ is also $(\bigcup\limits_{j\in[0,h]}X_v(\beta(uu_{l-j})))$-equivalent to the coloring $\beta''$. If the vertex $u$ does not belong to $C_{1,c_z}$, then the coloring $\beta_f$ is $\Uu$-equivalent to $\beta''$, and in particular $\beta_f(uu_{l-h}) = 1$. Since $P(j)$ is true for all $j\leq h$, by Lemma~\ref{lem:P(i-1)_and_P_weak(i)_no_path} there is no path around $v$, a contradiction.

So the vertex $u$ belongs to $C_{1,c_z}$, and now $\beta_f(uu_{l-h}) = 6$. Since the fan $X_v(4)$ is now a path that does not contain $z$, it means that in the coloring $\beta''$ there is a vertex $w''$ in $X_v(4)$ which is missing the color $c_z$ and which also belongs to $C_{1,c_z}$. It means that in the coloring $\beta_f$, the fan $X_v(c_z)$ is now a path containing $z$. Let $\Uu' = X_u^{\beta_f}(3) = (uu'_1,\cdots, uu'_{l'})$. The coloring $\beta_f$ is $\Uu_{>u_{l-h}}$-equivalent to $\beta''$ and is also $(\bigcup\limits_{j\in[0,h]}X_v(\beta(uu_{l-j})))$-equivalent to $\beta''$, so for all $j\leq h$, we have $u'_{l'-j} = u_{l-j}$ by Lemma~\ref{lem:not_changing_last_vertices}. In particular $u'_{l-h} = u_{l-h}$. Since $P(h)$ is true, and $\beta_f(uu'_{l'-h}) = c_z$, the fan $X_v(c_z)$ is a cycle, a contradiction.

So the vertex $z$ belongs to $K_v(1,c_z)$ and the vertex $w'$ does not belong to this component. We now swap the component $C_{1,c_z} = K_{w}(1,c_z)$ and denote by $\beta_f$ the coloring obtained after the swap. The coloring $\beta_f$ is $\Vv$-equivalent to $\beta''$ so by Observation~\ref{obs:X-equivalent_plus_identical} the fan $\Vv$ is the same minimum cycle in the coloring $\beta_f$. The coloring $\beta_f$ is also $(\bigcup\limits_{j\in[0,h]}X_v(\beta(uu_{l-j})))$-equivalent to the coloring $\beta''$. In the coloring $\beta_f$ the subfan $X_u(5)_{\leq w'}$ is now a subfan avoiding $v$. If there is an edge $uu''$ in $E(X_u(5)_{\leq w'})$ colored with a color in $\beta_f(\Vv)$, then $X_u(\beta_f(uu''))$ is not entangled with $\Vv$ and by Lemma~\ref{lem:fans_around_Vv} we have a contradiction. So the subfan $X_u(5)_{\leq w'}$ is a $(\Vv,u)$-independent subfan avoiding $v$ and thus by Lemma~\ref{lem:(Vv,u)-independent_subfan_avoiding_v} the fan $X_v(5)$ is now a path that does not contain $w'$. Similarly to the previous case, this means that the vertex $u$ belongs to the component $C_{1,c_z}$, and thus that $\beta_f(uu_{l-h}) = c_z$. The fan $X_v(5)$ still contains the vertex $z$ which is missing the color $c_z$, and the fan $X_v(5)$ is a path, so the fan $X_v(c_z)$ is a path. Let $\Uu' = X^{\beta_f}_u(3) = (uu'_1,\cdots,uu'_{l'})$. Since $P(j)$ is true for all $j\leq h$ and since the coloring $\beta_f$ is $(\bigcup\limits_{j\in[0,h]}X_v(\beta(uu_{l-j})))$-equivalent to the coloring $\beta''$, by Lemma~\ref{lem:not_changing_last_vertices} for $j\leq h$, we have $u'_{l'-j} = u_{l-j}$. In particular, $u'_{l-h} = u_{l-h}$. The edge $uu'_{l'-h}$ is now colored $c_z$ and the property $P(h)$ is true, so the fan $X_v(c_z)$ is a cycle. This is a contradiction.
\end{subsubsubcase}
\end{subsubcase}

\end{subcase}

\end{case}

\end{proof}






Before proving the induction step of the proof we need to introduce a new property implied by $P(i)$.

\subsection{The property $Q(i)$}

\begin{definition}\label{def:Q(i)}
Let $i\geq 0$, we define the property $Q(i)$ as follows:

For any minimum cycle $\Vv$ in a coloring $\beta$, for any pair of vertices $u$ and $u'$ of $\Vv$, let $\Uu = X_u(m(u')) =(uu_1,\cdots,uu_l)$. If $\beta(uu_{l-i}) \neq m(v)$, then for any color $c\in\beta(\Vv)$, the fan $X_{u_{l-i-1}}(c)$ is a cycle entangled with $\Vv$ and $\Uu_{\geq u_{l-i-2}}$.
\end{definition}

We now prove that the property $Q(i)$ is implied by the property $P(i)$. And we first prove th following lemma concerning saturated cycles around the centrel vertex of a minimum cycle.

\begin{lemma}\label{lem:sat_cycle_centered_on_v_entangled_with_Uu}
Let $\Vv = (vv_1,\cdots,vv_k)$ be a minimum cycle in a coloring $\beta$, $u = v_j$ and $u' = v_{j'}$ two vertices of $\Vv$ and $\Ww = (vw_1,\cdots, vw_t)$ a saturated cycle around $v$. Then the fans $\Ww$ and $\Uu = X_u(m(u'))$ are entangled. 
\end{lemma}
\begin{proof}
    By Lemma~\ref{lem:fans_around_Vv}, the fan $\Uu$ is a cycle entangled with $\Vv$, so if $\Ww = \Vv$, the fans $\Ww$ and $\Uu$ are entangled as desired. So assume that $\Ww\neq \Vv$ and that $\Ww$ is not entangled with $\Uu$. Without loss of generality, we assume that the vertices $v$, $u$ and $u'$ are respectively missing the colors $1$, $2$, and $3$. Since $\Ww\neq \Vv$ and $\Ww$ is centered at $v$, we have that $\beta(\Ww)\cap\beta(\Vv) = \{1\}$. Since $\Ww$ and $\Uu$ are not entangled, there exists $c\in \beta(\Uu)\cap\beta(\Ww)$ such that $M(\Uu,c)\neq M(\Ww,c)$. Without loss of generality, since $c\not\in\{1,2,3\}$, we assume that $c= 4$ and that $u_i = M(\Uu,4)$ is the first such vertex in $\Uu$; up to shifting the indices in $\Ww$, we also assume that $m(w_t) = 4$, and thus that $\Ww = X_v(4)$.

    Since the cycle $\Ww$ is saturated, the vertex $w_t$ belongs to $K_v(1,4)$, so the vertex $z$ does not belong to $K_v(1,4)$. We swap the component $C_{1,4} = K_z(1,4)$ and denote by $\beta_2$ the coloring obtained after the swap. 
    
    If $u \not \in C_{1,4}$, or there is no edge colored $1$ in $\Uu_{<i}$, then the coloring $\beta_2$ is $(\Vv\cup\Ww\cup\Uu_{<i})$-equivalent to $\beta$. Hence, in the coloring $\beta_2$, the fan $\Vv$ is a minimum cycle by Observation~\ref{obs:Vv-minimum-stable}, but now the fan $X_u(m(u') = (uu_1,\cdot,uu_i)$ is now a path, by Lemma~\ref{lem:fans_around_Vv}, this is a contradiction.

    So $u \in C_{1,4}$, and there is an edge colored $1$ in $\Uu_{<i}$. Since by Lemma~\ref{lem:fans_around_Vv}, the cycle $\Uu$ is entangled with $\Vv$, the edge $uv_{j-1}$ and the edge $uv$ are in $E(\Uu)$. We denote by $x$ the vertex connected to $u$ y the edge colored $1$ and by $c_{j-1}$ the missing color of $v_{j-1}$ in $\beta$. Note that we may have $v_{j-1} = u'$, and thus $c_{j-1} = 2$. The fan $\Uu$ is of the form $(uu_1,\cdots,uv_{j-1},uv,ux,\cdots,uu_i,\cdots, uu')$. The coloring $\beta_2$ is $(\Vv)$- equivalent to $\beta$, so by Observation~\ref{obs:Vv-minimum-stable}, the cycle $\Vv$ is a minimum cycle in $\beta_2$. But now the fan $X_u(4)$ is a comet where $v$ and $u_i$ are missing the same color $1$, more precisely, $X_u(4) = (ux,\cdots,uu_i,\cdots, uu',uu_1,\cdots, uv_{j-1},uv)$. Note that $X_u(3)$ is a cycle which is a subsequence of $X_u(4)$.  If there is an edge colored with a color $c\in\beta(\Vv)$ in $X_u(4)$ between the edges $ux$ and $uu_i$, then the fan $X_v(c)$ is a comet, which is a contradiction by Lemma~\ref{lem:fans_around_Vv}.

    So there is no edge colored with a color $c\in\beta(\Vv)$ in $X_u(4)$ between the edges $ux$ and $uu_i$. Since the fan $\Vv$ is a minimum cycle, it is saturated by Lemma~\ref{lem:minimum_cycle_saturated}, so $u\in K_v(1,2)$, and thus $u_i\not\in K_v(1,2)$. We now swap the component $K_{u_i}(1,2)$ to obtain a coloring $\beta_3$. The coloring $\beta_3$ is $(\Vv\cup\Ww)$-equivalent to $\beta_2$, so the fan $\Vv$ is a minimum cycle in $\beta_3$ by Observation~\ref{obs:Vv-minimum-stable}.

    We now show that $\Vv$ is invertible in the coloring $\beta_3$. The cycle $\Vv$ is tight by Observation~\ref{obs:tight}, so the vertex $u$ belongs to the component $C_{2,j-1} = K_{v_{j-1}}(2,c_{j-1})$, thus the edges $vu$ and $vv_{j+1}$ also belong to $C_{2,j-1}$. In the coloring $\beta_3$, the fan $X_u(4)$ is now a path that we invert until we reach a coloring $\beta_4$ where $m(u)\in\beta(\Ww)$. Note that since $4\in\beta(\Ww)$, the inversion is well-defined and moreover, since $\beta_3$ is also $(\Ww)$-equivalent to $\beta$, we have $\beta_3(\Ww) = \beta(\Ww)$. Since $u\not\in\Ww$, by Observation~\ref{obs:not_touching_subsequence_stable}, the coloring $\beta_4$ is $(\Ww)$-equivalent to $\beta_3$, so $\Ww$ is still the same cycle in $\beta_4$. Moreover, since $\beta_3(X_u(4))\cap\beta_3(\Vv) = \{2\}$, the coloring $\beta_4$ is $(\Vv\setminus \{u\}\cup C_{2,j-1})$-equivalent to $\beta_3$.

    We denote by $w_s$ the vertex of $\Ww$ such that $m^{\beta_4}(u) = m^{\beta_4}(w_s)$, and we denote by $c_s$ this missing color. Note that we may have that $w_t = w_s$, and thus $c_s = 4$. The vertices $u$ and $w_s$ are missing the same color $c_s$, so they are not both part of the component $K_v(1,c_s)$ and we now have to distinguish the cases.

     \begin{case}[$u\not\in K_v(1,c_s)$]
    In this case, we swap the component $C_{1,c_s} = K_{u}(1,c_s)$ and obtain a coloring that we denote by $\beta_5$. Since $\{1,c_s,2,c_{j-1}\} = 4$, the coloring $\beta_5$ is $(C_{2,j-1})$-equivalent to $\beta_4$, so it is $(C_{2,j-1})$-equivalent to $\beta_3$. In the coloring $\beta_5$, the vertex $u$ is now missing the color $1$, so the fan $X_v(m(u)) = (vv_{j+1},\cdots, vv_{j-1},vu)$ is now a path that we invert, we denote by $\beta_6$ the coloring obtained after the inversion. In the coloring $\beta_6$, the vertices $v_{j+1}$ and $v$ are missing the color $2$, and the vertex $u$ is missing the color $c_{j-1}$. So now the component $C'_{2,j-1} = K_{v_{j'-1}}$ is exactly $C_{2,j-1}\cup\{vv_{j-1}\}\setminus \{vu,vv_{j+1}\}$ and we swap it. After this swap, the vertices $v$ and $u$ are missing the same color $c_{j-1}$, and the edge $uv$ is colroed $1$; we swap this edge and we denote by $\beta_7$ the coloring obtained after the swap. In the coloring $\beta_7$, the vertex $u$ is missing the color $1$, so the component $K_{u}(1,c_s)$ is now exactly $C_{1,c_s}$, so we swap back this component. Note that since $\{1,2,c_s,c_{j-1}\} = 4$, we can swap back $C_{1,c_s}$ before $C'_{2,j-1}$. In the coloring obtained after the swap, the fan $X_{u}(2)$ is now a path that we invert, and we denote by $\beta_8$ the coloring obtained after the inversion. In the coloring $\beta_8$, the vertex $u$ is now missing the color $2$, so the component $K_{v_{j-1}}(2,c_{j-1})$ is now exactly $C'_{2,j-1}\cup\{uv\}$. After swapping back this component we obtain exactly $\Vv^{-1}(\beta_3)$, a contradiction.   
    \end{case}

    \begin{case}[$u\in K_v(1,c_s)$]
        The principle in the same as in the previous case, but instead of changing the missing color of $u$, we will change the missing of of $v$ using the fan $X_v(c_s)$ to transform $\Vv$ into a path. As $u$ belongs to $K_v(1,c_s)$, the vertex $w_s$ does not belong to this component. So we swap the component $C_{1,c_s} = K_{w_s}(1,c_s)$ to obtain a coloring where $X_v(c_s)$ is now a path that we invert; we denote by $\beta_5$ the coloring obtained after the inversion. Note that since $X_v(c_s)$ was a cycle in $\beta_4$, we have $\beta_4(X_v(c_s))\cap\beta_4(\Vv) = \{1\}$,and so $\{2,c_{j-1}\}\cap\beta_4(X_v(c_s)) = \emptyset$. Hence the coloring $\beta_5$ is $(C_{2,j-1})$-equivalent to the coloring $\beta_4$. In the coloring $\beta_5$, the fan $X_v(2) = (vv_{j+1},\cdots, vu)$ is now a path that we invert, and we denote by $\beta_6$ the coloring obtained after the swap. Similarly to the previous case, in the coloring $\beta_6$, the vertices $v$ and $v_{j+1}$ are missing the color $2$, and the vertex $u$ is missing the color $c_{j-1}$. So in the coloring $\beta_6$, the component $C'_{2,j-1} = K_{v_{j-1}}(2,c_{j-1})$ is exactly $C_{2,j-1}\cup\{vv_{j-1}\}\setminus \{vv_{j'+1},vu\}$, and we swap it to obtain a coloring where the vertices $u$ and $v$ are missing the color $c_{j-1}$ and where the edge $uv$ is colored $c_s$. After swapping the edge $uv$, we obtain a coloring where, the fan $X_v(1)$ is now a path that we invert, we denote by $\beta_7$ the coloring obtained after the inversion. In the coloring $\beta_7$, the component $K_{w_s}(1,c_s)$ is exactly $C_{1,c_s}$ and we swap back this component. Note that since $|\{1,2,c_s,c_{j-1}\}| = 4$, we can swap back this component before $C'_{2,j-1}$. In the coloring obtained after the swap, the fan $X_{u}(2)$ is now a path that we invert, we denote by $\beta_8$ the coloring obtained after the swap. In the coloring $\beta_8$, the vertex $u$ is now missing the color $2$, so the component $K_{v_{j-1}}(2,c_{j-1})$ is now exactly $C'_{2,j-1}\cup \{vu\}$ and we swap back this component to obtain $\Vv^{-1}(\beta_3)$ as desired.
    \end{case}
\end{proof}


\begin{lemma}\label{lem:P(i)_implies_Q(i)}
Let $i\geq 0$, if $P(i)$ is true for all $j\leq i$, then $Q(j)$ is true for all $j\leq i$.
\end{lemma}
\begin{proof}
Let $i\geq 0$, $\Vv$ be a minimum cycle in a coloring $\beta$, $u$ and $u'$ two vertices of $\Vv$, $\Uu = X_u(m(u')) = (uu_1,\cdots,uu_l)$, and assume that $P(j)$ is true for all $j\leq i$. Without loss of generality, we assume that the vertices $v$, $u$ and $u'$ are respectively missing the colors $1$, $2$, and $3$. Let $t\leq i$, and $z =u_{l-t-1}$. We prove that $Q(t)$ is true.

\begin{claim}\label{cl:m(z)_not_in_beta(Vv)}
The vertex $z$ is not missing a color in $\beta(\Vv)$.
\end{claim}
\begin{proof}
Otherwise, assume that $m(z)\in\beta(\Vv)$. The fan $\Vv$ is a minimum cycle in $\beta$ so by Lemma~\ref{lem:fans_around_Vv}, then fan $\Uu$ is a cycle entangled with $\Vv$.

If $m(z)\neq 1$, then since $\Uu$ is a cycle entangled with $\Vv$ by Lemma~\ref{lem:fans_around_Vv}, we have $z\in V(\Vv)$ so by Lemma~\ref{lem:fans_around_Vv} for any color $c\in \beta(\Vv)$, $X_z(c)$ is a cycle entangled with $\Vv$. Moreover, since the property $P(t)$ is true, so $X_v(\beta(uz)) = (vw_1,\cdots,w_x)$ is a saturated cycle, and by Lemma~\ref{lem:sat_cycle_centered_on_v_entangled_with_Uu} is is entangled with $\Uu = X_u(m(u'))$ and $X_z(m(u)) = (zz_1,\cdots,zz_r)$, and thus $u_{l-t-2} = w_x = z_{r-1}$, so $Q(t)$ is true.

If $m(z) = 1$, then since $\Uu$ is entangled with $\Vv$, we have $z = v$. So for any $c\in\beta(\Vv)$, $X_z(c) = \Vv$ and thus is a cycle entangled with $\Vv$. Moreover, this means that $\beta(uz) = \beta(uv)$ and thus $m(u_{l-t-2}) = \beta(uv)$, so $u_{l-t-2}$ is the vertex just before $u$ in the cycle $\Vv$. By definition of $\Vv$, the fan $X_z(m(u)$ contains this vertex, and thus $Q(t)$ is true. In both cases, we have a contradiction.
\end{proof}

\begin{claim}\label{cl:beta(Uu_>z)_cap_beta(Vv)_is_empty}
There is no edge in $E(\Uu_{>z})$ colored with a color $beta(\Vv)$.
\end{claim}
\begin{proof}
We first prove that there is no vertex in $V(\Uu_{>z})\setminus \{u'\}$ missing a color $c \in\beta(\Vv)$. Otherwise, assume that there exists such a vertex $z'$. The cycle $\Vv$ is minimum in $\beta(\Vv)$, so by Lemma~\ref{lem:fans_around_Vv}, the fan $\Uu$ is entangled with $\Vv$. If $c\neq 1$, then $z'\in V(\Vv)$. By Lemma~\ref{lem:fans_around_Vv}, the fan $\Uu' = X_u(m(z')) = (uu'_1,\cdots,uu'_{l'})$ is a cycle entangled with $\Vv$, so $u'_{l'} = z'$ and $V(\Uu) = V(\Uu')$. Thus there exists $t'<t$ such that $z = u{l'-t'-1}$. Since $t$ is minimum, $Q(t')$ is true, and thus $Q(t)$ is true.

If $c = 1$, then $z' = v$ since $\Uu$ is entangled with $\Vv$, and $\beta(uz') = \beta(uv)$. Let $z''$ be the vertex just before $z'$ in $\Uu$. Since $\beta(uz') = \beta(uv)$, then $m(z'') = \beta(uv)\in\beta(\Vv)$. Since $m(z)\not\in\beta(\Vv)$, we have that $z''\neq z$. This means that $z''$ is a vertex in $V(\Uu_{>z})\setminus \{u'\}$ missing a color in $\beta(\Vv)$, this is a contradiction.
\end{proof}

Let $c\in\beta(\Vv)$, we prove that $\Zz = X_z(c) = (zz_1,\cdots,zz_r)$ is a cycle entangled with $\Vv$ and $\Uu_{\geq z}$.

By Claim~\ref{cl:m(z)_not_in_beta(Vv)} $m(z)\not\in\beta(\Vv)$, so without loss of generality, we assume that $z$ is missing the color $4$. By Lemma~\ref{lem:X_u_i(beta_(Vv)_cup_beta(Uu_<u_i))_not_a_path} the fan $\Zz$ is not a path. Before proving that $\Zz$ is not a comet, we first prove that is it entangled with $\Vv$ and $\Uu_{\geq u_{l-t-2}}$.

\begin{proposition}\label{prop:Zz_entangled_with_Vv_and_Uu_>i}
The fan $\Zz$ is entangled with $\Vv$ and $\Uu_{\geq u_{l-t-2}}$.
\end{proposition}
\begin{proof}
Otherwise, assume that there exists $s$ such that $m(z_s)\in\beta(\Vv)\cup \beta(\Uu_{\geq z})$ and $z_s\not\in V(\Vv)\cup V(\Uu_{\geq z})$. Without loss of generality, we assume that such an $s$ is minimum. We also assume that there is no edge colored with a color in $\beta(\Vv)$ in $E(\Zz_{[z_2,z_{s-1}]})$. Otherwise, if such an edge $zz_x$ exists, is suffices to consider the fan $X_z(\beta(zz_x)) = (zz_x,\cdots,zz_s)$. We now have to distinguish the cases.
\begin{case}[$m(z_s) = 1$]
~\newline
In this case, since, $P(t)$ is true, $X_v(4)$ is a saturated cycle containing $z$, so $v\in K_z(1,4)$, and thus $z_s\not\in K_z(1,4)$. We now swap the component $C_{1,4} = K_{z_s}(1,4)$, and denote by $\beta'$ the coloring obtained after the swap. In the coloring $\beta'$, the fan $X_z(c)$ is now a path. The coloring $\beta'$ is $\Vv$-equivalent to $\beta$, so $\Vv$ is still a minimum cycle in $\beta'$. If the coloring $\beta'$ is also $\Uu_{\leq z}$-equivalent to $\beta$ (\textit{i.e.}, $C_{1,4}$ does not contain $u$ or there is no edge colored $1$ in $\Uu_{\leq z}$), then $z$ is still a vertex of $\Uu = X_u(3)$, and the fan $X_z(c)$ is now a path, by Lemma~\ref{lem:X_u_i(beta_(Vv)_cup_beta(Uu_<u_i))_not_a_path} this is a contradiction. So the vertex $u$ belongs to $C_{1,4}$, and there is an edge $uu_h$ colored $1$ in $\Uu_{<z}$. So in the coloring $\beta'$, the edge $uu_h$ is now colored $4$, and the edge $u_{l-t}$ is now colored $1$. The fan $X_v(4)$ is still a saturated cycle containing $z$, but now the fan $X_u(4)$ is also a cycle containing $z$. In this coloring the fan $X_z(c)$ is a path, so by Lemma~\ref{lem:P(i)_plus_independent_cycle_X_ui_not_path}, we have a contradiction.
\end{case}

\begin{case}[$m(z_s) = c' \in \beta(\Vv)\setminus \{1\}$]
~\newline
In this case, since $\Vv$ is minimum, it is saturated by Lemma~\ref{lem:minimum_cycle_saturated}, thus $z_s \not \in K_v(1,c')$. We now swap the component $C_{1,c'} = K_{z_s}(1,c')$, and denote by $\beta'$ the coloring obtained after the swap. This coloring is $\Vv$-equivalent to $\beta$, so $\Vv$ is still a minimum cycle in the coloring $\beta'$. By Claim~\ref{cl:beta(Uu_>z)_cap_beta(Vv)_is_empty}, there is no edge with a color in $\beta(\Vv)$ in $\Uu_{>z}$, so $\beta'$ is $\Uu_{>z}$-equivalent to $\beta$. Moreover, let $\Uu' = X^{\beta'}_u(3) = (uu'_1,\cdots, uu'_{l'})$; the coloring $\beta'$ is also $(\bigcup\limits_{j \in [0,t]}X_v(\beta(uu_j)))$-equivalent to $\beta$ since each of these fans are saturated cycles, and the vertex $z_s$ does not belong to any of them. So by Lemma~\ref{lem:not_changing_last_vertices}, in the coloring $\beta'$, for any $j\leq (t+1)$, $u'_{l'-j} = u_{l-j}$. In particular, $u'_{l'-t-1} = u_{l-t-1} = z$. But now $X_z(c)$ is not entangled with $\{v\}$ since it contains the vertex $z_s$ which is missing the color $1$. This case is similar to the previous one.
\end{case}

\begin{case}[$m(z_s) = c' \in \beta(\Uu_{>z})$]
~\newline
Let $u_{l-h}$ be the vertex of $\Uu_{>z}$ missing the color $c'$. In this case, since $P(j)$ is true for all $j\leq t$, the fan $X_v(\beta(uu_{l-j}))$ is a saturated cycle. In particular, the vertex $u_{l-h}$ belongs to the component $K_v(1,c')$, and so $z_s$ does not belong to this component. We now swap the component $C_{1,c'} = K_{z_s}(1,c')$, and denote by $\beta'$ the coloring obtained after the swap. Let $\Uu' = X^{\beta'}_u(3) = (uu'_1,\cdots,uu'_{l'})$. If the coloring $\beta'$ is $\Uu_{>z}$-equivalent to $\beta$, then for the same reason as in the previous case, $z$ is exactly the vertex $u'_{l'-t-1}$, and $X_z(c')$ is now not entangled with $\{v\}$ since it contains the vertex $z_s$ that is missing the color $1$. This case is similar to the first one. So $\beta'$ is not $\Uu_{>z}$-equivalent to $\beta$, and thus since it is $\{u_{l-h}\}$-equivalent to $\beta$, the component $C_{1,c'}$ contains the vertex $u$. We now have to distinguish whether or not, in the coloring $\beta$ there is an edge $uu_p$ colored $1$ in $\Uu_{<z}$.

\begin{subcase}[There an edge $uu_p$ colored $1$ in $\Uu_{<z}$]
~\newline
In this case, in the coloring $\beta'$, the edge $uu_p$ is now colored $c'$, and the edge $uu_{l-h+1}$ is now colored $1$. In the coloring $\beta'$, the fan $X_u(4)$ is now a cycle since it contains the vertex $u_{l-h}$ which is still missing the color $c'$, and $X_v(c')$ now contains the vertex $z$ which is still missing the color $4$. The fan $X_v(4)$ is still a cycle containing also the vertex $z$, and the fan $\Uu'$ now contains an edge $uu_{l-p}$ colored $1$ such that $p\leq t$. 

We now consider the components of $K(1,4)$. If the vertex $z$ does not belong to the component $K_v(1,4)$, then we swap it to obtain a coloring $\beta''$ where $X_v(4)$ is now a path. Let $\Uu'' = X^{\beta''}_u(3) = (uu''_1,\cdots,uu''_{l''})$. If $u\not \in K_v(1,4)$, then $\beta''(uu_{l-p}) = \beta''(uu''_{l''-t}) = 1$, but $p \leq t$, and $P(j)$ is true for all $j\leq t$, so by Lemma~\ref{lem:P(i-1)_and_P_weak(i)_no_path} we have a contradiction. Similarly, if $u\in K_v(1,4)$, then now $\beta''(uu_{l-t}) = 1$. Since $\beta''$ is $\bigcup\limits_{j \in [0,t-1]}X_v(\beta(uu_j))$-equivalent to $\beta'$, by Lemma~\ref{lem:not_changing_last_vertices}, for any $j\leq t$, $u''_{l''-j} = u'_{l'-t}$. So the edge $uu_{l-t}$ is exactly the edge $uu''_{l''-t}$. This edge is colored $1$, and $P(j)$ is true for all $j\leq t$, so by Lemma~\ref{lem:P(i-1)_and_P_weak(i)_no_path}, we have a contradiction. 

So the vertex $z$ belongs to $K_v(1,4)$, and therefore the vertex $z_s$ does not belong to this component. We now swap the component $C_{1,4} = K_{z_s}(1,4)$, and denote by $\beta''$ the coloring obtained after the swap.  Let $\Uu'' = X^{\beta''}_u(3) = (uu''_1,\cdots,uu''_{l''})$. Whether or not the vertex $u$ belongs to the component $C_{1,4}$, the fan $X_u(3)$ contains an edge $uu''_{l''-j}$ colored $1$ where $j\leq t$ (if $u$ belongs to the component, $\beta''(uu''{l''-t}) = 1$, and $\beta''(uu''_{l''-p}) = 1$). Moreover, we have that the fan $X_u(4)$ is a cycle containing $z$, the fan $X_v(4)$ is a cycle containing $z$, the fan $X_z(c)$ is a path, and , and the property $P(j)$ is true for all $j\leq t$, so by Lemma~\ref{lem:P(i)_plus_independent_cycle_X_ui_not_path}, we have a contradiction.
\end{subcase}

\begin{subcase}[There is no edge colored $1$ in $\Uu_{<z}$]
~\newline
In this case, the coloring $\beta'$ is $\Uu_{\leq z}$-equivalent to $\beta$. We now consider the components of $K(1,4)$. If $z$ does not belong to $K_v(1,4)$, then we swap the component $K_z(1,4)$ and obtain a coloring where $X_u(3)$ still contains the vertex $z$ which is now missing the color $1$. In the coloring, the cycle $\Vv$ is still a minimum cycle since $\beta'$ is $\Vv$-equivalent to $\beta$, so by Lemma~\ref{lem:fans_around_Vv}, we have a contradiction.

So the vertex $z$ belongs to $K_v(1,4)$, and thus $z_s$ does not belong to this component. We now swap the component $C_{1,4} = K_{z_s}(1,4)$, and denote by $\beta''$ the coloring obtained after the swap. The coloring $\beta''$ is $\Uu_{\leq z}$-equivalent to $\beta'$, so $z\in X_u(3)$. However, now the fan $X_z(c)$ is a path, by Lemma~\ref{lem:X_u_i(beta_(Vv)_cup_beta(Uu_<u_i))_not_a_path}, we have a contradiction.
\end{subcase}
\end{case}

\begin{case}[$m(z_s) = c' = m(u_{l-t-2})$]
~\newline
In this case, since $c'\not\in \beta(\Vv)$, without loss of generality, we assume that $c' = 5$. We now consider the components of $K(1,5)$. If $u_{l-t-2}$ does not belong to $K_{z}(4,5)$, then we swap the component $C_{4,5} = K_{u_{l-t-2}}(4,5)$, and denote by $\beta'$ the coloring obtained after the swap. Let $\Uu' = X^{\beta'}(3) = (uu'_1,\cdots,uu'_{l'})$. The coloring $\beta'$ is $(\Vv\cup\Uu_{>z})$-equivalent to $\beta$, and for any $j\leq t$, $u'_{l' - j} = u_{l-j}$, and $u'_{l'-j} = u_{l-j-1}$ otherwise. Note that this means that $l' = l-1$, \textit{i.e.} $|\Uu'| = |\Uu| -1$. If the color $5$ is not in $X^{\beta}_v(4)$, then $X_v(4)$ is still a cycle containing $z$, and thus it does not contain $u_{l-t-2} = M(X_u(3),4)$, since the property $P(t)$ is true, we have a contradiction.

So the color $5$ is in $X^{\beta}_v(4)$. If $v$ belongs to $C_{4,5}$, then we are a in case similar to the previous one where $X^{\beta'}_v(4)$ is a cycle containing $z$, and thus which does not contain $u_{l-t-2} = M(X_u(3),4)$. Since $P(t)$ is true, we have a contradiction. So $v$ does not belong to $C_{4,5}$, and now, in the coloring $\beta'$, the fan $X_v(5)$ is a comet containing the vertices $z$ and $u_{l-t-2}$ that are both missing the color $4$. We now consider the components of $K(1,4)$. Since the property $P(t)$ is true, $X_v(4)$ is a saturated cycle, to $u_{l-t-2}$ belongs to $K_v(1,4)$, and thus $z$ does not belong to this component. We now swap the component $K_z(1,4)$, and obtain a coloring where $\{uz\}$ is a $(\Vv,u)$-independent subfan avoiding $v$, and where $X_v(5)$ is a path containing $z$, by Lemma~\ref{lem:(Vv,u)-independent_subfan_avoiding_v} we have a contradiction.

So the vertex $u_{l-t-2}$ belongs to the component $K_z(4,5)$, and therefore, the vertex $x_s$ does not belong to this component. We now swap the component $K_{z_s}(4,5)$, to obtain a coloring $(\Vv\cup \Uu)$-equivalent to $\beta$, where $X_z(c)$ is now a path, by Lemma~\ref{lem:X_u_i(beta_(Vv)_cup_beta(Uu_<u_i))_not_a_path}, we again get a contradiction.
\end{case}
\end{proof}

So the fan $\Zz$ is entangled with $\Vv$ and $\Uu_{\leq u_{l-t-2}}$. We now prove that it is not a comet. Assume that $\Zz$ is a comet, then there exists $h<r$ such that $m(z_h) = m(z_r)= c$. By Proposition~\ref{prop:Zz_entangled_with_Vv_and_Uu_>i}, $c\not\in \beta(\Vv)\cup \beta(\Uu_{\geq u_{l-2-t}}$. Without loss of generality, we assume that $c = 5$, and we now consider the components of $K(1,5)$. The vertices $z_h$ and $z_r$ are not both part of $K_v(1,5)$. 

If $z_h$ does not belong to $K_v(1,5)$, then we swap $C_{1,5} = K_v(1,5)$, and denote by $\beta'$ the coloring obtained after the swap. Let $\Uu' = X^{\beta'}_u(3) = (uu'_1,\cdots, uu'_{l'})$. The property $P(j)$ is true for all $j\leq t$, so the coloring $\beta'$ is $(\bigcup\limits_{j \in [0,p]} X_v(\beta(uu_{l-j})))$-equivalent to $\beta$ since each of these fans are saturated cycle. Hence by Lemma~\ref{lem:not_changing_last_vertices}, for any $j\leq (t+1)$, $u'_{l'-j} = u_{l-j}$. In particular, $z = u'_{l'-t-1}$. If the vertex $z$ does not belong to $C_{1,5}$ or $c\neq 1$, then the coloring $\beta'$ is $\Zz_{<z_h}$ equivalent to $\beta$. The fan $X_z(c)$ now contains the vertex $z_h$ which is missing the color $1$, by Proposition~\ref{prop:Zz_entangled_with_Vv_and_Uu_>i} we have a contradiction. So the vertex $z$ belongs to $C_{1,5}$, and $c = 1$. Thus, in the coloring $\beta'$, the edge $zz_{1}$ is now colored $5$, and the edge $zz_{s+1}$ is now colored $1$. If the vertex $z_r$ belongs to the component $C_{1,5}$, it is now missing the color $1$ in the coloring $\beta'$, and $X_z(1)$ is now a fan that contains this vertex. So the fan $X_v(1)$ is not entangled with $\Vv$, a contradiction by Proposition~\ref{prop:Zz_entangled_with_Vv_and_Uu_>i}. If the vertex $z_r$ does not belong to the component, then the fan $X_z(1)$ now contains the vertex $z_s$ which is missing the color $1$, again, a contradiction by Proposition~\ref{prop:Zz_entangled_with_Vv_and_Uu_>i}.

So the vertex $z_h$ belongs to $K_v(1,5)$, and thus the vertex $z_r$ does not belong to the component. We now swap the component $C_{1,5} = K_{z_s}(1,5)$ and denote by $\beta'$ the coloring obtained after the swap. Similarly to the previous case, If the vertex $z$ does not belong to $C_{1,5}$, or if $c\neq 1$, then the coloring $\beta'$ is $\Zz_{<z_r}$-equivalent to the coloring $\beta$, and now $X_z(c)$ contains the vertex $z_r$ missing the color $1$, by Proposition~\ref{prop:Zz_entangled_with_Vv_and_Uu_>i} this is a contradiction. So the vertex $z$ belongs to $C_{1,5}$, and $c = 1$. In this case, the fan $X_z(1)$ stills contains the vertex $z_r$ which is missing the color $1$. Again by Proposition~\ref{prop:Zz_entangled_with_Vv_and_Uu_>i}, this is a contradiction.

Therefore, the fan $\Zz$ is a cycle entangled with $\Vv$ and $\Uu_{\geq u_{l-t-2}}$ and thus $Q(t)$ is true as desired.
\end{proof}

We are now ready to prove that $P(i)$ is true for all $i$.

\subsection{Proof of $P(i)$}

\begin{proof}[Proof of Lemma~\ref{lem:P(i)_is_true}]
Let $i\geq 0$, $\Vv$ be a minimum cycle in a coloring $\beta$, $u$ and $u'$ two vertices of $\Vv$, $\Uu = X_u(m(u')) = (uu_1,\cdots,uu_l)$, and assume that $P(i)$ not verified. Without loss of generality, we assume that $i$ is minimum, ans that the vertices $v$, $u$ and $u'$ are respectively missing the colors $1$, $2$ and $3$. By Lemma~\ref{lem:P(0)}, the property $P(0)$ is true, so $i>0$. Assume that $\beta(uu_{l-i})\neq 1$ and let $\Xx = X_v(\beta(uu_{l-i}))$. 

\begin{claim}\label{cl:beta(Uu_>u_l-i)_cap_beta(Vv)_is_empty}
There is no edge in $E(\Uu_{>u_{l-i}})$ colored with a color in $\beta(\Vv)$
\end{claim}
\begin{proof}
The proof is similar to the proof of Claim~\ref{cl:beta(Uu_>z)_cap_beta(Vv)_is_empty} of Lemma~\ref{lem:P(i)_implies_Q(i)}.
\end{proof}

We first prove that $P_{weak}(i)$ is true (\textit{i.e.} that $\Xx$ is not a path).

\begin{claim}
The property $P_{weak}(i)$ is true.
\end{claim}
\begin{proof}
Assume that $\beta'(uu_{l-i})\neq 1$ and that $\Xx =X_v(\beta(uu_{l-i}))$ is a path. Then we have that $\beta(uu_{l-i})\not \in \beta(\Vv)$. Without loss of generality, we assume that $\beta(uu_{l-i}) = 4$. Moreover, we have that $m(u_{l-i})\neq 1$. Since $P(j)$ is true for all $j<i$, for all $j<i$, if $\beta(uu_{l-j})\neq 1$, then $X_u(\beta(uu_{l-j})$ is a saturated cycle. We now invert $\Xx$ until we reach a coloring where $X_v(4)$ is a path of length $1$; we denote by $z$ the only vertex of this coloring. Up to a relabeling of the colors, we assume that $v$ is also missing the color $1$ in $\beta'$. The coloring $\beta'$ is $\Vv$-equivalent to the coloring $\beta$, so $\Vv$ is the same minimum cycle in the coloring $\beta$. So by Lemma~\ref{lem:fans_around_Vv} the fan $\Uu' = X_u(m(u')) = (uu'_1,\cdots,uu'_{l'})$ is a cycle entangled with $\Vv$. Moreover, the coloring $\beta'$ is $(\bigcup\limits_{j<i}X_v(\beta(uu_{l-j})))$-equivalent to $\beta$, so by Lemma~\ref{lem:not_changing_last_vertices}, for any $j\leq i$, $u'_{l'-j} = u_{l-j}$, the fan $X_v(\beta'(uu'_{l'-j}))$ is a saturated cycle containing $u'_{l'-j-1}$. So in particular, $uu'_{l'-i} \in E(\Uu')$, and there is a vertex missing the color $4$ in $\Uu'$. Let $z'$ be this vertex. Note that since $X_v(4)$ is a path, for all $j<i$, $X_v(\beta'(uu'_{l-j}))$ does not contain the vertex $z'$.

We now swap the edge $vz$, and denote by $\beta''$ the coloring obtained after the swap. If the coloring $\beta''$ is $\Uu'$-equivalent to $\beta'$, then it means that $v\not\in V(\Uu')$. So in the coloring $\beta''$ the fan $X_u(3) = \Uu'$ contains the vertex $z'$ which is still missing the color $4$. This color is also the missing color of the vertex $v$. Thus, $\Uu'$ is not entangled with $\Vv$, and by Lemma~\ref{lem:fans_around_Vv}, we have a contradiction.

So the vertex $v$ belongs to $V(\Uu')$, and in the coloring $\beta'$, the fan $X_u(1)$ contains the vertex $z'$ which is missing the color $4$. If there is an edge $uu''$ of $E(V_u(1)_{\leq z'})$ colored with a color of $\beta''$, then $X_u(\beta''(uu''))$ is not entangled with $\Vv$, so by Lemma~\ref{lem:fans_around_Vv}, we have a contradiction. Therefore, the subfan $X_u(1)_{\leq z'}$ is a $(\Vv,u)$-independent subfan avoiding $v$. The coloring $\beta''$ is $\Uu'_{\leq v}$-equivalent to the coloring $\beta'$, so in the coloring $\beta''$, the fan $\Uu'' = X_u(3)$ is equal to $(uu'_1,\cdots,uv,uu'_{l'-i)},\cdots, uu'_{l'} = uu')$.

Since $P(j)$ is true for all $j<i$, for all $j<i$ the fan $X_v(\beta''(uu'_{l'-j}))$ is a saturated cycle containing $u'_{l'-j-1}$. In particular, the fan $X_v(\beta''(uu'_{l'-(i-1)})$ is a saturated cycle containing $u'_{l'-i}$. Without loss of generality, we assume that $m(u'_{l'-i}) = 5$. The vertex $u'_{l'-i}$ belongs to the component $K_v(4,5)$, so the vertex $z'$ does not belong to this component. We now swap the component $C_{4,5} = K_{z'}(4,5)$, and denote by $\beta_3$ the coloring obtained after the swap. Note that $\beta_3$ is $\Vv$-equivalent to $\beta''$, so by Observation~\ref{obs:Vv-minimum-stable} the cycle $\Vv$ is the same minimum cycle in the coloring $\beta_3$. The coloring $\beta_3$ is also $\Uu''$-equivalent to $\beta''$, so we still have that $X_u(3) = (uu'_1,\cdots,uv,uu'_{l'-i},\cdots,uu'_{l'} = uu')$.

If the vertex $z$ does not belong to $C_{4,5}$, then we can swap back the edge $vz$. The fan $X_u(3) = X_u(1)$ still contains the vertex $z'$ which is missing the color $5$, and $X_v(\beta''(uu'_{l'-i-1}))$ is still a saturated cycle containing the vertex $u'_{l'-i}$. Since $P(i-1)$ is true, we have a contradiction. So the vertex $z$ belongs to $C_{4,5}$, and in the coloring $\beta_3$ the vertex $z$ is missing the color $5$.

Since the property $P(j)$ is true for all $j<i$, by Lemma~\ref{lem:P(i)_implies_Q(i)}, the property $Q(i-1)$ is true, and so the fan $X_{u'_{l-i}}(2)$ is a cycle containing $z'$, and therefore there is an edge $u'_{l'-i}z$. We denote by $c'$ the color of this edge. We now swap this edge, and denote by $\beta_4$ the coloring obtained after the swap. The coloring $\beta_4$ is $\Vv$-equivalent to $\beta_3$, so the fan $\Vv$ is the same minimum cycle in the coloring $\beta_4$ by Observation~\ref{obs:Vv-minimum-stable}. The coloring $\beta_4$ is also $X_u(3)_{< u'_{l'-i}}$, so the vertex $u'_{l'-i}$ is still in $X_u(3)$. Note that now the vertex $u'_{l'-i}$ and $z'$ are both missing the color $c'$. We now have to distinguish the case.

\begin{case}[$c' = 1$]
~\newline
In this case, the fan $X_u(1)$ contains the vertex $z'$ missing the color $1$, and the fan $X_u(3)$ contains the vertex $u'_{l'-i}$ missing the color $1$. So the fan $X_u(3)$ is a comet containing two vertices missing the color $1$, so by Lemma~\ref{lem:fans_around_Vv}, we have a contradiction.
\end{case}
\begin{case}[$c'\in\beta_3(\Vv)$]
~\newline
In this case, since $u'_{l'-i}\in V(X_u(3))$, the fan $X_u(3)$ is not entangled with $\Vv$, so by Lemma~\ref{lem:fans_around_Vv}, we have a contradiction.
\end{case}

Without loss of generality, we now assume that $c' = 6$.

\begin{case}[$6\in\beta_3(X_u(3)_{<u'_{l'-i}})$]
~\newline
In this case, the fan $X_u(3)$ is now a comet where two vertices are missing the color $6$, thus by Lemma~\ref{lem:fans_around_Vv}, we also have a contradiction.
\end{case}

\begin{case}[$6 \in \beta_3(X_u(3)_{>u'_{l'-i}})$]
~\newline
Let $t<i$ such that $m^{\beta_3}(u'_{l'-t}) = 6$. Since $P(t-1)$ is true, in the coloring $\beta_3$, the fan $X_v(6)$ is a cycle containing $u'_{l'-t}$. Since $P(i-1)$ is true, the fan $X_v(5)$ is a cycle containing $u'_{l'-i}$. We first prove that in the coloring $\beta_3$, we have $X_v(5) = X_v(6)$. In the coloring $\beta_4$, the vertex $u'_{l'-i}$ is missing the color $6$, so the fan $X_u(3)$ is equal to $(uu'_1,\cdots, uv,uu'_{l'-i},uu'_{l'-(t-1)},\cdots, uu'_{l'} = uu')$, and $u'_{l'-i}$ is now the vertex missing the color $6$ in this cycle. Since $P(t-1)$ is true, the fan $X_v(6)$ is now a cycle containing $u'_{l'-i}$. The only vertices whose missing color is different in $\beta_3$ and $\beta_4$ are the vertices $u'_{l'-i}$ and $z'$. In the coloring $\beta_3$, since $z'\not\in X_v(6)$, if $u'_{l'-i}\not\in V(X_v(6))$ the coloring $\beta_4$ is $X_v(6)$-equivalent to the coloring $\beta_3$. This means that in the coloring $\beta_4$, the fan $X_v(6)$ is a cycle containing the vertex $u'_{l'-t}$, and not containing $u'_{l'-i}$, a contradiction. So in the coloring $\beta_3$, the vertex $u'_{l'-i}$ belongs to $X_v(6)$, and thus $X_v(5) = X_v(6)$ as desired. 

So, in the coloring $\beta_3$, the cycle $X_v(5)$ contains the vertex $u'_{l'-t}$ which is missing the color $6$. We now consider the coloring $\beta_4$. The fan $X_v(5)$ still contains the vertex $u'_{l'-t}$ which is still missing the color $6$. The fan $X_v(6)$ is a saturated cycle containing the vertex $u'_{l'-i}$, so the fna $X_v(5)$ is a comet containing $X_v(6)$ as a subfan. The cycle $X_v(6)$ is saturated, so $u'_{l'-i}$ belongs to $K_v(4,6)$, and thus $z'$ does not belong to this component.

We now swap the component $C_{4,6} = K_{z'}(4,6)$, and denote by $\beta_5$ the coloring obtained after the swap. The coloring $\beta_5$ is $\Vv$-equivalent to $\beta_4$, so the fan $\Vv$ is the same minimum cycle in the coloring $\beta_5$. Since the vertex $u'_{l'-i}\not\in C_{4,6}$, and $\beta_4(uu'_{l'-}) = 4$, the vertex $u$ does not belong either to $C_{4,6}$, and therefore the coloring $\beta_5$ is $X_u(3)$-equivalent to the coloring $\beta_4$. The fan $X_u(1)$ still contains the vertex $z'$ which is now missing the color $4$, so the subfan $X_u(1)_{\leq z'}$ is a $(\Vv,u)$-independent subfan avoiding $v$. By Lemma~\ref{lem:(Vv,u)-independent_subfan_avoiding_v}, the fan $X_v(1)$ is a path that does not contain $z'$. In the coloring $\beta_5$, the vertex $z$ is still missing the color $5$, and we still have $\beta_5(vz) = 1$. If the vertex $u'_{l'-t}$ does to belong to $C_{4,6}$, then the coloring $\beta_5$ is $X_v(5)$-equivalent to $\beta_4$, and therefore the fan $X_v(1)$ is a comet containing $X_v(6)$ as a subfan. So the vertex $u'_{l'-t}$ belongs to the component $C_{4,6}$, and it is now missing the color $4$.

In the coloring $\beta_5$, the fan $X_u(5)$ still contains the vertex $u'_{l'-t}$ which is now missing the color $4$. So there is no edge $uu''$ in $E(X_u(5)_{\leq u'_{l'-t}}))$ colroed with a color in $\beta_5(\Vv)$, otherwise, $X_u(\beta_5(uu''))$ is not entangled with $\Vv$, and by Lemma~\ref{lem:fans_around_Vv} we have a contradiction. So the subfan $X_u(5)_{\leq u'_{l'-t}}$ is a $(\Vv,u)$-independent subfan avoiding $v$. By Lemma~\ref{lem:(Vv,u)-independent_subfan_avoiding_v}, the fan $X_v(5)$ is a path that does not contain $u'_{l'-t}$, a contradiction.
\end{case}

\begin{case}[$6 \not \in \beta_3(X_u(3))\cap\beta_3(\Vv)\cup\{1\}$]
~\newline
In the coloring $\beta_4$, the vertex $u'_{l'-i}$ is missing the color $6$, and $uu'_{l'-i}\in E(X_u(3))$. So we have $X_u(6) = X_u(3)$. Since the fan $X_u(5)$ also contains the vertex $u'$, either $X_u(5) = X_u(6) = X_u(3)$, or $X_u(5)$ is a comet which contains $X_u(3)$ as a subfan.
\begin{subcase}[$X_u(5) = X_u(3)$]
~\newline
Let $z''$ be the vertex of $X_u(3)$ missing the color $5$. Note that we may have $z'' = z$. Since $P(i-1)$ is true, the fan $X_v(5)$ is now a cycle containing $z''$. But $u'_{l'-i}$ is the only vertex whose missing color is different in $\beta_3$ and $\beta_4$, so in the coloring $\beta_4$, the fan $X_v(5)$ still contains the vertex $u'_{l'-i}$ which is now missing the color $6$. Therefore, the fan $X_v(6)$ is equal to the fan $X_v(5)$ and is a saturated cycle containing $z''$ and $u'_{l'-i}$. The vertex $z$ is still missing the color $5$, so the fan $X_v(1)$ is now a comet containing $X_v(5)$ as a subfan.

Since the fan $X_v(6)$ is saturated, the vertex $u'_{l'-i}$ belongs to the component $K_v{4,6}$, and thus the vertex $z'$ does not belong to this component. We now swap the component $C_{4,6} = K_{z'}(4,6)$, and denote by $\beta_5$ the coloring obtained after the swap. The coloring $\beta_5$ is $\Vv$-equivalent to $\beta_4$, so the fan $\Vv$ is a minimum cycle in this coloring. Now the fan $X_u(1)$ still contains the vertex $z'$ which is now missing the color $4$, so it is a $(\Vv,u)$-independent subfan avoiding $v$ so by Lemma~\ref{lem:(Vv,u)-independent_subfan_avoiding_v}, the fan $X_v(1)$ is a path. But the coloring $\beta_5$ is also $X_v(1)$-equivalent to the coloring $\beta_4$, so the fan $X_v(1)$ is a comet. This is a contradiction.
\end{subcase}

\begin{subcase}[$X_u(5)$ is a comet containing $X_u(3)$]
~\newline
Let $u'_{l'-t}$ be the first vertex of $X_u(5)$ which is not in $X_u(3)$, and let $z''$ be the vertex of $X_u(3)$ missing the color $c_t = m(u'_{l'-t})$. In the coloring $\beta_3$, since $P(t-1)$ is true, the fan $X_v(c_t)$ is a saturated cycle containing $u'_{l'-t}$. If the coloring $\beta_4$ is $X_v(c_t)$-equivalent to the coloring $\beta_3$, then in the coloring $\beta_4$ the fan $X_v(c_t)$ still contains the vertex $u'_{l'-t}$, and thus does not contain the vertex $z''$. Since $P(t-1)$ is true, we have a contradiction.

So the coloring $\beta_4$ is not $X_v(c_t)$-equivalent to the coloring $\beta_3$. Since $u'_{l'-i}$ and $z'$ are the only vertices whose missing color are different in $\beta_3$ and $\beta_4$, and $z'\not\in V(X_v(c_t))$, we have that $u'_{l'-i}\in V(X_v(c_t))$. In the coloring $\beta_3$ the vertex $u'_{l'-i}$ is also in $X_v(5)$, so in this coloring we have $X_v(5) = X_v(c_t)$. Therefore, the vertex $u'_{l'-t}$ also belongs to $X_v(5)$ in the coloring $\beta_4$.

In the coloring $\beta_4$, the vertex $u'_{l'-i}$ is now missing the color $6$, and since $P(t-1)$ is true, the fan $X_v(c_t)$ is a saturated cycle containing the vertex $z''$. So in this coloring, we have $X_v(c_t) = X_v(6)$. However, in this coloring, the vertex $u'_{l'-t}$ still belongs to $X_v(5)$, it also belongs to $X_u(5)$ and is still missing hte color $c_t$. The cycle $X_v(c_t)$ is saturated, so the vertex $z''$ belongs to the component $K_v(4,c_t)$, and thus the vertex $u'_{l'-t}$ does not belong to this component. We now swap the component $K_{u'_{l'-t}}(4,c_t)$ and denote by $\beta_5$ the coloring obtained after the swap.

The coloring $\beta_5$ is $\Vv$-equivalent to $\beta_4$, so by Observation~\ref{obs:Vv-minimum-stable}, the fan $\Vv$ is a minimum cycle in the coloring $\beta_5$. The coloring $\beta_5$ is also $((X_u(5)\cup X_v(5))\setminus \{u'_{l'-t}\})$-equivalent to the coloring $\beta_4$, so the vertex $u'_{l'-t}$ still belongs to both $X_u(5)$ and $X_v(5)$. So the subfan $X_u(5)_{\leq u'_{l'-t}}$ is a $(\Vv,u)$-independent subfan avoiding $v$, by Lemma~\ref{lem:(Vv,u)-independent_subfan_avoiding_v}, the fan $X_v(5)$ is a path that does not contain $u'_{l'-t}$. Again we have a contradiction.
\end{subcase}
\end{case}
\end{proof}

By the previous claim, we have that $\Xx$ is not a path, we now prove that it is not a comet. Assume that $\Xx = (vx_1,\cdots, vx_t)$ is a comet where $x_s$ and $x_t$ are missing the same color $c_s$. Since $x_s$ and $x_t$ are both missing the color $c_s$ at least one of them is not in $K_v(1,c_s)$. Since $P(j)$ is true for all $j<i$, for all $j<i$, if $\beta(uu_{l-j})\neq 1$, then $X_v(\beta(uu_{l-j}))$ is a cycle, so $\beta(\Xx)\cap (\bigcup\limits_{j\in [0,i-1]}\beta(X_v(\beta(uu_{l-j})))) = \emptyset$.

\begin{case}[$c_s \not\in \beta(\Uu_{u>_{l-1}})$]
~\newline
If $x_s$ is not in $K_v(1,c_s)$, then we swap the component $C_{1,c_s} = K_{x_s}(1,c_s)$ and obtain a coloring $\beta'$ which is $\Vv$-equivalent to $\beta$, so the fan $\Vv$ is the same minimum cycle in the coloring $\beta'$. In the coloring $\beta'$, the fan $X_v(4)$ is now a path. Moreover, $c_s\not\in\beta(\Uu_{\geq u_{l-i}})$, and by Claim~\ref{cl:beta(Uu_>u_l-i)_cap_beta(Vv)_is_empty}, so $1\not\in \beta(\Uu_{\geq u_{l-i}})$. Therefore, the coloring $\beta'$ is $\Uu_{\geq u_{l-i}}$-equivalent to the coloring $\beta$. Let $\Uu' = X^{\beta'}_u(3) = (uu'_1,\cdots,uu'_{l'})$. Since $\beta(\Xx)\cap (\bigcup\limits_{j\in [0,i-1]}\beta(X_v(\beta(uu_{l-j})))) = \emptyset$, the coloring $\beta'$ is $(\bigcup\limits_{j\in [0,i-1]}\beta(X_v(\beta(uu_{l-j})))$-equivalent to $\beta$. So by Lemma~\ref{lem:not_changing_last_vertices}, for all $j\leq i$, we have $u'_{l'-i} = u_{l-i}$.  In particular $u'_{l'-i} = u_{l-i}$. Since $\beta'(uu'_{l'-i}) = 4$, and $P{weak}(i)$ is true, the fan $X_v(4)$ is not a path.

Similarly, if $x_t\not\in K_v(1,c_s)$, we swap the component $C_{1,c_s} = K_{x_t}(1,c_s)$. Note that $\Vv$ is a minimum cycle, so it is saturated by Lemma~\ref{lem:minimum_cycle_saturated}, and thus $x_t\not\in V(\Vv)$. The coloring $\beta'$ is therefore $\Vv$-equivalent to $\Vv$, so the fan $\Vv$ is the same minimum cycle in this coloring. The fan $X_v(4)$ is now a path the coloring $\beta'$. Let $\Uu' = X^{\beta'}_u(3) = (uu'_1,\cdots,uu'_{l'})$. Similarly to the previous case, the coloring $\beta'$ is $\Uu_{>u_{l-i}}$-equivalent to $\beta$ and $(\bigcup\limits_{j\in [0,i-1]}\beta(X_v(\beta(uu_{l-j})))$-equivalent to $\beta$. So by Lemma~\ref{lem:not_changing_last_vertices}, for all $j\leq i$, we have $u'_{l'-j} = u_{l-j}$. In particular, $u'_{l'-i} = u_{l-i}$, and $\beta'(uu'_{l'-i}) = 4$. Since $P_{weak}(i)$ is true, the fan $X_v(4)$ is not a path, a contradiction.
\end{case}

\begin{case}[$c_s \in \beta(\Uu_{>u_{l-1}})$]
~\newline
Let $t'$ be such that $m(u_{l-t'}) = c_s$. Since $P(j)$ is true for all $j<i$, the fan $X_v(c_s)$ is saturated cycle containing $u_{l-t'} = x_t$. So the vertex $x_s$ does not belong to $K_v(1,c_s)$. We now swap the component $C_{1,c_s} = K_{x_s}(1,c_s)$ to obtain a coloring $\beta'$ where $X_v(4)$ is now a path. The coloring $\beta'$ is $\Vv$-equivalent to $\beta$, so by Observation~\ref{obs:Vv-minimum-stable}, the cycle $\Vv$ is the dame minimum cycle in the coloring $\beta'$. If the vertex $u$ does not belong to $C_{1,c_s}$, then the coloring $\beta'$ is also $\Uu$-equivalent to $\beta$, and thus $X_u(3) = \Uu$. Since $\beta'(uu_{l-i}) = 4$, and $P_{weak}(i)$ is true, the fan $X_v(4)$ is not a path. this is a contradiction.

So the vertex $u$ belongs to $C_{1,c_s}$, and in the coloring $\beta'$, the edge $uu_{l-(t'-1)}$ is now colored $1$. Let $\Uu' = X^{\beta'}_u(3)$. The coloring $\beta'$ is $\Uu_{>u_{l-t'}}$-equivalent to $\beta$. The coloring $\beta'$ is also $(\bigcup\limits_{j\in [0,t'-1]}\beta(X_v(\beta(uu_{l-j})))$-equivalent to $\beta$, so by Lemma~\ref{lem:not_changing_last_vertices}, for any $j\leq t'$ we have $u'_{l'-j} = u_{l-j}$. In particular, $u'_{l'-(t'-1)} = u_{l-(t'-1)}$. Now the edge $uu'_{l'-(t'-1)}$ is colored $1$, and the fan $X_v(4)$ is a path. Since $P(j)$ is true for all $j\leq t'$, by Lemma~\ref{lem:P(i-1)_and_P_weak(i)_no_path} there is not path around $v$, a contradiction. 
\end{case}

So the fan $\Xx = (vx_1,\cdots,vx_t)$ is a cycle, we now prove that it is saturated. Otherwise, there exists $x_s$ such that $x_s\not\in K_v(1,m(x_s))$. Note that since $P(j)$ is true for all $j<i$, for all $j<i$, the fan $X_v(\beta(uu_{l-j}))$ is a saturated cycle, so $\beta(X_v(\beta(uu_{l-j})))\cap \beta(\Xx) = \emptyset$, and in particular $x_s\not\in \beta(\Uu_{>u_{l-i}})$.

\begin{case}[$m(x_s) \neq 4$]
~\newline
Without loss of generality, assume that $m(x_s) = 5$. Since $x_s$ does not belong to $K_v(1,5)$, we swap the component $C_{1,5} = K_{x_s}(1,5)$ and obtain a coloring $\beta'$ where $X_v(4)$ is a path. The coloring $\beta'$ is $\Vv$-equivalent to $\beta$, so by Observation~\ref{obs:Vv-minimum-stable}, the cycle $\Vv$ is the same minimum cycle in the coloring $\beta'$. Let $\Uu' = X_u^{\beta'}(3) = (uu'_1,\cdots, uu'_{l'})$. Moreover, $5\not\in\beta(\Uu_{\geq u_{l-i}})$, and by Claim~\ref{cl:beta(Uu_>u_l-i)_cap_beta(Vv)_is_empty}, the color $1$ does not appear either in $\Uu_{>u_{l-i}}$. The coloring $\beta'$ is also $(\bigcup\limits_{j\in [0,i-1]}\beta(X_v(\beta(uu_{l-j})))$-equivalent to $\beta$, so by Lemma~\ref{lem:not_changing_last_vertices}, for any $j\leq i$, we have $u'_{l'-j} = u_{l-j}$. In particular, $u'_{l'-i} = u_{l-i}$. The edge $uu'_{l'-i}$ is still colored $4$ in the coloring $\beta'$ and the property $P_{weak}(i)$ is true, so $X_v(4)$ is not a path, a contradiction.
\end{case} 

\begin{case}[$m(x_s) = 4$]
~\newline
In this case, we swap the component $C_{1,4} = K_{x_s}(1,4)$ and denote by $\beta'$ the coloring obtained after the swap. If the vertex $u$ does not belong to this component, then we are in a coloring similar to the previous case. So the vertex $u$ belongs to $C_{1,4}$, and we have $\beta'(uu_{l-i}) = 1$. In the coloring $\beta'$ is $\Vv$-equivalent to $\beta$, so by Observation~\ref{obs:Vv-minimum-stable}, the cycle $\Vv$ is the same minimum cycle in this coloring. The fan $X_v(4)$ is now a path in the coloring $\beta'$. Let $\Uu' =X^{\beta'}_u(3) =  (uu'_1,\cdots, uu'_{l'})$. The coloring $\beta'$ is $\Uu_{>u_{l-i}}$-equivalent to $\beta$, and is also $(\bigcup\limits_{j\in [0,i-1]}\beta(X_v(\beta(uu_{l-j})))$-equivalent to $\beta$. So by Lemma~\ref{lem:not_changing_last_vertices} for all $j\leq i$, we have $u'_{l'-j} = u_{l-j}$. In particular $u'_{l'-i} = u_{l-i}$. The property $P(j)$ is true for all $j<i$,a nd $P_{weak}(i)$ is also true, so by Lemma~\ref{lem:P(i-1)_and_P_weak(i)_no_path} there is no path around $v$. This is a contradiction.
\end{case}

So the fan $\Xx = (vx_1,\cdots,vx_t)$ is a saturated cycle, and thus $x_t\in K_v(1,4)$. Since $P(i)$ is false, we have $x_t\neq u_{l-i-1}$. So the vertex $u_{l-i-1}$ which is also missing the color $4$ does not belong to $K_v(1,4)$. We now swap the component $C_{1,4} = K_{u_{l-i-1}}(1,4)$ and denote by $\beta'$ the coloring obtained after the swap. Be Lemma~\ref{lem:u_i_not_in_K_v(m(u_i),m(v))_generalized}, the vertex $u$ belongs to $C_{1,4}$, there is an edge $uu''$ colored $1$ in $\Uu_{<u_{l-i}}$, and the subfan $X_u(1)_{\leq u_{l-i}}$ is a $(\Vv,u)$-independent subfan. So in the coloring $\beta'$, the vertex $u_{l-i}$ is now missing the color $1$, the edge $uu''$ is now colored $4$, and the subfan $X_u(4)_{\leq u_{l-i}}$ is a $(\Vv,u)$-independent subfan avoiding $v$. By Lemma~\ref{lem:(Vv,u)-independent_subfan_avoiding_v}, the fan $X_v(4)$ is a path. However, the coloring $\beta'$ is $X_v(4)$-equivalent to the coloring $\beta$, so the fan $X_v(4)$ is a cycle, a contradiction.
\end{proof}

\section{Cycles interactions}\label{sec:cycles_interaction}

In this section we prove Lemma~\ref{lem:cycles_interactions}.

\begin{proof}
We first prove that all the three cycles are tight and saturated.

\begin{claim}\label{cl:V_X_Y_saturated}
The cycles $\Vv$, $\Xx$, and $\Yy$ are saturated and tight.
\end{claim}
\begin{proof}
As the fan $\Vv$ is not invertible, it is saturated by Lemma~\ref{lem:minimum_cycle_saturated}. If $\Xx$ or $\Yy$ are not saturated (without loss of generality, we can assume that $\Xx$ is not saturated), then we swap a component $K_u(c_v,c_u)$ with $u$ in $\Xx$ and $u\not\in K_v(c_v,c_u)$ to transform $\beta$ into a coloring where $\Vv$ is still a cycle of the same size, and where a fan around $v$ is a path, by Lemma~\ref{lem:only_cycles}, $\Vv$ is invertible in this coloring, and so it is in the original coloring.
Similarly, assume that $\Xx$ or $\Yy$ is not tight, without loss of generality, we can assume that $\Xx$ is not tight. Then we can find two consecutive vertices of $\Xx$, $u_{i}$, and $u_{i-1}$ such that the component $K_{u_{i-1}}(m(u_i),m(u_{i-1}))$ does not contain $u_{i}$. If we swap this component, we obtain a coloring where $\Vv$ is still a cycle of the same size, and where a fan around $v$ is a comet, again by Lemma~\ref{lem:only_cycles}, $\Vv$ is invertible in this coloring, and so it is in the coloring $\beta$.
\end{proof}

By Lemma~\ref{lem:fans_around_Vv}, we already have that if $(z,z')\in \Vv^2$, then $X_z(c_{z'})$ is a cycle containing $z'$, so we now assume that $(z,z')$ is not in $\Vv^2$.

\begin{claim}\label{cl:z_not_path}
The fan $\Zz$ is not a path.
\end{claim}
\begin{proof}
As $\Zz$ is a path, we invert it until we reach a coloring where $m(z)\in (\beta(\Vv)\cup\beta(\Xx)\cup\beta(\Yy))\setminus \{m_{\beta}(z)\}$. In this coloring, the fan $\Vv$ is still a cycle of the same size, and, there is a fan around $v$ which is a path or a comet, by Lemma~\ref{lem:only_cycles}, this is a contradiction.
\end{proof}

\begin{claim}
The fan $\Zz$ is entangled with $\Vv$, $\Xx$, and $\Yy$.
\end{claim}
\begin{proof}
Let us assume that there exists $z''\in \Zz\setminus(\Vv\cup\Yy\cup\Xx)$ with $m(z'')\in (\beta(\Vv)\cup\beta(\Yy)\cup\beta(\Xx))\setminus\{c_z\}$. If $m(z'') = c_v$, since the cycles are saturated by Claim~\ref{cl:V_X_Y_saturated}, $K_{z''}(c_z,c_v)$ does not contain any vertex of $(\Vv\cup\Yy\cup\Xx)$, and after swapping it, we obtain a coloring where $\Vv$ is still a cycle of the same size and where $\Zz$ is a path, by Claim~\ref{cl:z_not_path}, this is a contradiction. If $m(z'') \neq c_v$, then, since the cycles are saturated, the component $K_{z''}(c_v,m(z''))$ does not contain any vertex of $(\Vv\cup\Yy\cup\Xx)$, so if we swap it, we obtain a coloring which corresponds to the previous case.
\end{proof}

\begin{claim}
The fan $\Zz$ is not a comet.
\end{claim}
\begin{proof}
Assume that $\Zz$ is a comet, there exist $z_1$ and $z_2$ with $m(z_1) = m(z_2) = c$. By the previous claim, we have that $c\not\in(\beta(\Vv)\cup\beta(\Xx)\cup\beta(\Yy))$, otherwise, $\Zz$ is not entangled with one of these cycles. Hence, the component $K_z(c,m(z))$ either contains $z_1$ or $z_2$, and without loss of generality we can assume that $z_1\not\in K_z(c,m(z))$. If we swap $K_{z_1}(c,m(z))$ we obtain a coloring where no edge of $(\Vv\cup\Xx\cup\Yy)$ has changed and where $\Zz$ is a path, by Claim~\ref{cl:z_not_path} this is a contradiction.
\end{proof}

By the previous claims, $\Zz$ is a cycle, and as it is entangled with the three other cycles, it contains $z'$; this concludes the proof.

\end{proof}

\paragraph*{Acknowledgements.}

We thank Marthe Bonamy, Franti\v{s}ek Kardo\v{s} and \'{E}ric Sopena for helpful discussions, and gratefully thank Penny Haxell for pointing out a mistake in an earlier version

\bibliography{reference}

\end{document}